\documentclass{amsart}

\usepackage{geometry}
\usepackage{color}
\usepackage{xcolor}
\usepackage{enumitem}
\usepackage{hyperref}
\usepackage{cleveref}
\usepackage{doi}

\usepackage{graphicx}
\graphicspath{{./}}
\usepackage{wrapfig}
\usepackage{caption}
\usepackage{float}
\usepackage{tikz}
\usetikzlibrary{matrix}
\usetikzlibrary{arrows,decorations.markings}
\usetikzlibrary{intersections}

\usepackage{amssymb}
\usepackage{amsmath}
\usepackage{wasysym}
\usepackage{textcomp}
\usepackage{stmaryrd}
\usepackage{esint}

\usepackage{etoolbox}
\usepackage{xspace}

\definecolor{myred}{rgb}{0.65, 0, 0}
\definecolor{mygreen}{rgb}{0, 0.65, 0}
\definecolor{myblue}{rgb}{0, 0, 0.65}

\hypersetup{
	frenchlinks = false,
	colorlinks = true,						
	citecolor = mygreen,					
	linkcolor = myblue,						
}

\newcommand*{\transpose}[1]{#1^\mathrm{T}}

\newcommand*{\pscal}[1]{\langle #1\rangle} 
\newcommand*{\Pscal}[1]{\left\langle #1\right\rangle}

\newcommand*{\Cal}[1]{\mathcal{#1}}
\newcommand*{\Rm}[1]{\mathrm{#1}}
\newcommand*{\Norme}[1]{\left\lVert#1\right\rVert}
\newcommand*{\norme}[1]{\lVert#1\rVert}
\newcommand*{\Absolu}[1]{\left\lvert#1\right\rvert}
\newcommand*{\absolu}[1]{\lvert#1\rvert}
\newcommand*{\Set}[1]{\left\lbrace #1 \right\rbrace}
\newcommand{\spectre}[1][]{\Sigma_\text{#1}}
\newcommand*{\floor}[1]{\lfloor #1\rfloor}

\DeclareMathOperator*{\real}{\mathrm{Re}}
\DeclareMathOperator*{\imag}{\mathrm{Im}}
\DeclareMathOperator{\range}{im}

\DeclareMathOperator*{\fred}{Fred}
\DeclareMathOperator*{\codim}{codim}
\DeclareMathOperator*{\Det}{Det}

\tikzstyle{fleche}=[->, >=latex]
\tikzset{
	xa/.store in=\xa, xa/.default=0, xa=0,%
	xb/.store in=\xb, xb/.default=0, xb=0,%
	xc/.store in=\xc, xc/.default=0, xc=0,%
}

\newcommand{\Imagepath}[1]{./#1}

\newcommand{\wrt}{w.r.t.\xspace}
\newcommand*{\kpp}{KPP\xspace}
\newcommand*{\sh}{SH\xspace}
\newcommand*{\gl}{{GL}\xspace}
\newcommand*{\guillemet}[1]{``#1"}
\renewcommand*{\iff}{\emph{iff}~}
\newcommand*{\ie}{\emph{i.e.}\xspace}

\newcommand*{\N}{\mathbb{N}}
\newcommand*{\Z}{\mathbb{Z}}

\newcommand*{\R}{\mathbb{R}}
\newcommand*{\C}{\mathbb{C}}
\newcommand*{\lra}{\longrightarrow}

\renewcommand*{\d}{\mathrm{d}}
\newcommand*{\spacemath}{\ }

\newcommand{\symbkpp}{\text{kpp}}
\newcommand{\symbsh}{\text{sh}}
\newcommand{\symbcouple}{\text{co}}

\newcommand{\symbc}{\text{c}}
\newcommand{\symbs}{\text{s}}
\newcommand{\symbapr}{\text{apr}}
\newcommand{\symbatt}{\text{att}}
\newcommand{\symbstab}{\text{stab}}

\newcommand{\Thetac}{{\Theta_\symbc}}
\newcommand{\Thetas}{{\Theta_\symbs}}

\newcommand{\vPic}{\varPi_\symbc}
\newcommand{\vPis}{\varPi_\symbs}
\newcommand{\vPich}{\varPi_\symbc^h}
\newcommand{\vPish}{\varPi_\symbs^h}
\newcommand{\vPih}{\varPi^h}
\newcommand{\Vc}{V_\symbc}
\newcommand{\Vs}{V_\symbs}
\newcommand{\Wc}{W_\symbc}
\newcommand{\Ws}{W_\symbs}
\newcommand{\Rc}{R_\symbc}
\newcommand{\Rs}{R_\symbs}

\newcommand*{\Au}[1][]{\Cal{A}^{\symbkpp #1}}
\newcommand*{\Av}[1][]{\Cal{A}^{\symbsh #1}}
\newcommand*{\Lu}[1][]{\Cal{L}^{\symbkpp #1}}
\newcommand*{\Lv}[1][]{\Cal{L}^{\symbsh #1}}

\newcommand*{\Glu}[1][\ensuremath{\lambda}]{G_{#1}^\symbkpp}
\newcommand*{\Glv}[1][\ensuremath{\lambda}]{G_{#1}^\symbsh}
\newcommand*{\Gluv}[1][\ensuremath{\lambda}]{G_{#1}^\symbcouple}
\newcommand*{\Gtu}[1][\ensuremath{t}]{\Cal{G}_{#1}^\symbkpp}
\newcommand*{\Gtv}[1][\ensuremath{t}]{\Cal{G}_{#1}^\symbsh}
\newcommand*{\Gtuv}[1][\ensuremath{t}]{\Cal{G}_{#1}^\symbcouple}
\newcommand*{\Wu}{\Cal{W}^\symbkpp}
\newcommand*{\Wv}{\Cal{W}^\symbsh}
\newcommand*{\omegau}{\omega_\symbkpp}
\newcommand*{\omegav}{\omega_\symbsh}

\newcommand*{\ul}{\Rm{u,l}}
\newcommand*{\cc}{\ensuremath{c.c}\xspace}

\newtheorem{theorem}{Theorem}[section]
\newtheorem{corollary}[theorem]{Corollary}
\newtheorem{proposition}[theorem]{Proposition}
\newtheorem{lemma}[theorem]{Lemma}

\theoremstyle{definition}

\theoremstyle{remark}
\newtheorem{remark}[theorem]{Remark}

\title[Nonlinear convective stability of a bi-unstable front]{Nonlinear convective stability of a critical pulled front undergoing a Turing bifurcation at its back: a case study}

\author[L. Gar{\'e}naux]{
	Louis Gar{\'e}naux%
	\address{Institut de Math{\'e}matiques de Toulouse, UMR 5219, Universit{\'e} de Toulouse CNRS, UPS IMT, F-31062 Toulouse Cedex 9, France.}%
	\email{louis.garenaux@math.univ-toulouse.fr}%
}
\date{\today}

\subjclass{35K57, 35C07, 35B35, 35B36, 35Q56, 35K46}

\keywords{Reaction-diffusion, Propagating monostable front, Asymptotic stability, Turing instability, Weighted spaces, Point-wise resolvent estimate, Mode filters, Amplitude equation}

\begin{document}

\maketitle

\begin{abstract}
We investigate a specific reaction-diffusion system that admits a monostable pulled front propagating at constant critical speed. When a small parameter changes sign, the stable equilibrium behind the front destabilizes, due to essential spectrum crossing the imaginary axis, causing a Turing bifurcation. Despite both equilibrium states are unstable, the front continues to exist, and is shown to be asymptotically stable, against suitably-localized perturbations, with algebraic temporal decay rate $t^{-3/2}$. To obtain such decay, we rely on point-wise semigroup estimates, and show that the Turing pattern behind the front remains bounded in time, by use of mode-filters.
\end{abstract}



\tableofcontents

\section{Introduction}

Analysis of evolution problems in PDE often reveals a large variety of nontrivial dynamics. Among parabolic problems, the class of reaction-diffusion equations furnish a wide class of stable -- and thus observable -- solutions: planar waves such as propagating fronts or periodic patterns, rotating spirals, \emph{etc}.

We are interested here in a system coupling a Kolmogorov-Petrovski-Piskunov (\kpp) equation together with a Swift-Hohenberg (\sh) equation, We work with $x\in\R$ and $t>0$:
\begin{equation}
\label{e:original_system}
\left\lbrace
\begin{array}{l}
\partial_t u = d \partial_{xx} u + \alpha u(1-u^2) + \beta v,\\[0.5em]
\partial_t v = - (\partial_{xx} + 1)^2v + v(\mu - \sigma v^2) - \gamma v (1-u),
\end{array}
\right.
\end{equation}
with parameters $d$, $\alpha$, $\mu$, $\sigma$, $\gamma$ positive, and $\beta\in\R$ nonzero.
The scalar \kpp equation is a typical model for front propagation from a stable to an unstable state \cite{KPP_37, Fisher_37}:
\begin{equation}
\label{e:original_KPP}
\partial_t u = d \partial_{xx} u + f(u).
\end{equation}
The diffusion coefficient $d$ is positive and the reaction term $f\in \Cal{C}^2$ is of \kpp type: it admits two equilibrium states $f(0) = 0 = f(1)$ with distinct stability \wrt time $f'(0) > 0 > f'(1)$. Generally, a concavity hypothesis is added: $f''(u)<0$ for $u\in(0,1)$. Although $f(u) = \alpha u(1-u)$ is the canonical choice, we will work with $f(u) = \alpha u(1-u^2)$. All statements and proofs below adapt to the $u(1-u)$ case.

Equation \eqref{e:original_KPP} admits a family of fronts $(q_c)_{c>0}$ -- with $q_c$ propagating at constant speed $c$ -- that connect the stable equilibrium point ($q_c \to 1$ when $x\to-\infty$) to the unstable one ($q_c\to0$ when $x\to+\infty$) \cite{Aronson_Weinberger_78}. 
Supercritical fronts $c>c_*:= 2\sqrt{d\, f'(0)} = 2\sqrt{d \alpha}$ are \emph{convectively stable} with exponential decay in time, against sufficiently localized perturbations. Indeed, when set in a moving frame, conjugating the problem with spatial exponential weights allow to both stabilize the essential spectrum and erase the eigenvalue at $\lambda = 0$, creating a spectral gap at linear level \cite{Sattinger_76}.
Stability of the critical front $q_* := q_{c_*}$ is more involved, since the linear essential spectrum can only be \emph{marginally stabilized}: in the optimal exponentially weighted space, essential spectrum lies at left of the imaginary axis but touches it, due to the presence of absolute spectrum $\Set{-\xi^2 : \xi \in \R} \subset \C$ up to the origin. Last advances in this direction \cite{Gallay_94, Faye_Holzer_18, Avery_Scheel_21} state that $q_*$ is asymptotically stable with decay rate $t^{-3/2}$, and that this rate is optimal. Gallay used renormalization group method; Faye and Holzer relied on point-wise estimates for the resolvent of the full problem; Avery and Scheel reduced the problem to its asymptotic properties through far-field decomposition, and use resolvent estimates on each side.
For subcritical speeds $0 < c < c_*$, there still exists non-monotonic, homoclinic trajectories $q_c$ from $1$ to $0$. However, Sturm Liouville theory ensures the existence of unstable point spectrum, preventing stability.

The scalar SH equation is a simple model for the formation of periodic pattern \cite{Swift_Hohenberg_77, Cross-Hohenberg-93}:
\begin{equation}
\label{e:original_SH}
\partial_t v = - (\partial_{xx} + 1)^2v + \mu v - \sigma v^3.
\end{equation}
When the small parameter $\mu$ becomes positive, \eqref{e:original_SH} undergoes a supercritical Turing bifurcation: a curve of essential spectrum crosses the imaginary axis twice, for nonzero Fourier frequencies $\pm 1$. This causes the constant equilibrium state $u = 0$ to bifurcate into periodic profile of amplitude $\sqrt{\mu}$. It can be obtained through center manifold theory \cite{Eckmann-Wayne-91}.

For all values of $\mu$, the coupled system \eqref{e:original_system} admits $Q = \transpose{(q_*, 0)}$ as a front with propagation speed $c_*$. The linear coupling term $\beta v$ feeds the front dynamic with oscillatory pattern, while the nonlinear coupling term $\gamma v (1-u)$ stabilize at linear level ahead of the front. Thus oscillating profile are expected to appear only behind the front.

Multiple works about bistable front -- \ie a front that connects two stable states: $f(-1) = 0 = f(1)$ and $f'(-1) < 0$, $f'(1) < 0$ -- undergoing a Turing bifurcation have been done. 
When bifurcation occurs behind the front, \cite{Sandstede_Scheel_01} showed that no modulated front appears: the Turing pattern travels at slow speed $\Cal{O}(\sqrt{\mu})$, hence is left behind the front. However, the front survive after bifurcation, although through a non-modulated form. Its nonlinear stability is showed in \cite{Beck_Ghazaryan_Sandstede_09} for a general second-order setting.

When the state ahead of a bistable front bifurcates, \cite{Sandstede_Scheel_01} showed existence and linear stability of modulated fronts -- \ie coherent structures that link stable state behind to Turing pattern ahead -- in a general second-order setting. For a system that ressembles \eqref{e:original_system}, \cite{Gallay_Schneider_Uecker_04} obtain nonlinear stability of such structure. We also mention \cite{Ghazaryan_Sandstede_07} for a quadratic coupling $\beta v^2$ instead of a linear one. The signed term therein allows to obtain \emph{a priori} estimates by applying comparison principle to simplify the system.
When bifurcation occurs both ahead and behind a bistable front, spectral stability of either traveling or standing pulses is showed for a general system \cite{Sandstede_Scheel_PRSEA_00} and \cite{Sandstede_Scheel_Nonlinearity_05}, while  nonlinear stability of a traveling pulse for a precise system is showed in \cite{Gallay_Schneider_Uecker_04}.

For monostable \kpp fronts, previous works rather investigated the Turing bifurcation in a non-local \kpp context. \cite{Berestycki_Nadin_Perthame_Ryzhik_09} first investigate the existence of steady states and propagating fronts, together with the monotonicity of this fronts. A precise construction of modulated fronts is made in \cite{Faye_Holzer_15}, while global properties are obtained in \cite{Hamel_Ryzhick_14}: boundedness of solutions when initial condition is non-positive, together with bounds on propagation speeds.
It seems that the question of the stability of such structure is open, as remarked in \cite{Nadin_Perthame_Tang-11, Hamel_Ryzhick_14}. Our conclusion is that Turing bifurcation behave the same when it occurs behind a monostable or a bistable front: see \cref{f:graphe_solution} for a view on a typical solution. However, the spectrum in the present situation is quite different. On one hand, the essential spectrum is unstable, which require to work in exponentially weighted spaces. Even so, the optimal exponential weight only allows to marginally stabilize the spectrum. On the other hand, the translation eigenvalue $0\in\C$ in the essential spectrum does not contribute to a zero of the so-called Evans function, due to the corresponding eigenfunction $q_*'$ having only weak exponential decay: $q_*'(x) \sim (a x + b)e^{-\frac{c_*}{2} x}$ when $x\to +\infty$. This removes the technical aspect of stability up to a shift in space, also referred to as orbital stability or modulation. 

\begin{figure}[t!]
\includegraphics{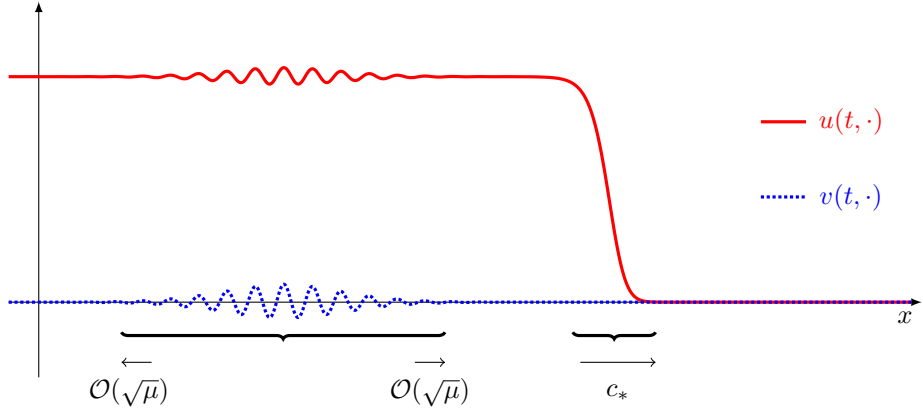}
\caption{Illustration of a typical solution $(u,v)$, for $\mu$ positive and small. The front -- right brace -- has fix shape and moves to the right at constant speed $c_*$. The Turing pattern -- left brace -- continuously expands at slow speed $\Cal{O}(\sqrt{\mu})$, towards left and right. Both states $(0,0)$ and $(1,0)$ -- respectively ahead  and behind the front -- are pointwise unstable, but convectively stable in the critical weighted space.}
\label{f:graphe_solution}
\end{figure}

Let us emphasize that the coupling at linear level $\beta v$ forces us to work with the same weights on the \kpp and \sh components, see further \cref{r:distinct_weights}. 
This typically happens at $+\infty$: conjugation by the critical \kpp weight, although necessary to stabilize the \kpp spectrum, may destabilize the \sh part of the spectrum, if general parameters $\alpha$, $d$, $\gamma$ are set. We counter this effect by imposing a lower bound on $\gamma$, which allows to stabilize the \sh spectrum at $+\infty$.

When no such parameter is available, one has to work with unstable spectrum in any weighted space. This is referred to as \emph{remnant} instability in \cite{Faye_Holzer_Scheel_Siemer_20}, where this exact situation is investigated. Using a precise decomposition of perturbations, they conclude on nonlinear stability of the front $Q_*$ for $\mu < 0$. Although we could have followed this direction in our situation, we prefer to set $\gamma$ large enough so that computations are lightened.

In the supercritical case $c > c_*$, it is possible to obtain spectral gap for the essential spectrum, and to remove the translation eigenvalue, by use of a strong enough weight \cite{Sattinger_76}. In such case, the situation is simpler than bistable case: exponential stability without a phase can be obtained by adapting the current proofs or the ones of \cite{Beck_Ghazaryan_Sandstede_09}. If one rather use the optimal supercritical weight, the situation has to be discussed. 

In the monostable critical case, instability ahead of the front is still an open problem, and would be an interesting direction to follow. For example, one can replace the coupling term $\gamma v(1-u)$ by $\gamma uv$. We point out that a modulated front for such system is necessarily non-positive, which prevent the use of a \kpp reaction term $u(1-u)$. To ensure existence of solutions globally in time, it is possible to rather study a monostable front emanating from a bistable reaction term $u(1-u^2)$.

\section{Main result}

\subsection{Notations}
We investigate now the stability of the traveling front $Q_* = \transpose{(q_*, 0)}$. Insert the following ansatz in \eqref{e:original_system}:
\begin{equation*}
\begin{pmatrix} u(t, x) \\ v(t,x) \end{pmatrix} 
= Q(\tilde{x}) + P(t, \tilde{x}),
\end{equation*}
with $\tilde{x} = x - c_*t$. Then the perturbation $P = \transpose{(p_1, p_2)}$ satisfies 
\begin{equation*}
\label{e:system_P}
\partial_t P = \Cal{A} P + N(P) := \begin{pmatrix}
\Au & \beta \\
0 & \Av
\end{pmatrix} 
P + 
\begin{pmatrix}
N_1(P) \\ N_2(P)
\end{pmatrix},
\end{equation*}
where the linear operator $\Cal{A}: L^2(\R)^2 \to L^2(\R)^2$ is closed, has dense domain $H^2(\R) \times H^4(\R)$ and is defined by:
\begin{equation}
\label{e:def_linear}
\Cal{A}^\symbkpp = d\partial_{xx} + c_*\partial_x + \alpha(1 - 3q_*(\tilde{x})), 
\hspace{2em}
\Cal{A}^\symbsh = -(1+\partial_{xx})^2 +c_* \partial_x + \mu - \gamma (1-q_*(\tilde{x})),
\end{equation}
while remaining nonlinear terms express as:
\begin{equation}
\label{e:def_nonlinear}
N_1(P) = -\alpha(3q_*(\tilde{x}) {p_1}^2 + {p_1}^3), 
\hspace{4em}
N_2(P) = \gamma p_1 p_2 -\sigma {p_2}^3.
\end{equation}
The essential spectrum of $\Cal{A}$ is unstable, see \cref{f:spectre}, we do not hope for an asymptotic stability result against general perturbation $P \in H^2(\R) \times H^4(\R)$. Instead, we restrict to weighted perturbations $P(t,\tilde{x}) = \omega(\tilde{x}) U(t, \tilde{x})$, with $U\in H^2(\R)\times H^4(\R)$. As we shall see in further \cref{r:distinct_weights}, the coupling term $\beta v$ imposes us to use scalar weight. Namely $\omega_* := \omegau \omegav$, with smooth positive functions
\begin{equation}
\label{e:def-weights}
\omegau(x) := 
\begin{cases}
1 & \text{if } x \leq -1,\\
e^{-\frac{c_*}{2d} x} & \text{if } x \geq 1,
\end{cases}
\hspace{4em}
\omegav(x) := 
\begin{cases}
e^{\theta x} & \text{if } x \leq -1,\\
1 & \text{if } x \geq 1,
\end{cases}
\end{equation}
where the exponent $\theta<0$ is small, see \cref{p:essential_spectrum}.
The weight $\omegau$ acts only at $+\infty$, and allows to marginally stabilize the \kpp spectrum, while $\omegav$ acts at $-\infty$ and stabilizes the \sh spectrum.
Hence, we insert the new ansatz 
\begin{equation}
\label{e:ansatz_perturbation}
\begin{pmatrix} u(t, x) \\ v(t,x) \end{pmatrix} 
= Q(\tilde{x}) + \omega_*(\tilde{x}) U(t, \tilde{x})
\end{equation}
in \eqref{e:original_system} to obtain that $U = \transpose{(u_1, u_2)}$ satisfies:
\begin{equation}
\label{e:system_U}
\partial_t U = \Cal{L} U + \Cal{N}(U) := 
\begin{pmatrix}
\Lu & \beta \\
0 & \Lv
\end{pmatrix} 
U + 
\begin{pmatrix}
\Cal{N}_1(U) \\ \Cal{N}_2(U)
\end{pmatrix},
\end{equation}
where all linear terms are closed, densely defined operators. They are obtained as conjugation of the unweighted linear operators, 
\begin{equation*}
\Cal{L}^i = \omega_*^{-1} \Cal{A}^i \omega_*
\end{equation*} 
when $i\in \lbrace \symbkpp, \symbsh \rbrace$ and nonlinear terms express according to $\Cal{N}_i(U) = \omega_*^{-1} N_i(\omega_* U)$.

\begin{figure}[t!]
\centering
\includegraphics[scale = 0.65]{\Imagepath{spectre_Au-Av}}
\includegraphics[scale = 0.65]{\Imagepath{spectre_Lu-Lv}}
\includegraphics[scale = 0.65]{\Imagepath{spectre_T_minus}}
\caption{Essential spectrum of linear operators. Fredholm borders corresponding to \kpp (respectively \sh) component are plain (respectively dashed). Curves in red with single arrow (respectively blue with double arrow) correspond to $+\infty$ (respectively $-\infty$) borders. 
Left: $\spectre[ess](\Cal{A})$ with two unstable curves, $\spectre(\Av[, +])$ is out of picture to the left. Center: $\spectre[ess](\Cal{L})$ with only marginally stable spectrum. Right: $\spectre[ess](\Cal{T}^- + c_*\partial_x)$ with one unstable curve, which corresponds to Turing instability. The essential spectrum of $\Cal{T}^-$ is real and obtained from the one represented by application of $z\mapsto \real{z}$. Notice that the spectrum of $\Cal{T}$ can not be marginally stabilized, due to the lack of a first order derivative.}
\label{f:spectre}
\end{figure}

Since the translation eigenmode $q_*'$ in weighted spaces is unbounded at $x\to +\infty$, we expect $U$ to decay in time with rate $t^{-3/2}$ if the following assumption of a marginally stable essential spectrum holds, see \cite{Avery_Scheel_21, Faye_Holzer_18}:
\begin{equation*}
\tag{\ensuremath{H_1}}
\label{e:hyp_marginal_stability}
\spectre[ess](\Cal{L}) \subset \Set{\real{\lambda} \leq -\eta} \cup \Set{-\xi^2 : \xi \in\R}.
\end{equation*} 
However at nonlinear level, it appears we also need to study the dynamic near the steady state $(u,v) = (1, 0)$ undergoing a Turing bifurcation. This is classic for such problem and was already done in \cite{Beck_Ghazaryan_Sandstede_09}. More details are given in \cref{ss:sketch_of_proof}. We do so with the second ansatz
\begin{equation}
\label{e:ansatz_V}
\begin{pmatrix} u(t,x) \\ v(t,x) \end{pmatrix}
= Q(\tilde{x}) + \rho_*(\tilde{x})\omegau(\tilde{x}) \, V(t,x).
\end{equation}
Here again $\tilde{x} = x- c_*t$, and the polynomial weight $\rho_*$ is a smooth positive function defined by
\begin{equation}
\label{e:def_algebraic_weight}
\rho_*(x) := 
\begin{cases}
1 & \text{if } x\leq -1,\\
\pscal{x} := \sqrt{1 + x^2} & \text{if } x\geq 1.
\end{cases}
\end{equation}
We note $\varpi := \rho_* \omegau$, and insert \eqref{e:ansatz_V} into \eqref{e:original_system} to obtain that
\begin{equation*}
\varpi \partial_t V - c_* (\partial_x \varpi) V = \Cal{A} (\varpi V) - c_* \partial_x (\varpi V)+ N(\varpi V),
\end{equation*}
which rewrites as
\begin{align}
\nonumber 
\partial_t V & = \Cal{T} V + \Cal{Q}(V), \\
\label{e:system_V}
& = \Cal{T}^- V + \Cal{Q}^-(V) + \Cal{S}(\tilde{x}, V),
\end{align}
where linear operators express as:
\begin{equation*}
\Cal{T} := \varpi(\tilde{x})^{-1}(\Cal{A} - c_* \partial_x) \varpi(\tilde{x}) +  c_* \frac{\partial_x \varpi}{\varpi} (\tilde{x})
\hspace{4em}
\Cal{T}^- = \lim_{\tilde{x}\to -\infty} \Cal{T}(\tilde{x}),
\end{equation*}
while nonlinear terms are defined by $\Cal{Q}(V) = \varpi(\tilde{x})^{-1} N(\varpi(\tilde{x}) V)$, and $\Cal{Q}^-(V) = \lim_{\tilde{x} \to -\infty} \Cal{Q}(V)$. We are left with an error term 
\begin{equation*}
\Cal{S}(V) := \Cal{T}V - \Cal{T}^- V + \Cal{Q}(V) - \Cal{Q}^-(V).
\end{equation*}
Since $\varpi(\tilde{x}) = 1$ for $\tilde{x}\leq -1$, we decompose $\Cal{T}^-$ as:
\begin{equation}
\label{e:def_T_minus}
\Cal{T}^- = \begin{pmatrix}
d\partial_{xx} - 2\alpha & \beta \\
0 & -(1+\partial_{xx})^2 
\end{pmatrix}
+ \begin{pmatrix}
0 & 0 \\
0 & \mu
\end{pmatrix}
=: \Cal{T}_0 + \Cal{T}_\mu.
\end{equation}
and $\Cal{Q}^-$ as
\begin{equation*}
\Cal{Q}^-(V) = \begin{pmatrix}
- 3\alpha {v_1}^2 \\
\gamma {v_1}{v_2}
\end{pmatrix}
+ 
\begin{pmatrix}
-\alpha {v_1}^3 \\
-\sigma {v_2}^3
\end{pmatrix} =: \Cal{Q}_2 + \Cal{Q}_3.
\end{equation*}
Thus we get the following expression for the source term:
\begin{align}
\label{e:rest-term}
\Cal{S}(\cdot, V) = {} & \left[\varpi^{-1} (\Cal{A} - c_* \partial_x) \varpi - (\Cal{A} -c_*\partial_x)   + c_* \frac{\partial_x \varpi}{\varpi} + 
\begin{pmatrix}
3\alpha (1 - q_*^2) & 0\\
0 & -\gamma (1 - q_*)
\end{pmatrix}
\right] V \\
& {} + (\varpi - 1)
\begin{pmatrix}
- 3\alpha q_* {v_1}^2 \\
\nonumber
\gamma {v_1}{v_2}
\end{pmatrix}
+ (\varpi^2 - 1)
\begin{pmatrix}
-\alpha {v_1}^3 \\ -\sigma {v_2}^3
\end{pmatrix}.
\end{align}
To control perturbation $V$, we will compute an amplitude equation, which reveals to be a Ginzburg-Landau equation with real coefficients
\begin{equation}
\tag{\gl}
\label{e:GL}
\partial_T A = 4\partial_{XX} A + A + b A\absolu{A}^2.
\end{equation}
An important property for our result to hold is that the Turing bifurcation is \emph{supercritical}:
\begin{equation*}
\label{e:hyp_Turing_surcritique}
\tag{\ensuremath{H_2}}
\text{In \eqref{e:GL}, the coefficient $b$ is negative.}
\end{equation*}

We emphasize that the perturbation $V$ is not traveling  in the laboratory frame $x$, but that we measure it using a weight that follows the front. Despite $V$ is measured in the steady frame $x$, 
it would be possible to replace ansatz \eqref{e:ansatz_V} by $\transpose{(u,v)} = Q(\tilde{x}) + \rho_*(\tilde{x}) \omegau(\tilde{x}) V(t, \tilde{x})$. The actual choice lighten the computation to  obtain the \gl equation.
Conjugation by $\varpi := \rho_* \omegau$ is needed so that \cref{p:decay_rest_term} holds: this weight is natural for \kpp equation, in the sense that its spatial decay at $+\infty$ copy the one of the translation eigenfunction $q_*'$ \cite{Faye_Holzer_18}.

The perturbations $U$ and $V$ are linked through
\begin{equation}
\label{e:lien_U-V}
V(t, x) = \frac{\omegav(\tilde{x})}{\rho_*(\tilde{x})} \ U(t, \tilde{x}).
\end{equation}
To measure space localization and regularity, we use Sobolev spaces, and introduce the following norm:
\begin{equation*}
\norme{U}_X := \norme{U}_{W^{2, \infty}(\R) \times W^{4, \infty}(\R)} + \norme{\rho_*^3 U}_{L^\infty(\R)}.
\end{equation*}
We can now state our main result.
\begin{theorem}
\label{t:main_result}
There exists an open, nonempty set of parameters $\Omega$ such that for $(\alpha, \beta, d, \sigma, \gamma)$ in $\Omega$, both hypothesis \eqref{e:hyp_marginal_stability} and \eqref{e:hyp_Turing_surcritique} holds. For such a choice of parameters, the following holds. There exists positive constants $C$, $\mu_0$, and $\delta$ such that for all $0 < \mu < \mu_0$, if $U_0$, $V_0$ satisfies
\begin{equation*}
K_0 := \frac{1}{\sqrt{\mu}}\left(\norme{V_0}_{W^{1,\infty}(\R)} + \norme{U_0}_X\right) \leq \delta,
\end{equation*}
then the solution of \eqref{e:original_system} with initial condition $Q + \omegau \omegav U_0 = Q + \rho_*\omegau V_0$ exists for all time, and writes as $Q(\tilde{x}) + \omegau(\tilde{x}) \omegav (\tilde{x}) U(t, \tilde{x}) = Q(\tilde{x}) + \rho_*(\tilde{x})\omegau(\tilde{x}) V(t, x)$. For all $t>0$, it satisfies:
\begin{equation*}
\Norme{\frac{U(t, \cdot)}{\rho_*}}_{L^\infty(\R)} \leq C\frac{\norme{U_0 \rho_*^3}_{L^\infty(\R)}}{(1+t)^{3/2}},
\hspace{4em}
\Norme{V(t, \cdot)}_{L^\infty(\R)} \leq C \sqrt{\mu}\, (1 + K_0 e^{-t/\mu}).
\end{equation*}
\end{theorem}

It happens we can refine the decay of perturbation $U$.

\begin{corollary}
Let $1 < p \leq +\infty$ and $(\alpha, \beta, d, \sigma, \gamma) \in \Omega$, with $\Omega$ as in \cref{t:main_result}. There exists positive constants $C$, $\mu_0$, $\delta$ and $\eta$ such that if 
\begin{equation*}
K_0 \leq \delta, 
\hspace{4em}
\norme{U_0 \rho_*^3}_{L^p(\R)} \leq \delta, 
\end{equation*}
then the solution $U = \transpose{(u_1, u_2)}$ of \eqref{e:original_system} with initial condition $U_0$ satisfies:
\begin{equation*}
\Norme{\frac{u_1(t, \cdot)}{\rho_*}}_{L^p(\R)} \leq C\frac{\norme{U_0 \rho_*^3}_{L^p(\R)}}{(1+t)^{\frac{3}{2} - \frac{1}{2p}}}, 
\hspace{4em}
\norme{u_2(t, \cdot)}_{L^p(\R)} \leq C e^{-\eta t}  \norme{U_0}_{L^p(\R)}.
\end{equation*}
\end{corollary}

\subsection{Sketch of the proof}
\label{ss:sketch_of_proof}
We informally present the main ideas of our proof.

First, the dynamic for fully weighted perturbation $U$ writes as:
\begin{equation*}
\partial_t U = \Cal{L} U + \Cal{N}(U).
\end{equation*}
Where the linear operator $\Cal{L}$ has marginally stable spectrum, similarly to the \kpp equation, and the nonlinear part $\Cal{N}$ is unbounded \wrt to $x\in\R$. At linear level, we expect an algebraic decay: $t^{-3/2}$. We follow the approach from \cite{Faye_Holzer_18}: we first look at the solution to the linear Cauchy problem
\begin{equation}
\label{e:linear_cauchy_problem}
\partial_t p = \Cal{L} p, 
\hspace{4em}
p(0, \cdot) = p_0,
\end{equation}
and assume that it is expressed through a kernel: $p(t, x) = \int_\R \Cal{G}_t(x,y) p_0(y) \d y$. The matrix valued function $\Cal{G}_\cdot(\cdot, y)$ has to satisfy the linear problem \eqref{e:linear_cauchy_problem}, with a Dirac delta initial condition: $\Cal{G}_0 (x, y) = \delta_y(x) \begin{pmatrix}1 & 0\\ 0 & 1\end{pmatrix}$.
Remark that $\Cal{G}$ is an upper-triangular matrix, due to the triangular structure of $\Cal{L}$ and $\Cal{G}_0$.
We note it
\begin{equation*}
\Cal{G} = \begin{pmatrix}
\Gtu[] & \Gtuv[] \\
0 & \Gtv[]
\end{pmatrix},
\end{equation*}
where $\Cal{G}^i$ solves $\partial_t \Cal{G}_t = \Cal{L}^i \Cal{G}_t$ when $i\in\Set{\symbkpp, \symbsh}$, and $\Gtuv$ accounts for the coupling terms.
We then apply Laplace transform to obtain the spectral Green kernel $G_\lambda(x,y) = \int_0^{+\infty} e^{-\lambda t} \Cal{G}_t(x,y) \d t$, that satisfies the \emph{fundamental eigenproblem}
\begin{equation*}
(\lambda - \Cal{L}) G_\lambda = \delta_y,
\end{equation*}
for suitable $\lambda\in\C$. As above, it is an upper-triangular matrix that writes
\begin{equation*}
G = \begin{pmatrix}
\Glu[] & \Gluv[] \\
0 & \Glv[]
\end{pmatrix},
\end{equation*}
where each $G^i$ is the Laplace transform of $\Cal{G}^i$ for $i\in\Set{\symbkpp, \symbsh, \symbcouple}$. Hence, they satisfy
\begin{equation*}
(\lambda - \Lu) \Glu(\cdot, y) = \delta_y, \hspace{4em}
(\lambda - \Lv) \Glv(\cdot, y) = \delta_y,
\end{equation*}
\begin{equation*}
(\lambda - \Lu) \Gluv(\cdot, y) = \beta \Glv(\cdot, y).
\end{equation*}
The homogeneous eigenproblems $(\lambda - \Cal{L}^i) \phi = 0$ with $i\in\Set{\symbkpp, \symbsh}$ are ODEs, their solutions admit exponential behaviors at $\pm\infty$:
\begin{equation*}
\phi(x) = e^{\nu(\lambda) x} (1 + e^{-r\absolu{x}} \kappa(x)),
\end{equation*}
where $\nu\in\C$ is a solution of the dispersion relations $\lambda - \Cal{L}^{i, \pm}(\nu) = 0$; the real $r>0$ does not depend on $\nu$ nor $\lambda$; $\kappa$ is a bounded function on the half-line $\R_\pm$.
Such solutions can be concatenated to construct $G_\lambda^i(\cdot, y) \in L^2(\R)$. 
Then, the coupled spectral green function expresses as the $L^2(\R)$-inner product
\begin{equation*}
\Gluv(x,y) = \pscal{\Glu(x, \cdot) , \beta \Glv(\cdot, y)}.
\end{equation*}
This leads to spatial localization of $G_\lambda$, which is converted into temporal decay for $\Cal{G}_t$ through the inverse Laplace transform:
\begin{equation*}
\Cal{G}_t(x, y) = \frac{1}{2i\pi} \int_\Lambda e^{\lambda t} G_\lambda(x,y) \d \lambda,
\end{equation*}
where $\Lambda$ is a contour in the resolvent set $\C\backslash\spectre(\Cal{L})$, that can be chosen as a continuous deformation of a sectorial contour, see \cite[section 3]{Davies_2002}. 
To control the full non-linear dynamic, we then use Duhamel's formula:
\begin{equation*}
U(t,x) = \int_\R \Cal{G}_t(x,y) U_0(y) \d y + \int_0^t \int_\R \Cal{G}_{t-\tau} (x,y) \Cal{N}(U(\tau, y)) \d y\d \tau.
\end{equation*}
The fact that $\Cal{N}(U)$ is unbounded\footnote{Due to $\absolu{\omegav(x)} \to \infty$ when $x\to -\infty$} is a major issue. In \cite{Beck_Ghazaryan_Sandstede_09}, it is absorbed by transforming nonlinear terms into linear ones: writing $\omegav^{-1}\left(\omegav u\right)^p = (\omegav u)^{p-1} u$ and showing that $\norme{\omegav u}_{L^\infty(\R)}$ is bounded \wrt time. We follow the same line.
The material presented above is detailed in \cref{s:decay_U}, and leads to the following proposition, the proof of which can be found in \cref{ss:non_linear_U}.
\begin{proposition}[$V$ bounded implies decay of $U$]
\label{p:V_implies_U}
Assume hypothesis \eqref{e:hyp_marginal_stability} holds.
There exists positive constants $\delta$, $\mu_0$, $C_\symbstab$ and $\eta$ such that for all $0 < \mu < \mu_0$, the following holds. Fix $C$ and $t_V$ positive constants, an initial condition $U_0$ such that $U_0 \rho_*^3 \in L^p(\R)$, and assume that for all $0 \leq t \leq t_V$,
\begin{equation*}
\norme{V(t)}_{L^\infty(\R)} \leq C \leq \delta.
\end{equation*} 
Then the solution $U$ for the Cauchy problem \eqref{e:system_U} with initial condition $U_0$ is defined for all $0 \leq t\leq T$, and satisfies
\begin{equation*}
\Norme{\frac{u_1(t)}{\rho_*}}_{L^p(\R)}\leq C_\symbstab\frac{\norme{U_0 \rho_*^3}_{L^p(\R)}}{(1+t)^{\frac{3}{2} - \frac{1}{2p}}}, 
\hspace{4em}
\norme{u_2(t)}_{L^p(\R)} \leq C_\symbstab e^{- \eta t}  \norme{U_0}_{L^p(\R)}.
\end{equation*}
Furthermore, $C_\symbstab$ depends neither on $C$ or $t_V$.
\end{proposition}

Second, we turn to $\frac{\omegav}{\rho_*} U = V$, and show it is bounded in time.\footnote{The presence of an extra $\frac{1}{\rho_*}$ is necessary so that source term $\Cal{S}(\tilde{x}, V)$ decay in time, see \cref{p:decay_rest_term}.}
Remind that 
\begin{equation*}
\partial_t V = \Cal{T} V + \Cal{Q}(V) = \Cal{T}^- V + \Cal{Q}^-(V) + \Cal{S}(\tilde{x}, V).
\end{equation*}
We show that $\Cal{S}$ is sufficiently localized in space so that we can extract a $\omegav$ from it. This allows to control $\Cal{S}(V)$ by $U/\rho_*$ -- see \eqref{e:lien_U-V} -- which ensures decay in time as done in \cite[below Proposition 3.3]{Ghazaryan_Sandstede_07} and \cite[Lemma 4.1]{Beck_Ghazaryan_Sandstede_09}. Hence after a suitable time, $V$ is driven by the dynamic at $-\infty$, which allows to get rid of the marginally stable curve of essential spectrum at $+\infty$. Since $\Cal{T}^-$ undergoes a Turing bifurcation, we follow the steps of \cite{Schneider_94_juil, Beck_Ghazaryan_Sandstede_09}, that mostly relies on the use of \emph{mode-filters}, see \cite{Schneider_94_dec}. 

We first show that periodic patterns are naturally selected: after a time $\frac{T}{\mu}$, perturbation $V$ has at first order an oscillating profile. This is commonly referred to as the \emph{approximation property}. For such profiles, using multi-scale analysis, dynamic of the whole system reduces to an amplitude equation, which in our case appears to be the Ginzburg-Landau (\gl) equation: setting $\varepsilon = \sqrt{\mu}$, if $V = \psi(A) + \Cal{O}(\varepsilon^2)$, with $\psi(\varepsilon, A)(t,x) :=  \varepsilon A(\varepsilon^2 t, \varepsilon x)e^{ix}\varrho_\symbc + \cc$,\footnote{Here and in the following, we note $\cc$ for the complex conjugate: $z + \cc := z + \bar{z} = 2\real(z)$.} with a suitable $\varrho_\symbc \in\R^2$, then $A(T,X) \in \C$ satisfies
\begin{equation*}
\partial_T A = 4\partial_{XX} A + A + b A\absolu{A}^2.
\end{equation*}
We refer to \cite{Mielke_02, Mielke_Schneider_95} for the derivation of amplitude equation.
For suitable values of the parameter $\gamma>0$, coefficient $b\in\R$ is negative, has shown in \cref{ss:derive-GL}. Hence the Turing bifurcation is supercritical, and \eqref{e:GL} is known to have a bounded global attractor -- see \cite[Theorem 3.4]{Mielke_Schneider_95} -- which ensures that $A$ is bounded \wrt time. In case where the bifurcation is subcritical, a quintic term is often add to recover a precise behavior, we do not explore this line. To conclude that $V$ stays small for all time, we use the \emph{approximation property}: if $V$ is close to $\psi(A)$ at time $t_0>0$, then the solution of the whole system emanating from $V_0$ is defined upon time $t_0 + T/\mu$, and remains close to $\psi(A)$. This last step can be applied as many time as needed, without deterioration of constants.
All these arguments are made precise in \cref{s:V_bounded}. They lead to the next proposition, which is proved on \cpageref{pf:V_bounded}.
\begin{proposition}[Decay of $U$ implies $V$ bounded]
\label{p:U_implies_V}
Assume hypothesis \eqref{e:hyp_Turing_surcritique} holds.
There exists positive constants $\delta$, $\mu_0$ such that for all $0 < \mu < \mu_0$, the following holds. 
Assume that 
\begin{equation*}
K_0 := \frac{1}{\sqrt{\mu}}\left(\norme{V_0}_{W^{1,\infty}(\R)} + \norme{U_0}_X\right) \leq \delta.
\end{equation*}
Then $V(t)$ is defined for all times $t\geq 0$. Assume further that $C_1$ and $t_U$ are positive constants, such that for all $0 \leq t \leq t_U$:
\begin{equation*}
\Norme{\frac{U(t)}{\rho_*}}_{L^\infty(\R)} \leq C_1 \frac{\norme{U_0 \rho_*^3}_{L^\infty(\R)}}{(1+t)^{3/2}}.
\end{equation*}
Then there exists $C_2$ -- that may depend on $C_1$, but neither on $t_U$, $\delta$ or $\mu_0$ -- such that for all $0 \leq t\leq t_U$, 
\begin{equation*}
\norme{V(t)}_{L^\infty(\R)} \leq C_2 \sqrt{\mu}\, (1 + K_0 e^{-t/\mu}).
\end{equation*}
\end{proposition}
Finally, we combine \cref{p:V_implies_U,p:U_implies_V} to prove \cref{t:main_result}. It may seems unclear how to use jointly those two propositions. The important point is that $C_\symbstab$ is independent of the bound on $V$. It reads as follows.
\begin{proof}
\textit{\Cref{t:main_result}.\hspace{2ex}}
First apply \cref{l:hyp_are_fulfilled} to obtain the existence of $\Omega$ that allows to fulfill both hypothesis \eqref{e:hyp_marginal_stability} and \eqref{e:hyp_Turing_surcritique}. Take $\mu_0$ small enough, and fix $\rho_* V_0 = \omegav U_0$ such that:
\begin{equation*}
U_0 \rho_*^3 \in L^\infty(\R), \hspace{4em}
K_0 \leq \delta.
\end{equation*}
Since $\Cal{T}$ is sectorial, and $V \mapsto \Cal{Q}(V)$ is locally Lipschitz, the solution $t \in I \mapsto V(t) \in L^\infty(\R)$ to equation $\eqref{e:system_V}$ is uniquely defined on a open, nonempty, maximal set $I$ -- see \cite{Henry_semilinear_parabolic_equation}. Hence there exists $t_0$ and $C$ positive constants such that 
\begin{equation*}
\norme{V(t)}_{L^\infty(\R)} \leq C
\end{equation*}
for all $0\leq t \leq t_0$. Applying \cref{p:V_implies_U}, this ensures that for all $0 \leq t \leq t_0$, the perturbation $U/\rho_*$ is uniquely defined in $L^\infty(\R)$, and satisfies
\begin{equation}
\label{e:proof_th_decay_U}
N(t) := \frac{1}{(1+t)^{3/2}} \Norme{\frac{U(t)}{\rho_*}}_{L^\infty(\R)} \leq C_\symbstab \norme{U_0 \rho_*^3}_{L^\infty(\R)}.
\end{equation}
Consider $t_1$ the first time where this inequality may fail:
\begin{equation*}
t_1 := \inf I, 
\hspace{4em}
I := \Set{t>0 : N(t) > C_\symbstab \norme{U_0\rho_*^3}_{L^\infty(\R)}}.
\end{equation*}
Assuming by contradiction that $I$ is nonempty, remark that $0 < t_0 \leq t_1 < +\infty$. From \cref{p:U_implies_V}, $V$ is defined for all times $0 \leq t \leq t_1$, with bound 
\begin{equation*}
\norme{V(t)}_{L^\infty(\R)} \leq C_2 \sqrt{\mu}\, (1 + K_0),
\hspace{4em}
0 \leq t \leq t_1
\end{equation*}
where $C_2$ depends on $C_\symbstab$. Reasoning as above, there exists $t_2 > t_1$ such that $V$ is defined up to time $t_2$, with bound
\begin{equation*}
\norme{V(t)}_{L^\infty(\R)} \leq 2 C_2 \sqrt{\mu}\, (1 + K_0),
\hspace{4em}
0 \leq t \leq t_2.
\end{equation*}
Assuming that $\mu_0(1 + K_0)$ is small enough, we can assume $2 C_2 \sqrt{\mu}\, (1 + K_0) \leq \delta$ so that \cref{p:V_implies_U} applies again, and we recover \eqref{e:proof_th_decay_U} for times $0 \leq t \leq t_2$. This is a contradiction with the definition of $t_1$, hence we conclude that $I$ is empty, and that \eqref{e:proof_th_decay_U} holds for all times. Applying \cref{p:U_implies_V}, we recover the claimed bound on $V$. Assuming now that $\norme{\rho_*^3 U_0}_{L^p} \leq \delta$, we can finally apply \cref{p:V_implies_U} to recover $L^p(\R)$ estimates on $U$ for all times. 
\end{proof}

\begin{remark}
To understand how $V$ stays bounded in time, we separate critical from stable frequencies, using mode-filters in Fourier space. This imposes us to work with uniformly localized spaces, which are difficult to combine with Green's kernel approach. 
An alternative strategy would be to describe Turing patterns as solution of the eigenvalue ODE problem -- hoping they will reflect on the Green kernel -- and separate critical from stable mode in Laplace space. It would allow to work both in Sobolev spaces, and with the full dynamic $\Cal{T}$ rather than the asymptotic one $\Cal{T}^-$. This last point allow to remove the source term $\Cal{S}$, and as so to break the $U$ - $V$ interaction in our proof. 
\end{remark}

\section{Decay of perturbations in fully weighted space}
\label{s:decay_U}

\subsection{Essential and point spectrum}
\label{ss:spectrum_localization}
We say that a scalar differential operator $\Cal{T}(x) = \sum_{j = 0}^n  a_j(x) \partial_x^j$ is \emph{exponentially asymptotic} if it converges with uniform exponential rate at $\pm\infty$: there exists $r>0$ such that for all $0\leq j\leq n$, 
\begin{equation*}
\lim_{x\to \pm\infty} e^{r\absolu{x}} \absolu{a_j(x) - a_j^\pm} = 0.
\end{equation*}
We note $\spectre(\Cal{T})$ the spectrum of $\Cal{T}$, \ie the set of complex numbers $\lambda$ such that $\lambda - \Cal{T} : H^n(\R) \subset L^2(\R) \to L^2(\R)$ is not bounded invertible. To study the spectrum, we decompose it into two distinct parts. We say that $\lambda - \Cal{T}$ is Fredholm if its Fredholm index
\begin{equation*}
\fred(\lambda-\Cal{T}) := \dim \ker(\lambda - \Cal{T}) - \codim \range(\lambda - \Cal{T})
\end{equation*}
is defined and finite. We define the point spectrum as 
\begin{equation*}
\spectre[pt](\Cal{T}) := \Set{\lambda \in\spectre(\Cal{T}) : \fred(\lambda - \Cal{T}) = 0},
\end{equation*}
it corresponds to the subset in which the rank-nullity theorem holds.
Then, the essential spectrum is the complementary set:
\begin{equation*}
\spectre[ess](\Cal{T}) := \Set{\lambda \in\spectre(\Cal{T}) : \fred(\lambda - \Cal{T}) \text{ is non defined or nonzero}}.
\end{equation*}
With such definitions, the asymptotic operator with piece-wise constant coefficients
\begin{equation*}
\Cal{T}^\infty(x) := \begin{cases}
\Cal{T}^+ = \sum_{j=0}^n a_j^+ \partial_x^j & \text{if } x>0,\\
\Cal{T}^- = \sum_{j=0}^n a_j^- \partial_x^j & \text{if } x<0,
\end{cases}
\end{equation*}
share the same essential spectrum as $\Cal{T}$, see \cite[p 40]{Kapitula_Promislow}. If both $\Cal{T}^+$ and $\Cal{T}^-$ are elliptic -- in one dimensional space, it comes down to both $(-1)^{n/2}a^+_n$ and $(-1)^{n/2}a_n^-$ being positive -- then $\spectre[ess](\Cal{T}) = \spectre[ess](\Cal{T}^\infty)$ is located to the left of the asymptotic Fredholm borders
\begin{equation*}
\label{e:asymptotic_spectrum}
\spectre(\Cal{T}^+) \cup \spectre(\Cal{T}^-) = \Set{\sum_{j=0}^n a_j^+ (i\xi)^j : \xi\in\R} \cup \Set{\sum_{j=0}^n a_j^- (i\xi)^j : \xi\in\R} \subset \C,
\end{equation*}
see again \cite[Theorem 3.1.13]{Kapitula_Promislow}.\footnote{If $\Cal{T}^\pm$ are elliptic, then both $\lambda - \Cal{T}^\pm$ are invertible for $\real{\lambda}$ large enough. This is equivalent to have $\frac{n}{2}$ stable (respectively unstable) spatial eigenvalues for $\lambda - \Cal{T}^+$ (respectively $\lambda - \Cal{T}^-$) and ensures that the region $S_1$ that contains $(\eta, +\infty)$ -- for $\eta\in\R$ large enough -- does not belong to $\spectre[ess](\Cal{T}^\infty)$.} 

For a matrix operator: $\Cal{T} = \begin{pmatrix}
\Cal{T}_{1,1} & \Cal{T}_{1,2} \\
\Cal{T}_{2,1} & \Cal{T}_{2,2}
\end{pmatrix}$, the asymptotic spectra
$\spectre(\Cal{T}^\pm)$ are still obtained through Fourier transform. The operator $\lambda - \Cal{T}^+$ is invertible \iff the matrix $\lambda - \hat{\Cal{T}}^+(\xi)$ is invertible for all $\xi \in\R$ -- since Fourier transform is an isometry -- which is equivalent to $\xi \mapsto \det(\lambda - \hat{\Cal{T}}^+(\xi))$ never vanishing. In our special case, $\Cal{L}^+$ is triangular, hence the determinant is nonzero if both diagonal coefficients maintain away from zero. We conclude that 
\begin{equation*}
\spectre(\Cal{L}^+) = \spectre(\Lu[,+]) \cup \spectre(\Lv[,+]), 
\end{equation*}
and the same goes at $-\infty$. Once again, the essential spectrum of $\Cal{L}$ is located to the left of $\spectre(\Cal{L}^+) \cup \spectre(\Cal{L}^-)$.

\begin{proposition}[Marginally stable essential spectrum]
\label{p:essential_spectrum}
Fix $\alpha$, $d$ positive. Then there exists $\mu_0>0$ such that for all $0 < \mu < \mu_0$, there exists $\theta < 0$ such that the monotonic weight $\omega_* = \omegau \, \omegav$, with
\begin{equation*}
\omegau(x) = 
\begin{cases}
1 & \text{if } x \leq -1,\\
e^{-\frac{c_*}{2d} x} & \text{if } x \geq 1,
\end{cases}
\hspace{4em}
\omegav(x) = 
\begin{cases}
e^{\theta x} & \text{if } x \leq -1,\\
1 & \text{if } x \geq 1,
\end{cases}
\end{equation*}
satisfies the following. For all $\gamma > \gamma_\Rm{rem}$, with $\gamma_\Rm{rem}$ as in \cite{Faye_Holzer_Scheel_Siemer_20}:
\begin{equation*}
\gamma_\Rm{rem}:= 8\left(\frac{\alpha}{d}\right)^2 + 4\frac{\alpha}{d} - 2\alpha + \mu_0,
\end{equation*}
the operator $\Cal{L} = \omega_*^{-1} \Cal{A} \omega_*$ has marginally stable essential spectrum. 
More precisely, the Fredholm curve corresponding to the essential \kpp spectrum touches the imaginary axis only at $\lambda = 0$:
\begin{equation*}
\spectre(\Lu[,+]) = \Set{-\xi^2 : \xi \in \R},
\end{equation*}
while the three other Fredholm borders have spectral gap: there exists $\eta>0$ depending only on $\mu_0$, $d$, $\alpha$ and $\gamma$ such that
\begin{equation}
\label{e:spectral_stability_eta2}
\spectre(\Lu[,-])\cup \spectre(\Lv[,+]) \cup \spectre(\Lv[,-]) \subset \Set{\lambda\in\C: \real{\lambda}\leq -3\eta}.
\end{equation}
In particular, hypothesis \eqref{e:hyp_marginal_stability} is fulfilled.
Furthermore, $\theta \to 0$ when $\mu_0 \to 0$.
\end{proposition}
\begin{proof}
The asymptotic operators $\Lu[,\pm]$ and $\Lv[,\pm]$ are obtained as conjugation of $\Au[,\pm]$ and $\Av[,\pm]$ with pure exponential weight $e^{\theta x}$ or $e^{-\frac{c_*}{2d} x}$. Direct computations (see \cref{l:conjugaison_operateur_matrice} below) show that $\Lu[,+] = \Au[,+](-\frac{c_*}{2d} + \partial_x)$, where we identify the operator $\Au[,+]$ with its symbol $\Au[,+](X) = dX^2 + c_*X + \alpha$ evaluated at $X = \partial_x$. Hence, the Fredholm border is given by 
\begin{equation*}
\spectre(\Lu[,+]) = \Set{\Au[,+]\left(-\frac{c_*}{2d} + i\xi\right) : \xi\in\R} = \Set{-\xi^2 : \xi\in\R}.
\end{equation*}
Using similar notations for the three other curves, we get
\begin{equation*}
\real\left(\Au[,-](\theta + i\xi)\right) = \real\left(d(\theta + i\xi)^2 + c_*(\theta+i\xi) - 2\alpha\right) \leq -2\alpha + d\theta^2 + c_*\theta < -\alpha 
\end{equation*}
for $\theta \in (-\theta_0, \theta_0)$, where $\theta_0>0$ only depends on $\alpha$ and $d$. Similarly with $\Av[,-](X) = -(X^2 + 1)^2 + cX + \mu$, a direct computation shows that $\real\left(\Av[,-](\theta + i\xi)\right)$ is maximal at $\pm \xi_0 = \pm\sqrt{1+3\theta^2}$, hence:
\begin{equation*}
\real\left(\Av[,-](\theta + i\xi)\right) \leq  
- \theta^4 - 2\theta^2 + c_*\theta + \mu + \xi_0^4 - 1 = \mu + c_* \theta + 4 \theta^2 + 8\theta^4
\end{equation*}
Fix $\eta> 0$ small, it is easily seen that the right hand side is less than $-3\eta$ for some $\theta < 0$ that goes to $0$ when $\mu_0 \to 0$.

Finally for $\Av[,+](X) = \Av[,-](X) - \gamma$, the same calculation with $\zeta_0 := \sqrt{1+ 3(\frac{c_*}{2d})^2}$ shows that
\begin{equation*}
\real\left(\Av[,+]\left(-\frac{c_*}{2d} + i\xi\right)\right) \leq  -\left(\frac{c_*}{2d}\right)^4 - 2\left(\frac{c_*}{2d}\right)^2 - \frac{c_*^2}{2d} + \mu_0 - \gamma + \zeta_0^4 - 1 = \gamma_\Rm{rem} - \gamma.
\end{equation*}
Hence for $\gamma > \gamma_\Rm{rem}$, there exists $\eta>0$ such that $\real\left(\Av[,+]\left(-\frac{c_*}{2} + i\xi\right)\right) < -3\eta$ as claimed.
\end{proof}
\begin{remark}
\label{r:distinct_weights}
We examine the case of distinct weights on \kpp and \sh component. Note $\Omega(x) = \Rm{Diag}(\omega_1(x), \omega_2(x))$, and replace ansatz \eqref{e:ansatz_perturbation} by $\transpose{(u, v)}(t,x) = Q(\tilde{x}) + \Omega(\tilde{x}) U(t, \tilde{x})$, such that perturbation $U$ is driven at linear level by
\begin{equation*}
\partial_t U = 
\begin{pmatrix}
{\omega_1}^{-1}\Au \omega_1 & \frac{\omega_2}{\omega_1} \beta \\
0 & {\omega_2}^{-1}\Av \omega_2
\end{pmatrix} U.
\end{equation*}
For our study to be relevant, we need $\frac{\omega_2(x)}{\omega_1(x)}$ to be bounded \wrt $x$. The computations done in \cref{p:point_spectrum} shows that both $\omega_1(x) \leq e^{-\frac{c_*}{2d}x}$ for $x\geq 1$ and $\omega_2(x) \geq e^{\theta x}$ for $x\leq -1$ are necessary to obtain marginal spectral stability. With the condition $\omega_2(x)\leq C \omega_1(x)$ for $x\in\R$, we are left with $\omega_1 = \omega_2$.
\end{remark}

\begin{proposition}[Stable point spectrum]
\label{p:point_spectrum}
With the same assumptions on $\mu_0$, $\mu$, $\alpha$, $d$, $\gamma$, $\eta$ as in previous \cref{p:essential_spectrum}, the operator $\Cal{L}$ has no eigenvalue in the following set:
\begin{equation*}
\Omega_\Rm{eig}:=\Set{\lambda \in\C : \real{\lambda} > -3\eta}\backslash (-\infty, 0).
\end{equation*}
\end{proposition}
\begin{proof}
We assume by contradiction that there exists $\lambda\in\Omega_\Rm{eig}$ and a nonzero $\phi = (\phi^\symbkpp, \phi^\symbsh)\in H^2(\R, \R)\times H^4(\R,\R)$ such that 
\begin{equation*}
(\lambda - \Cal{L}) \phi = 0.
\end{equation*}
In particular, the second eigenproblem is decoupled: $(\lambda - \Lv)\phi^\symbsh = 0$. We show it has no other solution than $\phi^\symbsh = 0$ in $H^4(\R)$, which will imply that the first eigenproblem $(\lambda - \Lu)\phi^\symbkpp = \beta \phi^\symbsh$ admits no other solution than $\phi^\symbkpp = 0$. This is a contradiction.

Both $\Lv$ and $\Cal{A}^\symbsh$ are not easy to work with, because they respectively have non-constant coefficients or unstable essential spectrum. Instead, we use
\begin{equation*}
\Cal{B}^\symbsh := e^{-\theta x} \Cal{A}^\symbsh e^{\theta x} = e^{-\theta x} \omega_* \Lv \omega_*^{-1} e^{\theta x}
= b_0(x) + \sum_{j=1}^4 b_j \partial_x^j,
\end{equation*}
with $\theta$ and $\omega_*(x)$ defined in the above \cref{p:essential_spectrum}. Then $\varphi := e^{-\theta x} \omega_* \phi^\symbsh$ lies in $H^4(\R)$ due to $e^{-\theta x} \omega_*(x) \leq 1$,\footnote{Recall that $-\frac{c_*}{2} \leq \theta$ when $\mu$ is small enough.} and satisfies $(\lambda - \Cal{B}^\symbsh) \varphi = 0$. Take the $L^2(\R, \C)$-inner product with $\varphi$:
\begin{equation*}
\lambda\norme{\varphi}_{L^2}^2 = \pscal{b_0(x)\varphi, \varphi} + \sum_{j=1}^4 b_i \pscal{\partial_x^j \varphi, \varphi}.
\end{equation*}
Then $\pscal{\partial_x^j \varphi, \varphi} = (-1)^j \, \overline{\pscal{\partial_x^j \varphi, \varphi}}$, so that the real part of the above equality writes:
\begin{equation}
\label{e:no_eigenvalue_1}
\real(\lambda) \, \norme{\varphi}_{L^2}^2 = \int_{\R} b_0 \absolu{\varphi}^2 \d x - b_2 \norme{\partial_x \varphi}_{L^2}^2 + b_4 \norme{\partial_x^2 \varphi}_{L^2}^2 \leq -3\eta\norme{\varphi}_{L^2}^2.
\end{equation} 
The inequality is obtained using that $b_2 = 2(1+3\theta^2) > 0$, that $b_4 = -1 <0$ and finally that
\begin{equation*}
b_0(x) = \mu - \gamma(1 - q_*(x)) + c_* \theta - (\theta^2+1)^2 \leq \mu + c_*\theta \leq \mu + c_*\theta + 4 \theta^2 + 8\theta^4 \leq -3\eta.
\end{equation*} 
Recall we have chosen $\real{\lambda} > -3\eta$, hence \eqref{e:no_eigenvalue_1} implies $\varphi=0$ and then $\phi^\symbsh = 0$ as claimed.

Now the first eigenproblem writes $(\lambda - \Lu)\phi^\symbkpp = 0$, hence $\omega_* \phi^\symbkpp \in H^2(\R)$ is an eigenfunction for $\Au$. Using Sturm-Liouville theory -- see \emph{e.g.} \cite[p 33.]{Kapitula_Promislow} --  eigenvalues of $\Au$ are real and non-positive, since $q_*'\in H^2(\R)$ does not vanish and satisfies $\Au q_*' = 0$. Hence $\Au$ has no eigenvalues outside of $(-\infty, 0]$, and  so does $\Lu$. Since $\frac{1}{\omega_*(x)}q_*'(x) \sim a x + b$ when $x\to +\infty$, the derivative of the front does not contribute to an eigenfunction for $\Lu$, hence it has no eigenvalue oustide $(-\infty, 0)$. The first eigenproblem imposes either $\lambda\in (-\infty, 0)$ or $\phi^\symbkpp = 0$, which is a contradiction, and complete the proof.
\end{proof}

\subsection{Construction of decaying ODE solutions}
Here we solve the linear non-autonomous ODE $(\lambda - \Cal{L}^i)\phi = 0$, with unknown $\phi$ and $i\in\Set{\symbkpp, \symbsh}$. To keep notations simple, we concentrate on $i = \symbsh$. Since $\Lv$ is exponentially asymptotic, $\phi$ expresses at first order using the solutions of the asymptotic ODE, as described in the incoming Lemma. After vectorialization, $(\lambda - \Lv)\phi = 0$ writes $\partial_x \Phi = A^\symbsh(\lambda, x) \Phi$, with a matrix $A^\symbsh(\lambda, x)$ that exponentially converges towards matrices $A^{\symbsh, \pm}(\lambda)$ when $x\to \pm\infty$. 
\begin{lemma}
\label{l:ODE_solutions}
Let $A : t\in\R \mapsto A(t) \in\Cal{M}_n(\C)$ be a continuous, matrix-valued function, that converges at exponential speed towards $A_\infty$ when $t\to +\infty$. Let $\alpha>0$ that satisfies $\norme{A(t) - A_\infty} \leq Ce^{-\alpha t}$ for $t\geq 0$.\footnote{Here $\norme{.}$ is any operator norm.} Let $(v_i)_{1\leq i\leq n}$ be a basis of $\C^n$. Then there exists $(y_i)_{1\leq i \leq n}$ a basis of solutions for the ODE
\begin{equation}
\label{e:lemme_edo:edo1}
y'(t) = A(t) y(t),
\end{equation}
and $\kappa_i\in L^\infty(0, +\infty)^n$ such that for $t\geq 0$ we have
\begin{equation*}
y_i(t) = e^{tA_\infty} (v_i + e^{-\alpha t} \kappa_i(t)).
\end{equation*}
Furthermore, if $\lambda \mapsto A(t, \lambda)$, $A_\infty$ and $v$ are holomorphic with respect to an extra parameter $\lambda$, then $\lambda \mapsto \kappa(\lambda, \cdot)$ is also holomorphic.
\end{lemma}
We delay the proof to the later \cref{ss:proof_lemma_ODE}, and apply this Lemma to the matrix $A^\symbsh(\lambda, x)$, for $\lambda$ to the right of $\spectre(\Lv)$. There, eigenvalues of $A^{\symbsh, \pm}$ are distinct and holomorphic. Since asymptotic matrices are companion, they admit a basis $\left(v_i^{\symbsh, \pm}(\lambda)\right)$ of holomorphic eigenvectors.
This allows to define two basis of solutions $(\phi^{\symbsh, +}_i(\lambda, \cdot))_{1\leq i \leq 4}$ and $(\phi^{\symbsh, -}_i(\lambda, \cdot))_{1\leq i \leq 4}$ for the ODE $(\lambda - \Lv) \phi = 0$, with exponential behavior at $+\infty$ or $-\infty$:
\begin{equation*}
\phi_i^{\symbsh, \pm} = e^{\nu_i^{\symbsh, \pm} x} (v_i^{\symbsh, \pm} + e^{-r \absolu{x}} \kappa_i^{\symbsh, \pm}(x)),
\end{equation*}
where $\nu_i^{\symbsh, +}$ is the eigenvalue associated to $v_i^{\symbsh, +}$, and $\kappa_i^{\symbsh, +} \in L^\infty(\R_+)$. 
We use similar notations for $\Lu$.
\begin{lemma}
\label{l:spatial_eigenvalues_localization}
The dispersion relation $\lambda - \Lu[,\pm](\nu) = 0$ and $\lambda - \Lv[,\pm](\nu) = 0$ are respectively second order and fourth order polynomial in $\nu$. Their respective roots $\nu_i^{\symbkpp, \pm}(\lambda)$ and $\nu_i^{\symbsh, \pm}(\lambda)$ have the following localization:
\begin{enumerate}[label=(\arabic*)]
\item\label{i:spectral_gap} (spectral gap away from the spectrum)
There exists $\kappa_2>0$ such that for all $\lambda \in \C$ with $\real{\lambda} \geq -2\eta$, 
\begin{equation*}
\absolu{\real{\nu}}\geq \kappa_2,
\end{equation*}
where $\nu$ stands for $\nu_i^{\symbsh, \pm}(\lambda)$ or $\nu_i^{\symbkpp, -}(\lambda)$.
\item \label{i:pinched_double_root} (pinched double root at the origin) Let $I = (-\infty, 0) \subset\C$, and $V$ be a neighborhood of $0\in \C$. There exists $C$ positive such that for $\lambda \in V\backslash I$, and $\lambda \to 0$, 
\begin{equation*}
\nu_1^{\symbkpp, +}(\lambda) \sim - C\sqrt{\lambda}, \hspace{4em}
\nu_2^{\symbkpp, +}(\lambda) \sim C\sqrt{\lambda}.
\end{equation*}
\item \label{i:elliptic_operators} (elliptic operators) There exists $R$, $C$, $\theta$ positive constant such that for $\lambda\in \C$ with $\absolu{\lambda}\geq R$ and $\real{\lambda} \leq - \theta \absolu{\imag{\lambda}}$,
\begin{equation*}
\absolu{\real{\nu_i^{\symbkpp,\pm}(\lambda)}} \geq C \absolu{\lambda}^{1/2},
\hspace{4em}
\absolu{\real{\nu_i^{\symbsh,\pm}(\lambda)}} \geq C \absolu{\lambda}^{1/4}.
\end{equation*}
\end{enumerate}
\end{lemma}
\begin{proof}
\Cref{i:spectral_gap,i:pinched_double_root} are easily obtained. \Cref{i:elliptic_operators} relies on a scaling argument, see \cite[Lemma 3.1]{Faye_Holzer_18} for $\nu_i^{\symbkpp, \pm}$, the case of $\nu_i^{\symbsh, \pm}$ adapts as follow. First set $\psi(x) := \phi(x\absolu{\lambda}^{1/4})$, and apply \cref{l:ODE_solutions} to construct solutions that are close to solutions of $\lambda + \absolu{\lambda}\partial_x^4 \psi = 0$. In particular, the asymptotic matrix eigenvalues for $\psi$ are $\tilde{\nu}(\lambda) := z \sqrt[4]{\lambda}/\absolu{\lambda}$, where $z \in \Set{e^{i\frac{\pi}{4}}, e^{-i\frac{\pi}{4}}, e^{i\frac{3\pi}{4}}, e^{-i\frac{3\pi}{4}}}$. Notice that $\absolu{\real{z}} = 1/\sqrt{2}$, hence we get $\absolu{\real{\tilde{\nu}}} \geq C$, by restricting to $\lambda$ of the form $re^{i\theta}$, with $\absolu{\theta} \leq 3\pi /4$. Reverting back to the  original variable, the asymptotic matrix eigenvalues for $\phi$ satisfy $\absolu{\real{\nu}} \geq C \absolu{\lambda}^{1/4}$.
\end{proof}

In the following, we will compute determinant of several matrices, whose columns depend on $\phi_i^{\symbsh, \pm}$ or $\phi_i^{\symbkpp, \pm}$. Suppose we are given $n$ scalar functions $\phi_1, \dots, \phi_n$ depending on the space variable $x$. Then, we write 
\begin{equation}
\label{e:notation_determinant}
\Det(\phi_1, \dots, \phi_n) := \det 
\begin{pmatrix}
\phi_1 & \cdots & \phi_n\\
\vdots & & \vdots\\
\partial_x^{n-1} \phi_1 & \cdots & \partial_x^{n-1} \phi_n
\end{pmatrix} = \det\left(\left(\partial_x^{j-1}\phi_i\right)_{1\leq i,j\leq n}\right).
\end{equation}
where \guillemet{$\det$} stands for the determinant. We also introduce the wronskian function composed of all decaying solutions:
\begin{equation*}
\label{e:def_wronskian}
\Wu := \Det(\phi_1^{\symbkpp, +}, \phi_2^{\symbkpp, -}), 
\hspace{4em}
\Wv := \Det(\phi_1^{\symbsh, +}, \phi_2^{\symbsh, +}, \phi_3^{\symbsh, -}, \phi_4^{\symbsh, -}).
\end{equation*}
For each $\lambda \in \C$, the function $y\mapsto \Wu(\lambda, y)$ is either identically zero or does not vanish, see \cref{l:ode_wronskien}. The particular value $\Wu(\lambda, 0)$ is often called the Evans function. When $\lambda$ is such that $\phi_1^{\symbkpp, +}$ and $\phi_2^{\symbkpp, -}$ are decaying at $x\to \pm\infty$,\footnote{When $\lambda$ lies to the right of $\spectre[ess](\Lu)$, this decay property hold.} $\Wu(\lambda, 0)$ vanishes exactly when $\lambda$ is an eigenvalue for $\Lu$, with same multiplicity. The same goes for $\Wv$.

\subsection{Construction and estimations of the kernel for the resolvent operator}
\label{ss:construction_spectral_Green}
First, we construct and control the Green function $\Glu(x,y)$, which corresponds to the pure \kpp equation. Remark that such a result was already proved in \cite[Lemma 3.2]{Faye_Holzer_18}. We rewrite it in a more condensed form that separates the behaviors at $+\infty$ and $-\infty$. See the further \cref{r:extra_spatial_decay}.

\begin{proposition}
\label{p:control_G_lambda_11} Let $K\subset \C$ be a compact set to the right of $\spectre(\Lu[,-])$, and note $I = (-\infty, 0)$ the real negative axis that corresponds to the absolute spectrum of $\Lu$. Then $\lambda \mapsto \Glu(x, y)$ is holomorphic from $K\backslash I$ to $\R$. Furthermore, there exists positive constants $\kappa_1$, $C$ such that for all $\lambda \in K\backslash I$ and $0 \leq j \leq 2$, we have the following.
If $x\geq 0$ or $y\geq 0$,
\begin{equation*}
\label{e:majoration_G_lambda_11_xypositif}
\Absolu{\Glu(x,y)} \leq C e^{-\real{\sqrt{\lambda}}\absolu{x-y}} \spacemath  \omega_{\kappa_1, 0}(x) \spacemath \omega_{\kappa_1, 0}(y).
\end{equation*}
If on the contrary, $x\leq 0$ and $y\leq 0$, then
\begin{equation*}
\label{e:majoration_G_lambda_11_xynegatif}
\Absolu{\Glu(x,y)} \leq C e^{-\kappa_1\absolu{x-y}}.
\end{equation*}
\end{proposition}
\begin{proof}
We use the decaying ODE solution $\phi_2^{\symbkpp, -}$, $\phi_1^{\symbkpp, +}$ constructed above. To keep notation as light as possible, we write them $\phi^-$ and $\phi^+$ respectively. Similarly, we note $\psi^- = \phi^{\symbkpp, -}_1$ and $\psi^+ = \phi^{\symbkpp, +}_2$ the growing ODE solutions. Remind that both $(\phi^-, \psi^-)$ and $(\phi^+, \psi^+)$ are basis for solution of the ODE $(\lambda - \Lu) \varphi = 0$. The Green function expresses as 
\begin{equation*}
\Glu(x, y) = \frac{1}{\Wu(\lambda, y)} 
\begin{cases}
\phi^-(y) \phi^+(x) & \text{ if } y<x, \\
\phi^-(x) \phi^+(y) & \text{ if } x<y.
\end{cases}
\end{equation*}
We begin by the case $x\geq 0$ or $y\geq 0$. For $\lambda$ to the right of $\spectre(\Lu[,-])$, the spatial eigenvalues are such that the exponential obtained in \cite[Lemma 3.2]{Faye_Holzer_18} are smaller than $e^{-\absolu{x-y}\real{\sqrt{\lambda}}}$. Then each $\Cal{O}$ leads to our weights $\omega_{\kappa_1, 0}(x) \omega_{\kappa_1, 0}(y)$. The case $x\leq 0$ and $y\leq 0$ is in \cite[Lemma 3.2]{Faye_Holzer_18}.
\end{proof}

In the following, we will also need to bound the derivatives of the spectral Green kernel. Since our problem is parabolic, to differentiate the semigroup must gains us extra temporal decay. At spectral level, it converts into extra power of $\sqrt{\lambda}$.
Due to $\Glu$ being a kernel, Dirac delta correction terms appear when differentiating more than the order of $\Lu$. However, they will be easily absorbed at temporal level.

\begin{proposition}
\label{p:control_partial_G_lambda_11} Let $K\subset \C$ be a compact set to the right of $\spectre(\Lu[,-])$, and note $I = (-\infty, 0)$ the real negative axis that corresponds to the absolute spectrum of $\Lu$. Fix an integer $j\geq 1$. There exists a sum of Dirac delta derivatives:
\begin{equation*}
P_j(\delta_{x=y}) = \sum_{k = 0}^{j-2} a_k^j(\lambda, y) \delta_{x = y}^{(k)}
\end{equation*}
with coefficients $a_k$ holomorphic \wrt $\lambda \in K\backslash I$ and bounded with respect to $y\in\R$, such that $\lambda \mapsto \partial_x^j\Glu(x, y) - P_j(\delta_{x = y})$ is holomorphic from $K\backslash I$ to $\R$. Furthermore, there exists positive constants $\kappa_1$, $C$ such that for all $\lambda \in K\backslash I$, we have the following. 
If $x\geq 0$ or $y\geq 0$,
\begin{equation*}
\label{e:majoration_partial_G_lambda_11_xypositif}
\Absolu{\partial_x^j \Glu(x,y) - P_j(\delta_{x = y})} \leq C \absolu{\lambda}^{j/2} e^{-\real{\sqrt{\lambda}}\absolu{x-y}} \spacemath  \omega_{\kappa_1, 0}(x) \spacemath \omega_{\kappa_1, 0}(y).
\end{equation*}
If on the contrary, $x\leq 0$ and $y\leq 0$, then
\begin{equation*}
\label{e:majoration_partial_G_lambda_11_xynegatif}
\Absolu{\partial_x^j \Glu(x,y) - P_j(\delta_{x = y})} \leq C \absolu{\lambda}^{j/2} e^{-\kappa_1\absolu{x-y}}.
\end{equation*}
\end{proposition}
\begin{proof}
For the case $j = 1$ no correction terms appear, and the proof is close to \cite{Faye_Holzer_18}. We compute
\begin{equation*}
\partial_x \Glu(\cdot, y) = \frac{1}{\Wu_\lambda(y)} 
\begin{cases}
\phi^-(y) \partial_x \phi^+ & \text{ on } (y, + \infty)\\
\phi^+(y) \partial_x \phi^- & \text{ on } (- \infty, y)\\
\end{cases}
\end{equation*}
Begin first with case $x > y$.
When $x\geq 0$, we use $\partial_x \phi^+(x) = -\sqrt{\lambda} e^{-\sqrt{\lambda}}(1 + \kappa(x))$ rather than $\phi^+(x) = e^{-\sqrt{\lambda}}(1 + \kappa(x))$, which gains us the claimed $\absolu{\lambda}^{1/2}$ in comparison with the case $j = 0$. When $x\leq 0$, we project onto the $(\phi^-, \psi^-)$ basis:
\begin{equation*}
\partial_x \phi^+(x) = \frac{\Det(\partial_x \phi^+, \psi^-)}{\Det(\phi^-, \psi^-)} \phi^-(x) + \frac{\Det(\phi^-, \partial_x \phi^+)}{\Det(\phi^-, \psi^-)} \psi^-(x),
\end{equation*}
we refer to the above mentioned reference for more details, or to the proof of \cref{p:control_G_lambda_22}, see \cref{ss:proof_Glv}. 
Since both $\phi^-$ and $\psi^-$ are bounded by $e^{\kappa_1 x}$, and using that $\Det(\phi^-, \psi^-) > 0$ uniformly \wrt $\lambda$, we get to $\absolu{\partial_x \phi^+(x)} \leq C \absolu{\lambda}^{1/2} e^{\kappa_1 x}$.
This conclude the case $j = 1$. 

When $j = 2$, we compute
\begin{equation*}
\partial_x^2 \Glu(\cdot, y) = \delta_y + \frac{1}{\Wu_\lambda(y)} 
\begin{cases}
\phi^-(y) \partial_x^2 \phi^+ & \text{ on } (y, + \infty),\\
\phi^+(y) \partial_x^2 \phi^- & \text{ on } (- \infty, y),\\
\end{cases}
\end{equation*}
and similar computations as above show the claimed estimate for  $\partial_x^2 \Glu(x,y) - \delta_y$.

For $j \geq 2$, the correction terms can be computed recursively:
\begin{equation*}
P_{j+1} (\delta_{x = y}) = \delta_{x = y} \frac{\phi^- \partial_y^{j+1} \phi^+ - \phi^+ \partial_y^{j+1} \phi^-}{W_\lambda(y)}+ \sum_{k = 1}^{j-1} a_k^j(y) \delta_{x = y}^{(k)}
\end{equation*}
Then, $\partial_x^j \Glu(x,y) - P_j(\delta_{x = y})$ is handled as above.
\end{proof}
We state a similar result for $\Glv$. Since $\Lv$ is a fourth order operator, the proof slightly differs. Due to its length, we present the proof in the later  \cref{ss:proof_Glv}. Note that \cite{Howard_Hu_04} presents similar computations in a viscous shock front context.
\begin{proposition}
\label{p:control_G_lambda_22}
Let $K\subset \C$ be a compact set to the right of $\spectre(\Lv[,-]_\theta) \cup \spectre(\Lv[,+])$, for example $K\subset \Set{\real(\lambda) \geq -2 \eta}$. Then $\lambda \mapsto \Glv(\cdot,y)$ is holomorphic from $K$ to $H^4(\R\backslash \lbrace y \rbrace)$, and there exists $C, \kappa_2 > 0$ such that the following holds. For all $x,y\in\R^2$, and $\lambda\in K$:
\begin{equation*}
\Absolu{\Glv(x,y)} \leq C e^{-\kappa_2\absolu{x-y}}.
\end{equation*}
Furthermore, if $x\geq 0$ or $y\geq 0$, then we have:
\begin{equation}
\label{e:control_G_lambda_22_extra}
\Absolu{\Glv(x,y)} \leq C e^{-\kappa_2\absolu{x-y}} \omega_{\kappa_2, 0}(y).
\end{equation}
\end{proposition}

We can now obtain a statement on $\Gluv$, similar to the above \cref{p:control_G_lambda_11}. We use the structure of the equation satisfied by $\Gluv$ to directly transfer the spatial decay from $\Glu$ to $\Gluv$.
\begin{proposition}
\label{p:control_G_lambda_12}
Let $K\subset \C$ be a compact set to the right of $\spectre(\Lu[,-]) \cup \spectre(\Lv[,-]_\theta) \cup \spectre(\Lv[,+])$, and note $I = (-\infty, 0)$.
Then $\lambda \mapsto \Gluv(\cdot, y)$ is holomorphic from $K\backslash I$ to $H^2(\R\backslash \lbrace y \rbrace)$. Furthermore, there exists $\kappa_3, C > 0$ such that for all $\lambda \in K\backslash I$, we have the following. If $x\geq 0$ or $y\geq 0$, then:
\begin{equation*}
\label{e:majoration_G_lambda_12_xypositif}
\absolu{\Gluv(x,y)} \leq C e^{-\real{\sqrt{\lambda}}\absolu{x-y}} \spacemath \omega_{\kappa_3, 0}(x)\spacemath\omega_{\kappa_3, 0}(y).
\end{equation*}
If on the contrary $x\leq 0$ and $y\leq 0$,
\begin{equation*}
\label{e:majoration_G_lambda_12_xynegatif}
\absolu{\Gluv(x,y)} \leq C e^{-\kappa_3\absolu{x-y}}.
\end{equation*}
\end{proposition}
\begin{proof}
Recall that
\begin{equation}
\label{e:eqn_G_12}
(\lambda - \Lu)\Gluv = \beta \Glv.
\end{equation}
Since $\lambda$ is located to the right of $\spectre(\Lu)$, the operator $\lambda - \Lu$ is invertible, and we get 
\begin{equation*}
\Gluv(x,y) = \beta \int_\R \Glu(x,\tau) \Glv(\tau, y) \d \tau.
\end{equation*}
Indeed, we have constructed $\Glu(x,y)$ as the solution of the fundamental problem associated with \eqref{e:eqn_G_12}, \ie replace the right hand side by a Dirac delta. When the Dirac delta is replaced by any source term, it give birth to a convolution-like solution as above. 
We first assume $x\geq 0$ or $y\geq 0$. Then from \cref{p:control_G_lambda_11}, where we neglect the $\omega_{\kappa_1, 0}(\tau)$ terms, and from the refined estimate \eqref{e:control_G_lambda_22_extra} in \cref{p:control_G_lambda_22}, we obtain that
\begin{equation*}
\absolu{\Gluv(x,y)} \leq C \int_\R e^{-\real{\sqrt{\lambda}}\absolu{x-\tau}} \omega_{\kappa_1, 0}(x)\spacemath e^{-\kappa_2\absolu{\tau-y}} \omega_{\kappa_2, 0}(y) \spacemath \d\tau.
\end{equation*}
Using triangular inequality $\absolu{x-\tau} + \absolu{\tau - y} \geq \absolu{x-y}$, we obtain
\begin{align*}
\absolu{\Gluv(x,y)} & \leq C e^{-\real{\sqrt{\lambda}} \absolu{x-y}} \omega_{\kappa_1, 0}(x) \spacemath \omega_{\kappa_2, 0}(y)\int_\R e^{-(\kappa_2-\real{\sqrt{\lambda}})\absolu{\tau-y}} \d\tau,\\ 
& \leq C e^{-\real{\sqrt{\lambda}} \absolu{x-y}}\omega_{\kappa_1, 0}(x) \spacemath \omega_{\kappa_2, 0}(y),
\end{align*}
which is the first claimed estimate.
Now when $x\leq 0$ and $y\leq 0$, we first assume that $\kappa_1\leq \kappa_2$. We introduce $\kappa_3 := \frac{1}{2}\kappa_1$, and obtain -- again using triangular inequality -- that
\begin{equation*}
\absolu{\Gluv(x,y)} \leq C \int_\R e^{-\kappa_1\absolu{x-\tau}} e^{-\kappa_2\absolu{\tau-y}} \d\tau \leq C e^{-\kappa_3\absolu{x-y}} \int_\R e^{-\kappa_3\absolu{x-\tau}}e^{-(\kappa_2 - \kappa_3)\absolu{\tau-y}}\leq C e^{-\kappa_3\absolu{x-y}}.
\end{equation*}
If on the contrary we have $\kappa_1>\kappa_2$, we set $\kappa_3 := \frac{1}{2} \min(\kappa_1, \kappa_2)$ and obtain the same inequality.
\end{proof}

To finish this section, we show bounds on $G_\lambda$ for large $\lambda$. This is needed to transfer spectral behavior into temporal decay. Indeed, $G_\lambda$ will be integrated along a spectral contour that surround the spectrum and hence goes to infinity.

\begin{proposition}
\label{p:control_G_lambda_large}
There exists $R, C>0$ such that: for all $\lambda$ to the right of $\spectre(\Cal{L})$, with $\absolu{\lambda}\geq R$, and all $x,y\in\R^2$, we have 
\begin{equation}
\label{e:large_lambda_estimate}
\Absolu{\Glu(x,y)} \leq \frac{C}{\absolu{\lambda}^{1/2}}e^{-\absolu{\lambda}^{1/2}\absolu{x-y}}, 
\hspace{4em}
\Absolu{\Glv(x,y)} \leq \frac{C}{\absolu{\lambda}^{3/4}}e^{-\absolu{\lambda}^{1/4}\absolu{x-y}}.
\end{equation}
The first estimate still holds if $\Glu[]$ is replaced by $\Gluv[]$. In particular, for $i\in\lbrace \symbkpp, \symbsh, \symbcouple\rbrace$
\begin{equation*}
\Norme{G_\lambda^i}_{L^\infty(\R_x, L^1(\R_y))} \leq \frac{C}{\absolu{\lambda}},
\hspace{4em}
\Norme{G_\lambda^i}_{L^\infty(\R_y, L^1(\R_x))} \leq \frac{C}{\absolu{\lambda}}.
\end{equation*}
\end{proposition}
\begin{proof}
The last claimed inequalities are easily obtained from the first ones, and implies that $\norme{G_\lambda^i \cdot v}_{L^p(\R)} \leq \frac{C}{\absolu{\lambda}} \norme{v}_{L^p(\R)}$ using complex interpolation. 
The first estimates are obtained through the scaling argument began in the proof of \cref{l:spatial_eigenvalues_localization}. The second ones are necessary to obtain small time estimates on $\int_\Gamma e^{\lambda t} G_\lambda \d \lambda$, with $\Gamma$ being a ray that extends at infinity.
\end{proof}

\subsection{Linear dynamic}
We now convert spatial localization of spectral Green functions $G_\lambda$ into temporal decay in appropriate space using the inverse Laplace transform and adequate integration contours.
\begin{proposition}
\label{p:temporal_Green_v}
There exists positive constants $C$, $\eta$ such that:
\begin{equation*}
\absolu{\Gtv(x,y)} \leq Ce^{-\eta t}e^{-\kappa_2\absolu{x-y}}.
\end{equation*}
In particular, for any $1\leq p \leq +\infty$, we have 
\begin{equation}
\label{e:linear_estimate_v}
\norme{\Gtv \cdot w}_{L^p(\R)} \leq C e^{-\eta t}\norme{w}_{L^p(\R)},
\end{equation}
where we have noted $(\Gtv \cdot w)(x) := \int_\R \Gtv(x,y) w(y) \d y$.
\end{proposition}
\begin{proof}
Since $\Lv$ is sectorial with spectral gap $\eta$ from \cref{p:essential_spectrum,p:point_spectrum}, this proof easily follows from the spatial localization of $\Glv$, see \cref{p:control_G_lambda_22,p:control_G_lambda_large}. It is enough to use inverse Laplace transform with a spectral contour in $\Set{\lambda \in \C : \real(\lambda) \leq - \frac{\eta}{2}}$ for small $\lambda$, while restrict to a contour in $\Set{\lambda \in \C : \Rm{Arg}(\lambda) \leq 3\pi/4}$ in the region of large $\lambda$. Up to a change of notation $\tilde{\eta} = \eta/2$, the proof is complete.
\end{proof}

\begin{remark}
\label{r:extra_spatial_decay}
We now state a similar result for $\Gtu$ and $\Gtuv$. Remark that in \cite[Proposition 4.1]{Faye_Holzer_18}, the situation is similar except that the weight $\omega_*$ was exponentially decaying both at $+\infty$ and $-\infty$. In our setting, since $\omega_*(x) \to +\infty$ when $x\to -\infty$, we can not absorb any supplementary polynomial during the nonlinear argument. Hence, we need to refine the aforementioned result with respect to the polynomial weight. We use the spatial decay $\omega_{\kappa_1,0}(y)$ obtained for both $\Glu$ and $\Gluv$ in the above \cref{p:control_G_lambda_11,p:control_G_lambda_12}. This allows to bypass the weight at $-\infty$ during nonlinear argument.
\end{remark}

\begin{remark}
Depending on one aim, the following pointwise bounds for the derivatives may not be adapted. We allow ourselves to keep a single power of $\absolu{x-y}$, which transfers into $t^{-3/2}$ decay. To show stronger decay, it is necessary to keep more powers of $\rho_*$, thus weakening the weight. In our situation, more powers of $\rho_*$ would result in the translation eigenvalue $\lambda = 0$ appearing again, due to $\frac{q_*'}{\omegau \rho_*^2} \sim \frac{1}{x}$ at $x\to +\infty$. To guess what bound may be obtained, one can differentiate $\frac{1}{t^{1/2}}e^{-\frac{x^2}{t}}$ (resp. $\frac{x}{t^{3/2}}e^{-\frac{x^2}{t}}$) in the case where $\absolu{x -y} \geq Kt$ (resp. $\absolu{x -y} \leq Kt$)
\end{remark}

\begin{proposition}
\label{p:temporal_Green_u}
Recall that $\rho_*$ is defined by \eqref{e:def_algebraic_weight}. Fix $0 \leq j \leq 1$ and restrict to $t>1$. Then there exists positive constants $C$, $K$ and $\kappa_1$ such that the following two pointwise estimate hold.
If $\absolu{x-y} \geq Kt$, then 
\begin{equation*}
\Absolu{\partial_x^j \Gtu(x,y)} \leq \frac{C}{t^{\frac{j+1}{2}}} \left(1 + \frac{\absolu{x - y}}{\sqrt{t}}\right)^{j} e^{- \kappa_1\frac{\absolu{x-y}^2}{t}} h(x,y).
\end{equation*}
If on the contrary $\absolu{x-y} \leq Kt$, then
\begin{equation*}
\Absolu{\partial_x^j \Gtu(x,y)} \leq \frac{1 + \absolu{x - y}^{1+\floor{j/2}}}{t^{3/2 + \floor{j/2}}}
e^{- \kappa_1\frac{\absolu{x-y}^2}{t}} h(x,y).
\end{equation*}
In the above expressions, $x,y\mapsto h(x,y)/\rho_*(y)^2$ is integrable in $y$ uniformly in $x$:
\begin{equation*}
0 \leq h(x,y) \leq C \begin{cases}
\omega_{\kappa_1, 0}(x)\ \omega_{\kappa_1, 0}(y) & \text{if } x\geq 0 \text{ or } y \geq 0, \\
e^{-\kappa_1 \absolu{x-y}} & \text{if } x\leq 0 \text{ and } y \leq 0.
\end{cases}
\end{equation*}
In particular, for $0 \leq j \leq 1$, for all $1\leq q \leq p \leq +\infty$, all $x,y\in\R$ and $t\geq 1$,
\begin{equation}
\label{e:linear_estimate_u-1}
\Norme{\frac{\partial_x^j \Gtu\cdot w}{\rho_*}}_{L^p(\R)} \leq \frac{C}{t^{\frac{3}{2} - \frac{1}{2 p}}} \Norme{\rho_*^3 w}_{L^q(\R)}.
\end{equation}
All the above still hold when $\Gtu$ and $\kappa_1$ are replaced by $\Gtuv$ and $\kappa_3$ respectively.
\end{proposition}

\begin{proposition}
\label{p:temporal_Green_partial_u}
There exists positive constants $C$, $\kappa_1$ and a constant $c\neq 0$ such that for $t>1$, the two point-wise estimate of the above proposition hold with $j = 2$ for
\begin{equation*}
\Absolu{\partial_x^2 \Gtu(x,y) - c e^{-t} \delta_{x = y}}.
\end{equation*}
Hence, for all $1\leq q\leq p \leq +\infty$, all $x,y \in \R$ and $t\geq 1$,
\begin{equation}
\label{e:linear_estimate_u-2}
\Norme{\frac{\partial_x^2 \Gtu\cdot w}{\rho_*}}_{L^p(\R)} \leq \frac{C}{t^{\frac{3}{2} - \frac{1}{2 p}}} \Norme{\rho_*^3 w}_{L^q(\R)} + C e^{-t} \norme{w}_{L^p(\R)}.
\end{equation}
All the above still hold when $\Gtu$ and $\kappa_1$ are replaced by $\Gtuv$ and $\kappa_3$ respectively.
\end{proposition}

\begin{proof}
\textit{\Cref{p:temporal_Green_u,p:temporal_Green_partial_u}.\hspace{2ex}}
We first investigate the case $j = 0$. We recall that \cite[Proposition 4.1]{Faye_Holzer_18} proves there exists $K>0$ such that:
\begin{itemize}
\item if $\absolu{x-y} \geq K t$, then $\Absolu{\Gtu(x,y)} \leq \frac{C}{\sqrt{t}}e^{-\kappa\frac{\absolu{x-y}^2}{t}}$,
\item else, $\absolu{x-y} \leq K t$ implies $\Absolu{\Gtu(x,y)} \leq C\frac{1 + \absolu{x-y}}{t^{3/2}}e^{-\kappa\frac{\absolu{x-y}^2}{t}}$.
\end{itemize}
Having kept the extra exponential localization -- $\omega_{\kappa_1, 0}(x)\omega_{\kappa_1, 0}(y)$  when $x\geq 0$ or $y\geq 0$ -- in our estimates of $\Glu(x,y)$, the proof of \cite{Faye_Holzer_18} straightforwardly adapt with this extra decay. The same goes for the other case $x\leq 0$ and $y\leq 0$.

In case where $j > 0$, derivatives pass through the inverse Laplace transform formula, and we can choose a contour $\Lambda$ that is sectorial before differentiation:
\begin{equation*}
\partial_x^j \Gtu(\cdot, y) = \frac{1}{2i\pi} \int_\Lambda e^{\lambda t} \partial_x^j \Glu(\cdot, y) \d \lambda, 
\end{equation*}
We first treat the long time scenario: $\absolu{x -y} \leq K t$.
Using estimate from \cref{p:control_partial_G_lambda_11}, there exists a function $H$ holomorphic \wrt $\lambda$ outside $I := (-\infty, 0)$, such that $\norme{H(\lambda, \cdot , \cdot)}_{L^\infty(\R_x, L^1(\R_y))} < +\infty$ and
\begin{equation*}
\partial_x^j \Glu(x,y) - P_j(\delta_{x=y}) = \sqrt{\lambda}^j e^{-\sqrt{\lambda}\absolu{x-y}} H(\lambda, x, y).
\end{equation*}
Setting $\rho = \kappa \frac{\absolu{x-y}}{t}$ with $\kappa > 0$ to be fixed later. Close to the origin, we use a parabolic contour:
\begin{equation*}
\Lambda_1(\rho) = \Set{\lambda(\xi) = (\rho + i\xi)^2 : \xi\in(-\xi_0, \xi_0)},
\end{equation*}
and obtain with $\tilde{H}(\xi, x, y, t) := \frac{1}{\pi} e^{i(2\rho\xi t - \xi \absolu{x-y})} H_{\lambda(\xi)}(x,y)$ that
\begin{equation*}
\frac{1}{2i\pi} \int_{\Lambda_1} e^{\lambda t} (\partial_x^j \Glu(\cdot, y) - P_j(\delta_{x = y})) \d \lambda = e^{\rho^2 t - \rho\absolu{x-y}} \int_{-\xi_0}^{\xi_0} e^{-\xi^2 t} \tilde{H}(\xi, x, y, t) (\rho + i\xi)^{j+1} \d \xi,
\end{equation*}
Due to $\Lu$ having real coefficients, we have $\Glu[\bar{\lambda}] = \overline{\Glu}$, and the same goes for $H$. Since $\lambda(-\xi) = \overline{\lambda(\xi)}$ we deduce that the term below the integral in the right hand side of the above equation satisfies $z(-\xi) = \overline{z(\xi)}$. As a consequence, we compute
\begin{align*}
\frac{1}{2i\pi} \int_{\Lambda_1} e^{\lambda t} (\partial_x^j \Glu(\cdot, y) - P_j(\delta_{x = y})) \d \lambda = {} & e^{(\kappa^2 - \kappa)\frac{\absolu{x-y}^2}{t}} \int_{-\xi_0}^{\xi_0} e^{-\xi^2 t} \real\left(\tilde{H}(\xi, x, y, t) (\rho + i\xi)^{j+1}\right) \d \xi,\\
= {} & e^{(\kappa^2 - \kappa)\frac{\absolu{x-y}^2}{t}} \int_{-\xi_0}^{\xi_0} e^{-\xi^2 t} \left(\real{\tilde{H}(\xi, x, y, t)} \real{(\rho + i\xi)^{j+1}}\right. \\
& {} \left. - \imag{\tilde{H}(\xi, x, y, t)} \imag{(\rho + i\xi)^{j+1}} \right) \d \xi.
\end{align*}
First develop $(\rho + i\xi)^{j+1} = \sum_{k = 0}^{j+1} {j+1 \choose k} (i\xi)^k \rho^{j + 1 - k}$, and notice that $\rho \leq \kappa K$. Then for $k$ even, we bound 
\begin{equation*}
\rho^{j+1 - k} \leq \kappa^{1+j-k} K^{j - \floor{j/2} - k/2} \left(\frac{\absolu{x -y}}{t}\right)^{1 + \floor{j/2} - k/2},
\end{equation*} 
with $j - \floor{j/2} - k/2 \geq 0$.
Finally, each power of $\xi$ provide a power of $1/\sqrt{t}$ when integrated:
\begin{equation*}
0 \leq \int_{-\xi^0}^{\xi_0} \xi^{k} e^{-\xi^2 t} \d \xi \leq \frac{1}{t^{(k+1)/2}} \int_\R \left(\xi\sqrt{t}\right)^k e^{-\xi^2 t} \sqrt{t} \d \xi \leq \frac{C_k}{t^{(k+1)/2}}.
\end{equation*}
Hence, using that $\real{\tilde{H}}$ is bounded, we compute that 
\begin{equation*}
\Absolu{\int_{-\xi_0}^{\xi_0} e^{-\xi^2 t} \real{\tilde{H}(\xi, x, y, t)} \real{(\rho + i\xi)^{j+1}} \d \xi} \leq C_{j} \frac{1 + \absolu{x-y}^{1 + \floor{j/2}}}{t^{3/2 + \floor{j/2}}}.
\end{equation*}
The terms with $k$ odd, corresponding to the imaginary part of $(\rho + i\xi)^{j+1}$, are bounded similarly:
\begin{equation*}
\rho^{j + 1 - k} \leq \kappa^{j + 1 - k} K^{j - \floor{j/2} - \frac{k-1}{2}} \left(\frac{\absolu{x-y}}{t}\right)^{1 + \floor{j/2} - \frac{k+1}{2}}.
\end{equation*}
Since $\xi \mapsto \imag(\tilde{H}(\xi))$ is odd, it provide an extra power of $\xi$: $\absolu{\imag{\tilde{H}(\xi)}} \leq \xi C$. Hence we compute that
\begin{equation*}
\Absolu{\int_{-\xi_0}^{\xi_0} e^{-\xi^2 t} \imag{\tilde{H}(\xi, x, y, t)} \imag{(\rho + i\xi)^{j+1}} \d \xi} \leq C_{j} \frac{1 + \absolu{x-y}^{1 + \floor{j/2}}}{t^{3/2 + \floor{j/2}}}.
\end{equation*}

In the short time scenario: $\absolu{x-y} \geq K t$, we assume $K$ is large enough so that the parabolic contour $\Lambda_1$ is far enough from the origin, which in turn guarantees that large $\lambda$ estimates from \cref{p:control_G_lambda_large} applies. The computations are then similar to above, with $(\rho + i\xi)^{j+1}$ replaced by 
\begin{equation*}
(\rho + i\xi)^j = \sum_{k = 0}^j {j \choose k} (i\xi)^k \rho^{j-k}.
\end{equation*}
We now use $\rho = \kappa \frac{\absolu{x-y}}{t}$, to obtain that the $k$-th term in the above some contribute as
\begin{equation*}
{j \choose k} (i\xi)^k \rho^{j - k} = C \sqrt{t} \left(\xi \sqrt{t}\right)^{k} \times \frac{1}{t^{\frac{j+1}{2}}} \left(\frac{\absolu{x- y}}{\sqrt{t}}\right)^{j - k},
\end{equation*}
which leads to the claimed pointwise estimate.

We now express the correction terms. None appear for $j = 1$, while for $j = 2$, there is only one coefficient $a_{0,2} = 1$. It is holomorphic on $\lambda\in \C$, thus we can move the integration contour to the left of the imaginary axis, and obtain:
\begin{equation*}
\frac{1}{2i\pi} \int_\Lambda e^{\lambda t} \delta_{y = x} \d \lambda = c e^{-t} \delta_{x = y}.
\end{equation*}

Turning now to the $L^q$ - $L^p$ estimate \eqref{e:linear_estimate_u-1}, we use the extra exponential localization we have gained in $h$. First when $\absolu{x-y} \leq Kt$, we use both that $(1 + \absolu{x - y}) h(x,y) \leq \rho_*(x) \rho_*(y) h(x,y)$ and that $(\frac{\absolu{x-y}}{t})^{\floor{j/2}} \leq C$ to obtain that
\begin{equation*}
\int_{\absolu{x-y} \leq Kt} \frac{\partial_x^j\Gtu(x,y) w(y)}{\rho_*(x)} \d y \leq \frac{\norme{\rho_*^3 w}_{L^\infty(\R)}}{t^{3/2}} \int_{\absolu{x-y} \leq Kt} \frac{h(x,y)}{\rho_*(y)^2} \d y. 
\end{equation*}
The integral on the right hand side is finite. In the other case, we rather use that $K \leq \frac{1}{\sqrt{t}} \frac{\absolu{x-y}}{\sqrt{t}}$, and change variables to gain as much decay as needed:
\begin{align*}
\int_{\absolu{x-y} \geq Kt} \frac{\partial_x^j\Gtu(x,y) w(y)}{\rho_*(x)} \d y & \leq \frac{\norme{\rho_* w}_{L^\infty(\R)}}{t^\frac{j+3}{2}} \int_{\absolu{x-y} \geq Kt} \left(1 + \frac{\absolu{x - y}}{\sqrt{t}}\right)^{j+2} e^{- \kappa_1\frac{\absolu{x-y}^2}{t}} \d y,\\
& \leq \frac{\norme{\rho_*^3 w}_{L^\infty(\R)}}{t^{3/2}} \int_{\R} (1 + z)^{j+2} e^{- \kappa_1z^2} \d z.
\end{align*}
We have shown $(p, q) = (+\infty, +\infty)$. 
Then $(p, q) = (1, 1)$ reads
\begin{align*}
\iint_{\absolu{x-y} \leq K t}\Absolu{\frac{\Gtu(x,y) w(y)}{\rho_*(x)}} \d y \d x & \leq \frac{C}{t} \int_\R \Norme{x\mapsto \frac{1}{\sqrt{t}} e^{-\kappa\frac{\absolu{x-y}^2}{t}}}_{L^1(\R)} \absolu{\rho_*(y) w(y)} \d y, \\
& \leq \frac{C}{t} \norme{\rho_* w}_{L^1(\R)},
\end{align*}
together with $\norme{\rho_* w}_{L^p(\R)} \leq \norme{\rho_*^3 w}_{L^p(\R)}$.
The other term $\iint_{\absolu{x-y}\geq Kt}$ bound similarly and decay as fast as needed.
The case $(p, q) = (+\infty, 1)$ proves similarly. Interpolation between those three cases leads to the $1\leq q \leq p \leq +\infty$ estimate.
It only remains to show \eqref{e:linear_estimate_u-2}. The part which is absolutely continuous with respect to Lebesgue measure leads to the rational decay similarly as above. Then the Dirac delta provide the exponentially decaying term. The polynomial weight $\rho_*(x)^2 \geq 1$ is removed.
\end{proof}

\begin{remark}
\label{r:short_time_estimate}
Using standard parabolic regularity, we obtain for $0 < t\leq 1$ and $1\leq p \leq +\infty$ that:
\begin{equation*}
\Norme{\Gtu \cdot w}_{W^{j,p}(\R)} \leq C \Norme{w}_{W^{j,p}(\R)}.
\end{equation*}
Then from $\inf(1, \frac{1}{t}) \leq \frac{2}{1+t}$, estimates \eqref{e:linear_estimate_u-1}-\eqref{e:linear_estimate_u-2} hold true for $t\geq 0$, $p=q$, and $\frac{1}{t}$ replaced by $\frac{1}{1+t}$.
\end{remark}

\begin{remark}
When $p = +\infty$, inequality \eqref{e:linear_estimate_u-1} is the optimal $t^{-3/2}$ decay.
It does not seem possible to obtain estimates with $p<q$, since $\Cal{G}_t(x,y)$ can not absorb both integral in $x$ and $y$.
Remark that $\rho_*^3$ is arbitrary, and could be replaced by $\rho_* h$ with $\int_0^{+\infty}\frac{\d y}{h(y)} < +\infty$.
\end{remark}

\subsection{Nonlinear dynamic}
\label{ss:non_linear_U}
Here we finally prove \cref{p:V_implies_U}, \ie that $U = \transpose{(u_1, u_2)}$ decays in time provided $V = \transpose{(v_1, v_2)}$ is bounded. Throughout this section, we will assume that there exists $C_\mu>0$ such that for all $0\leq t\leq t_V$,
\begin{equation}
\label{e:hyp_V_bounded}
\Norme{\frac{\omegav}{\rho_*} U(t)}_{L^\infty(\R)} = \norme{V(t)}_{L^\infty(\R)} \leq C_\mu.
\end{equation}
To lighten notations, we may note $C_\mu$ instead of $C C_\mu$ when $C$ is a constant not depending on $C_\mu$.
Remind that $V(t)$ and $\frac{\omegav}{\rho_*} U(t)$ differs only by a change of space variable, so that the first equality is automatically satisfied.

\begin{proposition}
\label{p:asymptotic_stability_u}
For all $1 < p \leq +\infty$, there exists positive constants $C_\symbstab$, $\eta$, $\delta$, such that if $C_\mu \leq \delta$, then the solution $U$ of \eqref{e:system_U} emanating from initial condition $U_0$ -- with $U_0 \rho_*^3 \in L^p(\R)$ -- is defined for $0\leq t \leq t_V$, and satisfies
\begin{equation*}
\label{e:asymptotic_stability-u}
\Norme{\frac{u_1(t, \cdot)}{\rho_*}}_{L^p(\R)} \leq C_\symbstab \frac{\norme{U_0 \rho_*^3}_{L^p(\R)}}{(1+t)^{\frac{3}{2} - \frac{1}{2p}}},
\hspace{4em}
\norme{u_2(t, \cdot)}_{L^p(\R)} \leq C_\symbstab e^{-\frac{\eta}{2} t}  \norme{U_0}_{L^p(\R)}.
\end{equation*}
Furthermore, $C_\symbstab$ is independent of $C_\mu$.
\end{proposition}
\begin{proof}
In this proof, we abbreviate $K_p := \norme{U_0 \rho_*^3}_{L^p(\R)}$. 
We follow the same argument as in \cite[Lemma 3.2]{Beck_Ghazaryan_Sandstede_09}, which is adapted to the nonlinear system after conjugation by an unbounded weight. The key ingredient is to use the nonlinear term to absorb $\omegav$. Then nonlinear terms are seen as linear ones. 

We fix $0 < \tilde{\eta} < \eta$ where $\eta$ is given by the exponential decay of $u_2$ at linear level: see \eqref{e:linear_estimate_v} in \cref{p:temporal_Green_v}. For $1 < p \leq +\infty$, we note $\sigma(p) := \frac{3}{2} - \frac{1}{2p} > 1$.
Then we note
\begin{equation*}
\Theta_1(t) := \sup_{s\in(0, t)} (1+s)^{\sigma(p)} \Norme{\frac{u_1(s)}{\rho_*}}_{L^p(\R)}
\hspace{4em}
\Theta_2(t) := \sup_{s\in(0, t)} e^{\tilde{\eta} s} \Norme{u_2(s)}_{L^p(\R)},
\end{equation*}
together with $\Theta(t) = \max (\Theta_1(t), \Theta_2(t))$, and show that $\Theta$ is bounded in time.
Using \eqref{e:hyp_V_bounded}, we can absorb the unbounded $\omegav$ that comes from the nonlinear terms. We have both
\begin{equation*}
\absolu{\Cal{N}_1(U)} \leq \absolu{3 \alpha q_*} \absolu{\omegav u_1} \absolu{\omegau u_1} + \absolu{\alpha \omegau} \absolu{\omegav u_1}^2 \absolu{\omegau u_1} \leq C_\mu \absolu{\omegau u_1}
\end{equation*}
and 
\begin{equation*}
\label{e:NL2-control_omega}
\Absolu{\Cal{N}_2(U)} \leq C \absolu{\omegav u_1} \absolu{\omegau u_2} + C \absolu{u_2 \omegav}^2 \absolu{\omegau u_2} \leq C_\mu \absolu{\omegau u_2}.
\end{equation*}
In particular, $U\mapsto \Cal{N}(U)$ is globally Lipschitz, hence the solution $U$ is defined in $L^p(\R)$ for times $0 \leq t \leq T$, see \cite{Henry_semilinear_parabolic_equation}.
We can now use the $\omegau$ factor to gain a polynomial weight -- remind that to apply the linear estimates \eqref{e:linear_estimate_u-1}, we need one $\rho_*$ to control the Green kernel, and two extra ones to make the integral \wrt $y$ converge: $\omegau \leq \frac{C}{\rho_*^{4}}$ -- which leads to
\begin{equation*}
\Absolu{\Cal{N}_1(s,y)} \leq \frac{C_\mu}{\rho_*(y)^3(1+s)^{\sigma(p)}} \Absolu{\frac{(1+s)^{\sigma(p)}}{\rho_*(y)} u_1(s,y)},
\end{equation*}
hence
\begin{equation*}
\label{e:NL1-estimate-u}
\Norme{\rho_*^3\Cal{N}_1(s)}_{L^p(\R)} \leq
\frac{C_\mu}{(1+s)^{\sigma(p)}} \Theta_1(t).
\end{equation*}
Similarly we get to
\begin{equation*}
\label{e:NL2-estimate-v}
\Norme{\rho_*^3 \Cal{N}_2(U(s))}_{L^p} \leq C_\mu e^{-\tilde{\eta} s} \Norme{e^{\tilde{\eta} s} u_2(s,\cdot)}_{L^p} \leq C_\mu e^{-\tilde{\eta} s} \Theta_2(t).
\end{equation*}
The Duhamel's formula decomposes into 
\begin{equation*}
\label{e:Duhamel_v}
u_2(t,x) = \int_\R \Gtv(x,y) u_2^0(y) \d y + \int_0^t \int_\R \Gtv[t-s](x,y) \Cal{N}_2(s,y) \d y \d s.
\end{equation*}
\begin{equation*}
\label{e:Duhamel_u}
u_1(t,x) = \int_\R \Gtu u_1^0 + \Gtuv u_2^0 \d y + \int_0^t \int_\R \Gtu[t-s] \Cal{N}_1 + \Gtuv[t-s] \Cal{N}_2 \d y \d s.
\end{equation*}
We use Minkowski inequality as well as the injection $L^1((0,t), L^p(\R)) \subset L^p(\R, L^1(0,t))$ -- see \cref{l:Sobolev_embeding} -- onto Duhamel's formula on $u_2$ to obtain that
\begin{equation*}
\norme{u_2(t)}_{L^p(\R)} \leq \norme{\Gtv \cdot u_2^0}_{L^p} + \int_0^t \norme{\Gtv[t-s] \cdot \Cal{N}(s, \cdot)}_{L^p} \d s.
\end{equation*}
Now using the linear estimate \eqref{e:linear_estimate_v} from \cref{p:temporal_Green_v} and the above nonlinear estimate, we obtain:
\begin{equation*}
\norme{e^{\tilde{\eta} t} u_2(t)}_{L^p(\R)} \leq C \norme{u_2^0}_{L^p(\R)} + C_\mu \Theta_2(t) \int_0^t e^{-(\eta - \tilde{\eta}) (t-s)} \d s,
\end{equation*}
which after integration and taking the supremum on $t \in (0, \tau)$ reads
\begin{equation*}
\Theta_2(\tau) \leq \frac{C \norme{u_2^0}_{L^p(\R)}}{1 - \frac{C_\mu}{\eta - \tilde{\eta}}}, 
\end{equation*}
assuming that the denominator is positive. By imposing $C_\mu \leq \delta \leq \frac{\eta - \tilde{\eta}}{2}$, this condition is fulfilled and we recover $\Theta_2(t) \leq 2 C \norme{u_2^0}_{L^p(\R)}$ for $0 \leq t \leq T$. Remind that $C$ does not depend on $C_\mu$. This is the claimed exponential temporal decay for $u_2$.
Turning now towards  Duhamel's formula for $u_1$, we apply linear estimate \eqref{e:linear_estimate_u-1} with $p=q$ -- see also \cref{r:short_time_estimate} -- together with both above nonlinear estimates, to obtain that:
\begin{align*}
(1+t)^{\sigma(p)}\Norme{\frac{u_1(t,\cdot)}{\rho_*}}_{L^p} \leq{}&  C \norme{U_0 \rho_*^3}_{L^p(\R)} \\
& + \int_0^t \frac{(1+t)^{\sigma(p)}}{(1+t-s)^{\sigma(p)}} \left(\Norme{\rho_*^3\, \Cal{N}_1(s, \cdot)}_{L^p(\R)} + \Norme{\rho_*^3\, \Cal{N}_2(s, \cdot)}_{L^p(\R)} \right) \d s,\\
\leq{}& C K_p + C_\mu \Theta(t) \int_0^t \left(\frac{(1+t)}{(1+t-s)(1+s)}\right)^{\sigma(p)} \d s.
\end{align*}
Standard computations on integral -- see \cref{l:integrale_non_lineaire} and remark that $\sigma(p)>1$ -- lead to
\begin{equation*}
\Theta_1(t) \leq C K_p + C_\mu \Theta(t).
\end{equation*}
Summing the above $\Theta_2(t) \leq 2 C K_p$, we get to $\Theta(t) \leq \frac{C K_p}{1 - C_\mu}$ as soon as $1 -C_\mu>0$. By assuming that $C_\mu \leq \delta \leq \frac{1}{2}$, we get to $\Theta(t) \leq C K_p$, which implies the claimed temporal decay for $u_1$. Remark that $C_\symbstab := C$ does not depends on $C_\mu$.
\end{proof}

\section{Perturbations in partially weighted space are bounded in time}
\label{s:V_bounded}

\subsection{Mode filters}
\label{ss:mode-filters}
Since part of the spectrum of $\Cal{T}^-$ is unstable, the dynamic for $V$ is unstable at linear level. We count on the nonlinear terms to control $V$ for large bounded times $t\leq \frac{T}{\mu}$. To do so, we mostly follow the approach of Guido Schneider, see \cite{Schneider_94_juil}. We separate the \emph{critical} from the \emph{stable} modes in our solution: the first ones grow or are bounded at linear level, while the second ones decay exponentially, uniformly in $\mu$. Let us use a smooth, positive cut-off function $\chi$:
\begin{equation}
\label{e:cut-off}
\chi_\symbc(x) := \begin{cases}
1 & \text{if } x \in I_\symbc := [-\frac{9}{8}, -\frac{7}{8}]\cup [\frac{7}{8}, \frac{9}{8}],\\
0 & \text{if } x \notin I_\symbc + B(0, \frac{1}{8}),
\end{cases}
\end{equation}
with $B(x, r)$ the ball centered at $x$ with radius $r$. We then work in Fourier space: the matrix $\hat{\Cal{T}}^-(\xi)$ has two eigenvalues
\begin{equation*} 
\lambda_\symbc(\xi) = -(1-\xi^2) + \mu, 
\hspace{4em} 
\lambda_\symbs(\xi) = -d\xi^2 -2\alpha,
\end{equation*}
with associated eigenvectors $\varrho_i$, for $i\in\Set{\symbs, \symbc}$. We note $\varPi_i(\xi) \hat{V} = \pscal{\hat{V}, \varrho_i^*(\xi)} \varrho_i(\xi)$ the respective parallel projections onto each of the eigenspaces, and we separate critical from stable frequencies:
\begin{equation*}
\vPic V = \Cal{F}^{-1} \left(\xi\mapsto \chi_\symbc(\xi) \Pscal{\hat{V}(\xi), \varrho_\symbc^*(\xi)} \varrho_\symbc(\xi)\right), 
\hspace{4em}
\vPis V = V - \vPic V
\end{equation*}

Remark that $\chi_\symbc ^2 \neq \chi_\symbc$, so that neither $\vPic$ or $\vPis$ are projections. However, using $I_\symbc^h := [-\frac{5}{4}, -\frac{3}{4}]\cup [\frac{3}{4}, \frac{5}{4}]$ and $I_s^h := [-\frac{17}{16}, -\frac{15}{16}]\cup [\frac{15}{16}, \frac{17}{16}]$, we define  $\chi_\symbc^h$ and $\chi_s^h$ with respective support $I_\symbc^h + B(0, \frac{1}{4})$ and $I_s^h + B(0, \frac{1}{16})$, in a similar way as \eqref{e:cut-off}. Then with  
\begin{equation*}
\vPich V := \Cal{F}^{-1} (\chi_\symbc^h  \varPi_1 \hat{V}), 
\hspace{4em}
\vPish V := \Cal{F}^{-1} ((1-\chi_\symbs^h)  \varPi_1 \hat{V}) + \Cal{F}^{-1}(\varPi_2 \hat{V}), 
\end{equation*}
we have $\vPih_i \varPi_i = \varPi_i$ for $i\in \lbrace \symbc, \symbs\rbrace$. Ultimately, we will use a scalar \guillemet{projection} onto half of the critical Fourier modes, to select only frequencies close to $\xi = 1$:
\begin{equation}
\label{e:def_pi_1}
\pi_1^h V = \Cal{F}^{-1} \left(\xi \mapsto \chi_\symbc^h(\xi) \mathbf{1}_{\xi > 0} \frac{1}{\pscal{\rho_\symbc(1), \rho_\symbc^*(\xi)}}\pscal{\hat{V}(\xi), \varrho_\symbc^*(\xi)}\right).
\end{equation}
Such decompositions are well-behaved with Fourier transform. However, we need to measure objects with $L^\infty(\R)$ norms, since the pattern we want to study behave as the solution of a bistable equation. To combine those two constraint, we use the so-called uniformly-localized space, that were introduce by Guido Schneider, see for example \cite{Schneider_94_dec}. We first define a weight
\begin{equation*}
\rho_\ul(x) := \frac{1}{1+x^2}.
\end{equation*}
Then, we say that $u\in L^2_\ul(\R)$ if 
\begin{equation*}
\norme{u}_{L^2_\ul(\R)} := \sup_{y\in \R} \norme{\rho_\ul(\cdot - y) u}_{L^2(\R)} < +\infty.
\end{equation*}
Similarly, we define $H^s_\ul(\R)$ and its norm by $\norme{u}_{H^s_\ul(\R)} :=  \sup_{y\in \R} \norme{\rho_\ul(\cdot - y) u}_{H^s(\R)}$ when $s\geq 0$.
Then, we use the following injections to link with Sobolev spaces, that were used to estimate Green's kernel.
\begin{lemma}
\label{l:injection_H1_ul}
Let $s\in\R$ with $s\geq 1$. Then the following injections are continuous: $W^{s, \infty}(\R) \subset H^s_\ul(\R) \subset L^\infty(\R) \subset L^2_\ul(\R)$.
\end{lemma}
\begin{proof}
Let first recall the definition of $\norme{.}_{H^s_\ul(\R)}$ that we use.
Let $\rho_{\ul} (x) := \frac{1}{1+x^2}$, then
\begin{equation*}
\norme{u}_{H^s_\ul(\R)} := \sup_{y\in\R} \norme{u(\cdot - y) \rho_\ul}_{H^s(\R)}.
\end{equation*}
From one hand the injection $H^s(\R) \subset L^{\infty}(\R)$ leads to 
\begin{equation*}
\norme{u}_{H^s_\ul(\R)} \geq \sup_{y\in\R} \norme{u(\cdot - y ) \rho_\ul}_{L^\infty(-1,1)} \geq \frac{1}{2} \sup_{y\in\R} \norme{u(\cdot - y)}_{L^\infty(-1,1)} = \frac{1}{2} \norme{u}_{L^\infty(\R)},
\end{equation*}
It corresponds to the second injection.
From the other hand, $\norme{u \rho_\ul}_{H^s(\R)} \leq \norme{u}_{W^{s,\infty}(\R)}\norme{\rho_\ul}_{H^s(\R)}$ leads to the first and third injection.
\end{proof}

As announced, we will need to estimate operators in Fourier space. We will use the following Lemma, see \cite[Section 3.1]{Schneider-94-may} and \cite[Lemma 5]{Schneider_94_juil} for proofs.
\begin{lemma}
\label{l:estimate_mode-filters}
Let $d$, $n$ be positive integers, and $M$ an operator that acts in Fourier space as a point-wise linear application: $\hat{M}(\xi) \in \mathfrak{L}(\R^d, \R^n)$ with $Mu := \Cal{F}^{-1}(\xi \mapsto \hat{M}(\xi) \hat{u}(\xi))$. Then for $q, s \geq 0$, 
\begin{equation*}
\norme{Mu}_{H^s_\ul(\R)^d} \leq C(s,q) \norme{\xi \mapsto \pscal{\xi}^\frac{s-q}{2} \hat{M}(\xi)}_{\Cal{C}^2_b(\R, \mathfrak{L}(\R^d, \R^n))} \norme{u}_{H^q_\ul(\R)^n}, 
\end{equation*}
with a positive constant $C(s,q)$ independent of $u$ and $M$. Furthermore if $\alpha$ and $\xi_0$ are reals, then 
\begin{equation*}
\norme{x\mapsto e^{i\xi_0 x} Mu (\alpha x)}_{H^s_\ul(\R)^d} \leq C \norme{x\mapsto e^{i\xi_0 x} u(\alpha x)}_{H^q_\ul(\R)^n}, 
\end{equation*}
where the above constant satisfies
\begin{equation*}
C\leq C(s,q) \Norme{\xi \mapsto \pscal{\xi}^\frac{s-q}{2} \hat{M}\left(\frac{\xi}{\alpha} - \xi_0\right)}_{\Cal{C}^2_b(\R, \mathfrak{L}(\R^d, \R^n))}.
\end{equation*}
\end{lemma}

\subsection{Linear dynamic}
\begin{lemma}
\label{l:linear_behavior_V}
Let $T$ and $\mu_0$ be positive constants. Then there exists $C>0$ and $\kappa >0$ such that for all $0<\mu < \mu_0$, the following holds. For $s\geq 0$, and $0\leq t \leq \frac{T}{\mu}$,
\begin{equation*}
\norme{e^{t\Cal{T}^-} \vPich V}_{H^s_\ul(\R)} \leq C \norme{V}_{H^s_\ul(\R)},
\end{equation*}
while for any $t\geq 0$,
\begin{equation*}
\norme{e^{t\Cal{T}^-} \vPish V}_{H^s_\ul(\R)} \leq C e^{-\kappa t} \norme{V}_{H^s_\ul(\R)}.
\end{equation*}
Both estimates still holds when $\varPi_i^h$ are replaced by $\varPi_i$.
\end{lemma}
\begin{proof}
Since $\Cal{T}^-$ has constant coefficients -- see \eqref{e:def_T_minus} -- it acts in Fourier space through multiplication. Hence we rely on multiplier theory, see \cite{Mielke_Schneider_95}. Since at fixed Fourier parameter $\xi\in\R$, the eigenvalues of matrix $\hat{\Cal{T}}^-(\xi)$ satisfy $\real{\lambda_i(\xi)} \leq \mu$, we obtain using \cref{l:estimate_mode-filters} that for $t\leq \frac{T}{\mu}$:
\begin{equation*}
\Norme{e^{t\Cal{T}^-}\vPich V}_{H^s_\ul(\R)} \leq  \Norme{\exp(t\hat{\Cal{T}}^-) \hat{\vPich}}_{\Cal{C}^2(\R, \Cal{M}_2(\R))} \norme{V}_{H^s_\ul(\R)} \leq C e^{2t \mu} \norme{V}_{H^s_\ul(\R)},
\end{equation*}
with $\Cal{M}_2(\R)$ the set of $2\times 2$ matrices. 
The case of $e^{t\Cal{T}^-} \vPish$ adapts easily: for $t\geq 0$ and $\xi$ in the support of $\chi_\symbs$, the eigenvalues of $\hat{\Cal{T}}^-(\xi)$ satisfy $\real(\lambda_j(\xi)) \leq -2\kappa$.
To see that $\varPi_i^h$ may be replaced by $\varPi_i$, simply use that $e^{t\Cal{T}^-} \varPi_i = e^{t\Cal{T}^-} \vPih_i \varPi_i$ and that $\norme{\varPi_i V}_{H^s_\ul(\R)} \leq \norme{V}_{H^s_\ul(\R)}$ to obtain the result.
\end{proof}

In the following, we can combine \cref{l:injection_H1_ul,l:linear_behavior_V} to glue the mode filters techniques with our usual Sobolev spaces -- to the cost of one derivative -- \emph{via} estimates of the form:
\begin{equation*}
\norme{V(t)}_{L^\infty(\R)} \leq C \norme{V(t)}_{H^1_\ul(\R)} \leq C \vartheta(t) \norme{V_0}_{H^1_\ul(\R)} \leq C \vartheta(t) \norme{V_0}_{W^{1, \infty}(\R)}.
\end{equation*}

\subsection{Nonlinear dynamic: shadowing the global attractor of the Ginzburg-Landau equation}
Following \cite{Ghazaryan_Sandstede_07}, we drive the perturbation using the dynamic at $-\infty$, to the cost of an extra source term $\Cal{S}$. We first show that this term defined by \eqref{e:rest-term} is sufficiently localized in space so that we can extract a $\norme{\frac{V}{\omegav(\cdot - c_*t)}}_{L^\infty} = \norme{\frac{U}{\rho_*}}_{L^\infty}$ from it. 
In the rest of this section, we assume that there exists $C_1$ and $t_U$ positive constants such that for all $0\leq t \leq t_U$, 
\begin{equation}
\label{e:hyp_decay_U}
\Norme{\frac{U(t)}{\rho_*}}_{L^\infty(\R)} \leq C_1 \frac{\norme{U_0}_X}{(1+t)^{3/2}}.
\end{equation}
As above, we will note $C_1$ instead of $C C_1$ when $C$ is a constant not depending on $C_1$.
All arguments below rely on the fact that $U$ does not blow up in finite time. Hence in the rest of this section, we always restrict -- even when it is not clearly stated -- to times $0\leq t\leq t_U$. Furthermore, we recall the notation $\mu = \varepsilon^2$, with $\varepsilon > 0$.

We first obtain decay of derivatives of $U$, and then control the source term. Recall that
\begin{equation*}
\norme{U}_X := \norme{U}_{W^{2, \infty}(\R) \times W^{4, \infty}(\R)} + \norme{\rho_*^3 U}_{L^\infty(\R)}.
\end{equation*}
\begin{proposition}
\label{p:asymptotic_stability-partial_u}
There exist a positive constant $C$ such that for $0\leq t \leq t_U$:
\begin{equation*}
\Norme{\frac{U(t)}{\rho_*}}_{W^{2,\infty}(\R) \times W^{4,\infty}(\R)} \leq C \frac{\norme{U_0}_X}{(1+t)^\frac{3}{2}}.
\end{equation*}

\end{proposition}
\begin{proof}
First write the Duhamel formula for $u_1$, and then differentiate it:
\begin{equation*}
\label{e:Duhamel_partial_u}
\partial_x^j u_1(t,x) = \int_\R \partial_x^j\Gtu u_1^0 + \partial_x^j \Gtuv u_2^0 \d y + \int_0^t \int_\R \partial_x^j \Gtu[t-s] \Cal{N}_1 + \partial_x^j \Gtuv[t-s] \Cal{N}_2 \d y \d s.
\end{equation*}
We can then bound the right hand side using linear decay from \cref{p:temporal_Green_partial_u} -- see also \cref{r:short_time_estimate} -- and the decay \eqref{e:hyp_decay_U} of $U(t)$.
The localized weight $\omegau$ in nonlinear terms allows to gain as many powers of $\rho_*(y)$ as needed. 
We use the same approach for $\partial_x^k u_2$, the linear estimate $\norme{\Gtv u_2}_{W^{k,\infty}(\R)} \leq C e^{-\eta t} \norme{u_2}_{W^{k,\infty}(\R)}$ comes from standard parabolic regularity: $\Lv$ is sectorial with spectral gap.
\end{proof}

\begin{proposition}
\label{p:decay_rest_term}
There exists positive constants $\delta$ and $C$ such that for all $0\leq t\leq t_U$,
\begin{equation*}
\norme{\Cal{S}(\cdot-c_*t, V(t, \cdot))}_{H^1_\ul(\R)} \leq C \Norme{\frac{V(t, \cdot)}{\omegav(\cdot - c_* t)}}_{W^{2,\infty}(\R) \times W^{4, \infty}(\R)} \leq C \frac{\norme{U_0}_{X}}{(1+t)^{3/2}}.
\end{equation*}
\end{proposition}
\begin{proof}
From \cref{l:injection_H1_ul}, it is enough to bound $\norme{\Cal{S}(\tilde{\cdot}, V(t))}_{W^{1, \infty}(\R)}$.
For $\tilde{x} \leq -1$, we have $\varpi(\tilde{x}) = 1$, hence
\begin{equation*}
\Cal{S}(\tilde{x}, V) = (1-q_*(\tilde{x}))
\begin{pmatrix}
3\alpha (1+q_*(\tilde{x})) v_1 \\
- \gamma v_2
\end{pmatrix}.
\end{equation*}
We show that $\norme{\omegav (1-q_*)}_{W^{1,\infty}(-\infty, 0)}$ is finite. The equilibrium point $(q, q') = (1,0)$ for the front equation is a saddle, with one positive eigenvalue $\kappa = (-1+\sqrt{3})\sqrt{\frac{\alpha}{d}}$. Basic ODE dynamic then ensures that
\begin{equation*}
\sup_{x\leq 0}\absolu{q_*(x) - 1} e^{-\kappa x} < +\infty, 
\hspace{4em}
\sup_{x\leq 0}\absolu{q_*'(x)} e^{-\kappa x} < +\infty.
\end{equation*}
For $\tilde{x}\leq -1$ and $\absolu{\theta}$ small enough, $\omegav(\tilde{x}) = e^{\theta \tilde{x}} \leq e^{-\kappa \tilde{x}}$, which leads to 
\begin{equation*}
\label{e:decay_source_term_1}
\Norme{\Cal{S}(\cdot-c_*t, V(t, \cdot))}_{W^{1,\infty}(-\infty, 0)} \leq C \Norme{\frac{V(t)}{\omegav(\cdot - c_*t)}}_{W^{1,\infty}(-\infty, 0)}.
\end{equation*}
Now for $\tilde{x}\geq 1$, recall the expression of the source term \eqref{e:rest-term}. For linear terms, we use the commutator to gain one derivative: the operator 
\begin{equation*}
\varpi^{-1} (\Cal{A} -c_* \partial_x) \varpi - (\Cal{A} -c_* \partial_x) = \varpi^{-1} [\Cal{A} -c_* \partial_x, \varpi]
\end{equation*}
exhibit at most third order derivatives, so that it maps $W^{2, \infty}(\R) \times W^{4, \infty}(\R)$ onto $W^{1, \infty}(\R) \times W^{1, \infty}(\R)$. Furthermore $\norme{\varpi^{-1}\partial_x \varpi}_{W^{1,\infty}(0, +\infty)} < \infty$. Altogether, it leads to 
\begin{equation*}
\norme{\Cal{S}(\tilde{\cdot}, V)}_{W^{1, \infty}(0, +\infty)} \leq C \norme{V}_{W^{2,\infty} \times W^{4, \infty}(0, +\infty)} + C \norme{V}_{W^{1,\infty}(0, +\infty)}^2 + C \norme{V}_{W^{1,\infty}(0, +\infty)}^3.
\end{equation*}
Remark that $\omegau(\tilde{x}) = 1$, hence the first claimed estimate is shown. Then $\frac{V(t,x)}{\omegav(\tilde{x})} = \frac{U(t, \tilde{x})}{\rho_*(\tilde{x})}$, and
\cref{p:asymptotic_stability-partial_u} ensures the second claimed estimate. The proof is complete.
\end{proof}

\begin{remark}
\label{r:V_continuous_at_0}
Since $t\mapsto (1+t)^{-3/2}$ is integrable at $0$, the above control of the source term implies that solutions to \eqref{e:system_V} with initial condition at $t = 0$ are defined and continuous on an open interval. 
\end{remark}
We now state nonlinear estimate, which relies on mode filters, see \cref{ss:mode-filters}.
\begin{lemma}[Non-linear estimates]
\label{l:non_linear_estimates-V}
There exists $C>0$ such that for all $V \in H^1_\ul(\R)$, 
\begin{equation*}
\norme{\vPis\Cal{Q}^-(\varepsilon \vPic V + \varepsilon^2 \vPis V)}_{H^1_\ul(\R)} \leq C \varepsilon^2 (\norme{\vPic V}_{H^1_\ul(\R)} + \norme{\vPis V}_{H^1_\ul(\R)})^2,
\end{equation*}
while
\begin{equation*}
\norme{\vPic\Cal{Q}^-(\varepsilon \vPic V + \varepsilon^2 \vPis V)}_{H^1_\ul(\R)} \leq C \varepsilon^3 (\norme{\vPic V}_{H^1_\ul(\R)} + \norme{\vPis V}_{H^1_\ul(\R)})^2.
\end{equation*}
\end{lemma}
\begin{proof}
The first estimate is immediate. The second one comes from \cite{Schneider_94_dec}, and proves as follows. In Fourier space, multiplication becomes convolution. Since $\Rm{Supp}(\chi_\symbc) = [-5/4, -3/4]\cup [3/4, 5/4]$ and $\Rm{Supp}(\chi_\symbc * \chi_\symbc) = [-5/2, -3/2]\cup [-1/2, 1/2] \cup [3/2, 5/2]$ do not intersect, we deduce that 
\begin{equation*}
\vPic(\vPic V_1 \times \vPic V_2) = \Cal{F}^{-1}(\chi_\symbc \times (\chi_\symbc \hat{V}_1 * \chi_\symbc \hat{V}_2)) = 0.
\end{equation*}
Hence lowest order quadratic terms vanish when $\vPic$ is applied, leaving $\varepsilon^3$ terms at leading order.
\end{proof}

We now begin the proof of \cref{p:U_implies_V} by decomposing $\Cal{T}^-$ in Fourier space. This allows to show that after a time $T_\symbatt/\varepsilon^2$, the perturbation $V$ is at leading order a critical oscillating mode.  
\begin{lemma}[Attractivity]
\label{l:attractivity}
Let $T_\symbatt>0$ be fixed. There exists a constant $C>0$ -- depending on $C_1$ -- and a positive constant $\delta$ such that for all $V_0$ satisfying
\begin{equation}
\label{e:def_K1}
K_1 := \frac{1}{\varepsilon}\left(\norme{V_0}_{H^1_\ul(\R)} + \norme{U_0}_X \right) \leq \delta,
\end{equation}
the solution $V$ to \eqref{e:system_V} with initial condition $V_0$ exists for all time $0\leq t \leq \frac{T_\symbatt}{\varepsilon^2}$, and decomposes as $V = \Vc + \Vs$, with $\vPih_i V_i = V_i$. When  $0\leq t \leq \frac{T_\symbatt}{\varepsilon}$, it satisfies
\begin{equation*}
\norme{V_i(t)}_{H^1_\ul(\R)} \leq C K_1 \varepsilon, 
\hspace{4em}
i\in \Set{\symbc, \symbs},
\end{equation*}
while for times $\frac{T_\symbatt}{\varepsilon^{2/3}} \leq t \leq \frac{T_\symbatt}{\varepsilon^2}$, we have
\begin{equation*}
\norme{\Vc(t)}_{H^1_\ul(\R)} \leq C K_1 \varepsilon,
\hspace{4em}
\norme{\Vs(t)}_{H^1_\ul(\R)} \leq C K_1 \varepsilon^2.
\end{equation*}
Finally, the initial condition $A_0$ for \eqref{e:GL} is bounded uniformly in $\varepsilon$: $\Norme{A_0}_{H^1_\ul(\R)} \leq C K_1$,
with
\begin{equation}
\label{e:def_A0}
A_0(X) := \frac{1}{\varepsilon} e^{-i\frac{X}{\varepsilon}} 
\pi_1^h \Vc\left(\frac{T_\symbatt}{\varepsilon^2}, \frac{X}{\varepsilon}\right).
\end{equation}
\end{lemma}
\begin{proof}
Since $\Cal{T}^-$ is sectorial, $V \mapsto \Cal{Q^-}(V)$ is locally Lipschitz and $t\mapsto \Cal{S}(\tilde{x}, V(t, x))$ is integrable at $0$, solution $V$ to \eqref{e:system_V} with initial condition $V_0$ is uniquely defined as long as it does not blow up. Hence the estimates that follow ensure that $V$ is defined up to time $T_\symbatt/\varepsilon^2$.

To construct $V$, it is enough to solve the following system, with initial condition $V_i(0) = \varPi_i V(0)$:
\begin{equation}
\label{e:nonlinear_system_V}
\begin{cases}
\partial_t \Vc = \Cal{T}^- \Vc + \vPic \Cal{Q}^-(\Vc + \Vs) + \vPic \Cal{S}(\tilde{x}, \Vc+\Vs), \\
\partial_t \Vs = \Cal{T}^- \Vs + \vPis \Cal{Q}^-(\Vc + \Vs) + \vPis \Cal{S}(\tilde{x}, \Vc+\Vs),
\end{cases}
\end{equation}
and then set $V = \Vc + \Vs$. A solution $(\Vc, \Vs)$ of the decoupled system \eqref{e:nonlinear_system_V} is not guaranteed to write as $V_i = \varPi_i V$, however the critical-stable separation still holds, so that \cref{l:linear_behavior_V} applies:
\begin{equation*}
\norme{e^{t\Cal{T}^-} \Vs}_{H^1_\ul(\R)} = 
\norme{e^{t\Cal{T}^-} \vPish \Vs}_{H^1_\ul(\R)} \leq e^{-\kappa t} \norme{\Vs}_{H^1_\ul(\R)}.
\end{equation*} 
We introduce the local notations $W_i := \frac{V_i}{\varepsilon}$ and $\Theta_i(t) := \sup_{0\leq \tau \leq t}\norme{W_i(\tau)}_{H^1_\ul(\R)}$. Using Duhamel formula and decay of $\Cal{S}$, we see that
\begin{align*}
\norme{\Wc(t)}_{H^1_\ul(\R)} & \leq C \norme{\Wc(0)}_{H^1_\ul(\R)} + C \varepsilon \int_0^t (\Thetac(t) + \Thetas(t))^2 \d\tau + C_1 \frac{\norme{U_0}_X}{\varepsilon} \int_0^t \frac{\d \tau}{(1+ \tau)^{3/2}},\\
& \leq C_1 K_1 + C \varepsilon t (\Thetac(t) + \Thetas(t))^2 .
\end{align*}
In a similar way, standard integral computations -- see \cref{l:integrale_non_lineaire} -- ensures that
\begin{align*}
\norme{\Ws(t)}_{H^1_\ul(\R)} \leq {} & C e^{-\kappa t}\norme{\Ws(0)}_{H^1_\ul(\R)} + C \varepsilon (\Thetac(t) + \Thetas(t))^2 \int_0^t e^{-\kappa(t-\tau)} \d\tau \\ 
& + C_1 \frac{\norme{U_0}_X}{\varepsilon} \int_0^t \frac{e^{-\kappa(t-\tau)}\d \tau}{(1+ \tau)^{3/2}},\\
\leq {} & C_1 \frac{K_1}{(1+t)^{3/2}} + C \varepsilon (\Thetac(t) + \Thetas(t))^2.
\end{align*}
For $0 \leq t \leq \frac{T_\symbatt}{\varepsilon}$, we take the sup in the two above equations to obtain -- with $\Theta := \max(\Thetac, \Thetas)$ -- that
$\Theta(t) \leq C_1 K_1 + C \Theta(t)^2$. Applying a standard nonlinear argument -- see \cref{l:cloture_non_lineaire} -- with $K_1\leq \delta$ small enough, we recover $\Theta(T_\symbatt/\varepsilon) \leq C_1 K_1$, which is the first claimed estimate.
Now for $T_\symbatt/\varepsilon^{2/3} \leq t \leq T_\symbatt/\varepsilon$, we have $\frac{1}{(1+t)^{3/2}} \leq C \varepsilon$, hence $\norme{\Ws(t)}_{H^1_\ul(\R)} \leq C_1 K_1\varepsilon$. This little improvement will allow us to propagate estimates until time $T_\symbatt/\varepsilon^2$. For $t\geq 0$, we now use the new local notation 
\begin{equation*}
V(t + T_\symbatt/\varepsilon^{2/3}) =: \varepsilon \Wc(t) + \varepsilon^2 \Ws(t),
\end{equation*}
together with $\Theta_i(t) := \sup_{0\leq \tau \leq t}\norme{W_i(\tau)}_{H^1_\ul(\R)}$. Remark that from the above, we know that 
\begin{equation*}
\norme{\Wc(0)}_{H^1_\ul(\R)} + \norme{\Ws(0)}_{H^1_\ul(\R)} \leq C_1 K_1.
\end{equation*} 
Then nonlinear estimate from \cref{l:non_linear_estimates-V} ensures that for $t\leq T_\symbatt/\varepsilon^2$, 
\begin{align*}
\norme{\Wc(t)}_{H^1_\ul(\R)} \leq {} & C \norme{\Wc(0)}_{H^1_\ul(\R)} + C \varepsilon^2 t \, (\Thetac(t) + \Thetas(t))^2 + C_1 \frac{\norme{U_0}_X}{\varepsilon} \int_0^t \frac{\d \tau}{(1+ T_\symbatt/\varepsilon^{2/3} + \tau)^{3/2}},\\
\leq {} & C_1 K_1 + C(\Thetac(t) + \Thetas(t))^2 .
\end{align*}
Since $(1 + T_\symbatt/\varepsilon^{2/3} + \tau)^{-3/2} \leq \varepsilon$ when $\tau \geq 0$, we deduce as above that
\begin{align*}
\norme{\Ws(t)}_{H^1_\ul(\R)} \leq {} & C e^{-\kappa t}\norme{\Ws(0)}_{H^1_\ul(\R)} + C \int_0^t e^{-\kappa(t-\tau)} \d\tau (\Thetac(t) + \Thetas(t))^2 \\
& + C_1\frac{\norme{U_0}_X}{\varepsilon^2} \int_0^t \frac{e^{-\kappa(t-\tau)}\d \tau}{(1 + T_\symbatt/\varepsilon^{2/3} + \tau)^{3/2}},\\
\leq {} & C_1 K_1 + C (\Thetac(t) + \Thetas(t))^2.
\end{align*}
Taking the sup for $0 \leq t \leq T_\symbatt/\varepsilon^2$, we obtain as previously that $\Theta(T_\symbatt/\varepsilon^2) \leq C_1 K_1$ for $K_1$ small enough. This is the second claimed estimate.
To bound $A_0$, we use the scaled estimate in \cref{l:estimate_mode-filters}
\begin{equation*}
\Norme{X \mapsto e^{-i\frac{X}{\varepsilon}} M u\left(\frac{X}{\varepsilon}\right)}_{H^1_\ul(\R)} \leq \norme{\xi \mapsto \pscal{\xi}^{1/2} \hat{M}(\varepsilon \xi + 1)}_{\Cal{C}^2_b(\R, \mathfrak{L}(\R^2, \R))} \Norme{X \mapsto e^{-i \frac{X}{\varepsilon}}u\left(\frac{X}{\varepsilon}\right)}_{L^2_\ul(\R)},
\end{equation*}
with $M = \pi_1^h$ and $u = \Vc(\frac{T_\symbapr}{\varepsilon^2})$. To estimate $\hat{M}(\varepsilon \xi + 1)$, we can restrict to the $\Cal{C}^0_b$ norm, since derivative gains us power of $\varepsilon$. Recall that
\begin{equation*}
\hat{M}(\xi) \hat{u} = \chi_\symbc^h(\xi)\pscal{\varrho_\symbc(1), \varrho_\symbc^*(\xi)}^{-1} \pscal{\hat{u}, \varrho_\symbc^*(\xi)}
\end{equation*}
so that Cauchy-Schwartz leads to $\norme{M(\xi)}_{\mathfrak{L}(\R^2, \R^2)} \leq \absolu{\chi_\symbc^h(\xi) \varrho_\symbc^*(\xi) \pscal{\varrho_\symbc(1), \varrho_\symbc^*(\xi)}^{-1}} \leq C \absolu{\chi_\symbc^h(\xi)}$. Here we assume that the support of $\chi_\symbc$ is so small that $\pscal{\varrho_\symbc(1), \varrho_\symbc^*(\xi)}$ does not vanish by continuity. We emphasize that this support is still independent of $\mu$, so that the separation of frequencies comes with spectral gap.
Using the support of $\chi_\symbc^h$, we obtain 
\begin{equation*}
\norme{\xi \mapsto \pscal{\xi}^{1/2} \hat{M}(\varepsilon \xi + 1)}_{\Cal{C}^2_b(\R, \mathfrak{L}(\R^2, \R))}\leq \frac{C}{\sqrt{\varepsilon}}.
\end{equation*}
Now rescaling the $L^2_\ul(\R)$ norm, we get 
\begin{equation*}
\norme{X \mapsto e^{-i\frac{X}{\varepsilon}} u(X/\varepsilon)}_{L^2_\ul(\R)} \leq \sqrt{\varepsilon} \norme{e^{-ix} u}_{L^2_\ul(\R)} \leq \sqrt{\varepsilon} \norme{u}_{L^2_\ul(\R)}.
\end{equation*}
Hence we have shown that $\norme{A_0}_{H^1_\ul(\R)} \leq C \varepsilon^{-1} \norme{\Vc(\varepsilon^{-2} T_\symbapr)}_{L^2_\ul(\R)} \leq C_1K_1$.

\end{proof}
Now that $V$ is at leading order a critical oscillating mode, we can approximate it by $\psi(\varepsilon, A) := \varepsilon \psi_\symbc(\varepsilon, A) + \varepsilon^2 \psi_\symbs(\varepsilon, A)$, with $A$ the solution to the \gl equation with initial condition $A_0$, see \eqref{e:def_A0}, with the critical part of $\psi$ being:  
\begin{equation}
\label{e:def_psi_c}
\psi_\symbc(\varepsilon, A)(t, x) = (e^{ix} A(\varepsilon^2 t, \varepsilon x) + \cc) \varrho_\symbc, 
\end{equation}
and the stable part:
\begin{equation*}
\label{e:def_psi_s}
\psi_\symbs(\varepsilon, A)(t, x) = \absolu{A(\varepsilon^2 t, \varepsilon x)}^2 \varrho_0 + e^{ix} \partial_X A(\varepsilon^2 t, \varepsilon x) \varrho_1 + e^{2ix} A(\varepsilon^2 t, \varepsilon x)^2 \varrho_2 + \cc.
\end{equation*}
All $\varrho_i$ and $\varrho_\symbc$ are two-dimensional vectors with explicit expressions, see \cref{ss:derive-GL}. The choice of $\psi_\symbs$ is made clearer therein, it ensures that the error of approximation $R = V - \psi(\varepsilon, A)$ stays small when time evolves.
\begin{lemma}[Global attractor for Ginzburg-Landau]
\label{l:GL-atractor}
Let $T, X \mapsto A(T, X)$ be a solution of the Ginzburg-Landau equation \eqref{e:Ginzburg_Landau}. Then there exists a constant $C_\gl>0$ such that 
\begin{equation*}
\norme{A(T)}_{W^{1, \infty}(\R)} \leq C_\gl + e^{-T/2} \norme{A(0)}_{W^{1,\infty}(\R)}.
\end{equation*}
\end{lemma}
\begin{proof}
Since the coefficients of \gl are real, the dynamic is close to that of a bistable equation.
We follow \cite[Theorem 7]{Schneider_94_juil}. From computations made in \cite[Theorem 4.1]{Collet_Eckmann_90}, we recover that 
\begin{equation*}
\norme{A(T)h}_{L^2(\R)} \leq C_\gl + e^{-2T} \norme{A(0)h}_{L^2(\R)}
\end{equation*}
where $h_a(X) = \frac{1}{1+(X-a)^2}$. The same estimate holds for $\partial_x A$. It leads to 
\begin{equation*}
\absolu{A(T,X)} \leq \frac{1}{h(X)^2} \left(C_\gl + e^{-T} \norme{A(0)h}_{H^1(\R)}\right),
\end{equation*}
which after taking the supremum over $a \in \R$ writes $\norme{A(T)}_{L^\infty(\R)} \leq C_\gl + e^{-T} \norme{A(0)}_{H^1_\ul(\R)}$.
The bound on $\partial_X A(T)$ is then obtained from standard parabolic regularity: we differentiate the Duhamel formula and use the above bound of $A(T)$. To handle the term that is constant in time, which appears in the non-linear term, we simply deteriorate the exponential decay.
\end{proof}
We now show that when $t\leq \varepsilon^{-7/4}T_\symbapr$, the error $R$ stays small. We then apply this argument again to reach any finite time. For similar problems, this approximation property is usually shown on a time scale $\varepsilon^{-2}T_\symbapr$, which is optimal due to the linear growth of critical modes $\norme{e^{t\Cal{T}^-} \vPic V_0}_{H^1_\ul(\R)} \geq C e^{\varepsilon^2 t} \norme{\vPic V_0}_{H^1_\ul(\R)}$. In our case, the slowly decaying inhomogeneous term $\Cal{S}$ prevents us to go beyond $\varepsilon^{-7/4}T_\symbapr$.

\begin{lemma}[Approximation]
\label{l:approximation} For $T_\symbapr>0$ fixed, there exists a positive constant $C_\symbapr$ such that the following holds. Let $t_0 \geq \frac{T_\symbatt}{\varepsilon^2}$, $A_0\in H^1_\ul(\R)$ and a solution $V$ of \eqref{e:system_V} such that 
\begin{equation*}
\norme{R(t_0)}_{H^1_\ul(\R)} = \norme{V(t_0) - \psi(\varepsilon, A_0)}_{H^1_\ul(\R)} \leq \varepsilon^{5/4}.
\end{equation*}
Note $A$ the solution of \eqref{e:GL} with initial condition $A_0$ at $t = t_0$, and $R(t) = V(t) - \psi(\varepsilon, A(t))$. Then for $t \geq t_0$, the error term decomposes as $R(t) = \Rc(t) + \Rs(t)$, with $\vPih_i R_i(t) = R_i(t)$. Furthermore, for $0\leq t-t_0 \leq \frac{T_\symbapr}{\varepsilon^{1/4}}$,
\begin{equation*}
\norme{R_i(t)}_{H^1_\ul(\R)} \leq C_\symbapr \varepsilon^{5/4},
\hspace{4em}
i\in\Set{\symbc, \symbs},
\end{equation*}
while for $\frac{T_\symbapr}{\varepsilon^{1/4}} \leq t-t_0 \leq \frac{T_\symbapr}{\varepsilon^{7/4}}$,
\begin{equation*}
\norme{\Rc(t)}_{H^1_\ul(\R)} \leq C_\symbapr \varepsilon^{5/4},
\hspace{4em}
\norme{\Rs(t)}_{H^1_\ul(\R)} \leq C_\symbapr \varepsilon^{9/4}.
\end{equation*}
\end{lemma}
\begin{proof}
Insert $R(t) = V(t) - \psi(\varepsilon, A(t))$ in \eqref{e:system_V} to obtain that for $t\geq t_0$, $R$ satisfies
\begin{equation*}
\partial_t R = \Cal{T}^- R + 2B(\psi, R) + \Cal{Q}^-(R) + \Cal{S}(t) + \Rm{Res}(\psi), 
\end{equation*}
where $B$ is a symmetric bilinear function that satisfies $\Cal{Q}^-(V) = B(V, V)$, and the residual term is
\begin{equation*}
\Rm{Res}(\psi) = -\partial_t \psi + \Cal{T}^- \psi + \Cal{Q}^-(\psi).
\end{equation*}
We decompose Duhamel formula into a critical and stable part:
\begin{equation*}
\varPi_i R(t) = e^{t\Cal{T}^-} \varPi_i R(t_0) + \int_{t_0}^t e^{(t-s)\Cal{T}^-} \varPi_i \left(2B(\psi, R) + \Cal{Q}^-(R) + \Cal{S} + \Rm{Res}(\psi)\right) \d s,
\end{equation*}
with $i\in \Set{\symbc, \symbs}$.
To close a nonlinear argument, we follow \cite{Schneider_94_juil}. Since $A$ satisfies \eqref{e:GL}, the leading order terms in the residual $\Rm{Res}(\psi(\varepsilon, A))$ vanish, see \cref{ss:derive-GL}. Hence we can control $\Rm{Res}(\psi(\varepsilon, A))$ for times smaller than $\varepsilon^{-2} T_\symbapr$, as done in \cite[Lemma 14]{Schneider_94_juil}. To handle the source term -- which is not present in the above reference -- we use \cref{p:decay_rest_term} above, to conclude that $\norme{\Cal{S}(t, \cdot)}_{H^1_\ul(\R)} \leq C \varepsilon^3$ when $t \geq t_0 \geq \frac{T_\symbatt}{\varepsilon^2}$. This implies that
\begin{equation*}
\Norme{\int_{t_0}^t e^{(t-\tau)\Cal{T}^-} \vPic \Cal{S}(\tau) \d \tau}_{H^1_\ul(\R)} \leq C(t-t_0)\varepsilon^3,
\end{equation*}
while 
\begin{equation*}
\Norme{\int_{t_0}^t e^{(t-\tau)\Cal{T}^-} \vPis \Cal{S}(\tau) \d \tau}_{H^1_\ul(\R)} \leq C \varepsilon^3 \int_{t_0}^t e^{-\kappa(t-\tau)} \d \tau \leq \varepsilon^3.
\end{equation*}
Hence when $t-t_0\leq \varepsilon^{-1/4} T_\symbapr$, the estimates on $\Cal{S}$ can be sub-summed into the one on $\Rm{Res}(\psi(\varepsilon, A))$.
Now for $\varepsilon^{-1/4} T_\symbapr \leq t-t_0 \leq \varepsilon^{-7/4} T_\symbapr\leq $, we can once again sub-summed $\Cal{S}$ into $\Rm{Res}(\psi(\varepsilon, A))$.
The remaining terms are handled as in the proof of \cite[Lemma 11]{Schneider_94_juil}, see Lemma 13 therein.
\end{proof}

We are now ready to prove that $V$ remains bounded in time.
\begin{proof}\textit{\Cref{p:U_implies_V}.}\hspace{2ex}
\label{pf:V_bounded}
Beginning with $\norme{V(0)}_{H^1_\ul(\R)} \leq C \norme{V_0}_{W^{1, \infty}} \leq \varepsilon K_1 \leq \varepsilon \delta$, we wait a time $t_0 := \varepsilon^{-2} T_\symbatt$, see \cref{l:attractivity}. Now $V$ decomposes as 
$V = \Vc + \Vs$, with the leading order $\Vc$ writing as:
\begin{equation*}
\hat{\Vc}(t_0, \xi) = \Cal{F}(x\mapsto \varepsilon e^{ix} A_0(\varepsilon x))(\xi) \varrho_\symbc(\xi) + \cc,
\end{equation*}
\ie $\Vc(t_0)$ is obtained as the modulation of a large profile $x \mapsto A_0(\varepsilon x)$, with $\norme{A_0}_{H^1_\ul(\R)} \leq CK_1$ from \cref{l:attractivity}.
We propagate this ansatz by decomposing $V(t,x) = \psi(\varepsilon, A)(t,x) + R(t,x)$, see \eqref{e:def_psi_c}. Applying \cref{l:GL-atractor} and $\norme{A_0}_{L^\infty(\R)} \leq C \norme{A_0}_{H^1_\ul(\R)}$ from \cref{l:Sobolev_embeding}, we conclude that $\norme{A(\varepsilon^2 t)}_{W^{1, \infty}(\R)} \leq C_\gl + C K_1$ for all times $t\geq 0$.
However, $A$ and $\psi(\varepsilon, A)$ does not depend on the same space variable, so that rescaling $H^1_\ul(\R)$ norms loses us a $\varepsilon^{-1/2}$. Instead, we use \cref{l:Sobolev_embeding} and its injections to rescale in $W^{1, \infty}(\R)$ space:
\begin{align*}
\norme{\psi_\symbc(\varepsilon, A)(t)}_{H^1_\ul(\R)} & \leq C \norme{x\mapsto A(\varepsilon^2 t, \varepsilon x)}_{H^1_\ul(\R)}, \\
& \leq C \norme{x\mapsto A(\varepsilon^2 t, \varepsilon x)}_{W^{1, \infty}(\R)}, \\
& \leq C \norme{A(\varepsilon^2 t)}_{W^{1, \infty}(\R)}, \\
& \leq C_\gl + CK_1 e^{- t / \varepsilon^2}
\end{align*}
Recall the expression of $K_1$ \eqref{e:def_K1}.
Similar estimates lead to $\norme{\psi_\symbs(\varepsilon, A)(t)}_{H^1_\ul(\R)} \leq C_\gl + C K_1 e^{- \frac{t}{\varepsilon^2}}$.
Now for $t\geq t_0$ we can estimate
\begin{align*}
\norme{V(t)}_{H^1_\ul(\R)} & \leq \norme{\psi(\varepsilon, A)(t)}_{H^1_\ul(\R)} + \norme{V(t) - \psi(\varepsilon, A)(t)}_{H^1_\ul(\R)},\\
& \leq C \varepsilon(1 + K_1 e^{- t / \varepsilon^2}) + \norme{R(t)}_{H^1_\ul(\R)},
\end{align*}
and it only remains to bound the error $R(t)$. We initialy decompose it as $R_\symbc(t_0) = \Vc(t_0) - \varepsilon \psi_\symbc(A_0)$ and $R_\symbs(t_0) = \Vs(t_0) - \varepsilon^2 \psi_\symbs(A_0)$.
From the definition of $A_0$, 
\begin{equation*}
\pi_1^h \Vc(t_0, x) = \varepsilon e^{ix} A_0(\varepsilon x),
\end{equation*}
while definitions of $\psi_\symbc$ \eqref{e:def_psi_c} and $\pi_1^h$ \eqref{e:def_pi_1} ensure that 
\begin{equation*}
\pi_1^h (\varepsilon \psi_\symbc(A_0))(x) = \Cal{F}^{-1} \left(\chi_\symbc^h \mathbf{1}_{\xi > 0}\Cal{F}(\varepsilon e^{iy} A_0(\varepsilon y)\right) = \varepsilon e^{ix} \Cal{F}^{-1} (\chi_0\Cal{F}(A_0(\varepsilon y))).
\end{equation*}
where $\chi_0 (\xi) = \chi_\symbc^h(\xi - 1) \mathbf{1}_{\xi > 1}$ is supported near $\xi = 0$. Hence 
\begin{equation*}
\pi_1^h (\Vc(t_0) - \varepsilon \psi_\symbc(A_0))(x) = \varepsilon e^{ix} \Cal{F}^{-1}\left((1 - \chi_0) \Cal{F}(A_0(\varepsilon \cdot))\right).
\end{equation*}
Since $R_\symbc(t_0)$ has compact support in Fourier space, we can gain one derivative, and treat separately the behavior close to $1$ and $-1$:
\begin{equation*}
\norme{R_\symbc(t_0)}_{H^1_\ul(\R)} \leq C \norme{\pi_1^h R_\symbc(t_0)}_{L^2_\ul(\R)} + \norme{\pi_{-1}^h R_\symbc(t_0)}_{L^2_\ul(\R)}.
\end{equation*}
We first rescale space, then 
apply the scaled estimate from \cref{l:estimate_mode-filters}, and finally use the localization of $\chi_0$ together with the bound on $A_0$ -- see \cref{l:attractivity} -- to obtain that
\begin{align*}
\norme{\pi_1^h\Rc(t_0)}_{L^2_\ul(\R)} & \leq \frac{1}{\sqrt{\varepsilon}} \Norme{x\mapsto \pi_1^h\Rc\left(t_0, \frac{x}{\varepsilon}\right)}_{L^2_\ul(\R)}, \\
& \leq \sqrt{\varepsilon} \Norme{(1 - \chi_0(\varepsilon \xi - 1)) \pscal{\xi}^{-1/2}}_{\Cal{C}^2_b(\R)} \norme{X \mapsto e^{i X}A_0(X)}_{H^1_\ul(\R)},\\
& \leq \varepsilon^{3/2} C K_1.
\end{align*}
The conjugated part $\pi_{-1}^h R_\symbc(t_0)$ bound identically. 
Furthermore:
\begin{equation*}
\norme{R_\symbs(t_0)}_{H^1_\ul(\R)}\leq \norme{\Vs(t_0)}_{H^1_\ul(\R)} + \varepsilon^2 \norme{\psi_\symbs(A_0)}_{H^1_\ul(\R)} \leq C\varepsilon^2 + C\varepsilon^2 \norme{x\mapsto A_0(\varepsilon x)}_{H^2_\ul(\R)}.
\end{equation*}
Using that $x\mapsto A_0(\varepsilon x)$ has compact support in Fourier space, we can gain two derivatives, inject into $L^\infty(\R)$, scale, and finally inject back into $H^1_\ul(\R)$. It reads: $\norme{x\mapsto A_0(\varepsilon x)}_{H^2_\ul(\R)} \leq C \norme{A_0}_{H^1_\ul(\R)}$.
Hence $\norme{R(t_0)}_{H^1_\ul(\R)} \leq \varepsilon^{5/4}$, and \cref{l:approximation} ensures that 
\begin{equation*}
\norme{R(t)}_{H^1_\ul(\R)} \leq C \varepsilon^{5/4}
\end{equation*}
for times $t_0 \leq t \leq t_1 = t_0 + \varepsilon^{-7/4} T_\symbapr$. Applying the same Lemma as many time as needed, we conclude that the above bound on $R(t)$ holds for times $t_0 \leq t \leq t_n = t_0 + n \varepsilon^{-7/4} T_\symbapr$, without deterioration of the constant.
Hence 
\begin{equation*}
\norme{V(t)}_{H^1_\ul(\R)} \leq C \varepsilon (1 + K_1 e^{- t /\varepsilon^2}) + C \varepsilon^{5/4} \leq  C\sqrt{\mu}(1 + K_1 e^{- t /\mu}),
\end{equation*}
and the proof is complete.
\end{proof}

\appendix

\section{Ginzburg-Landau equation}
\label{ss:derive-GL}
Here we derive the Ginzburg-Landau amplitude equation from the dynamic at $-\infty$. We study the following system:
\begin{equation}
\label{e:system_GL}
\partial_t V = \Cal{T}_0 V + \Cal{T}_\mu V + \Cal{N}_2(V) + \Cal{N}_3(V),
\end{equation}
where the linear terms are given by:
\begin{equation*}
\Cal{T}_0 =
\begin{pmatrix}
d \partial_x^2 - 2\alpha & \beta \\
0 & -(1 + \partial_x^2)^2
\end{pmatrix},
\hspace{4em}
\Cal{T}_\mu =
\begin{pmatrix}
0 & 0 \\
0 & \mu
\end{pmatrix}, 
\end{equation*}
while the quadratic and cubic terms write as:
\begin{equation*}
\Cal{N}_2(V) = \begin{pmatrix}
- 3 \alpha {v_1}^2 \\ \gamma {v_1}{v_2}
\end{pmatrix}
\hspace{4em}
\Cal{N}_3(V) = - \begin{pmatrix}
\alpha {v_1}^3 \\ \sigma {v_2}^3
\end{pmatrix}.
\end{equation*}
In the following, we note $M(X)$ the symbol of the constant coefficient operator $\Cal{T}_0$, such that $\Cal{T}_0 = M(\partial_x)$. We also note $\varepsilon = \sqrt{\mu}$.
Assume that the solution to such system writes according to the following ansatz:
\begin{equation*}
V(t,x) = \varepsilon (e^{ix}A(\varepsilon^2 t, \varepsilon x) + \cc) \varrho_\symbc + \varepsilon^2(A_0\varrho_0 + e^{ix} A_1 \varrho_1 + e^{2ix} A_2 \varrho_2 + \cc).
\end{equation*}
Here and in the following $\cc$ is the complex conjugate: $z + \cc$ stands for $z + \bar{z} = 2\real{z}$. The amplitude $A\in\C$ is the main unknown, while each $A_i(\varepsilon^2 t, \varepsilon x)$ is an amplitude we will fix later. The vector $\varrho_\symbc$ is chosen to be an eigenvector for matrix $M(i)$, associated to eigenvalue $0$, while the vectors $\varrho_i$ will be fixed later. We also introduce $\varrho_\symbc^*$, in the kernel of $M(i)^*$, normalized so that $\pscal{\varrho_\symbc, \varrho_\symbc^*} = 1$. We propose 
\begin{equation*}
\varrho_\symbc =  \begin{pmatrix}
\beta \\ d + 2\alpha
\end{pmatrix}, 
\hspace{4em}
\varrho_\symbc^* := \begin{pmatrix}
0 \\ \frac{1}{d + 2\alpha}
\end{pmatrix}.
\end{equation*}
We inject the ansatz in system \eqref{e:system_GL}, and then identify terms of order $\varepsilon^l e^{ikx}$ for $l\in\N^*$ and $k\in\Z$. As usual for such computation, the $\varepsilon$ order will be trivially satisfied. Then $\varepsilon^2 e^{ikx}$ terms will determine amplitudes $A_k$ and vectors $\varrho_k$. Finally, the $\varepsilon^3 e^{ix}$ equation will lead to Ginzburg-Landau equation. Hence, we may neglect all terms of order $\varepsilon^3 e^{2ix}$ and $\varepsilon^3 e^{0}$. We easily obtain:
\begin{equation*}
\partial_t V = \varepsilon^3 (e^{ix} \partial_T A \varrho_\symbc + \cc)
\end{equation*}
with slow time variable $T = \varepsilon^2 t$. Similarly, we note $X = \varepsilon x$ the large space variable. To compute $\Cal{T}_0 u$, we use a formal Taylor expansion, see \cref{l:conjugaison_operateur_matrice}:
\begin{align*}
\Cal{T}_0 (\varepsilon e^{ix} A \varrho_\symbc) & {} = \varepsilon e^{ix} M(i + \partial_x) A \varrho_\symbc,\\
& {} = \varepsilon e^{ix}\left( A \, M(i) \varrho_\symbc + \varepsilon \partial_X A \, M'(i) \varrho_\symbc + \frac{1}{2}\varepsilon^2 \partial_{XX} A \, M''(i) \varrho_\symbc \right) + \Cal{O}(\varepsilon^4).
\end{align*} 
The computation can be easily adapt for $\Cal{T}_0 (\varepsilon^2 e^{2ikx}A_k \varrho_k)$, for $k\in\lbrace 0, 1, 2\rbrace$. 
The $\Cal{T}_\mu(V)$ term is easily developed, and we are left with the nonlinear terms. We note $B : \R^2 \times \R^2 \lra \R^2$ a symmetric bilinear application that satisfies $B(V, V) = \Cal{N}_2(V)$. It is given by:
\begin{equation*}
B(V, W) = \begin{pmatrix}
- 3 \alpha v_1 w_1 \\ \frac{\gamma}{2}(v_1 w_2 + v_2 w_1)
\end{pmatrix}.
\end{equation*}
Hence, we obtain:
\begin{align*}
\Cal{N}_2(V) = {} & \varepsilon^2 \left(e^{2ix} A^2 \Cal{N}_2(\varrho_\symbc) + \cc + 2\varepsilon^2 \absolu{A}^2 \Cal{N}_2(\varrho_\symbc)\right) + \varepsilon^3 (e^{ix} 2 A A_0 B(\varrho_\symbc, \varrho_0) + \cc) \\
& {} + \varepsilon^3 (e^{ix} 2 \bar{A} A_2 B(\varrho_\symbc, \varrho_2) + \cc) + \Cal{O}_{\absolu{k}\neq 1}(\varepsilon^3 e^{ikx} + \varepsilon^4).
\end{align*}
Similarly, the cubic term develop as:
\begin{equation*}
\Cal{N}_3(V) = \varepsilon^3 (e^ix 3 A \absolu{A}^2 \Cal{N}_3(\varrho_\symbc) + \cc) + \Cal{O}_{\absolu{k}\neq 1}(\varepsilon^3 e^{ikx} + \varepsilon^4).
\end{equation*}
We now collect previous calculus, and identify terms of same order  $\varepsilon^l e^{ikx}$. The $\varepsilon^2 e^{ikx}$ equations write
\begin{equation*}
\partial_X A \, M'(i) \varrho_\symbc + A_1 \, M(i) \varrho_\symbc = 0, 
\end{equation*}
\begin{equation*}
A_0 M(0) \varrho_0 + 2 \absolu{A}^2 \Cal{N}_2(\varrho_\symbc) = 0, 
\hspace{4em}
A_2 M(2i) \varrho_2 + A^2 \Cal{N}_2(\varrho_\symbc) = 0.
\end{equation*}

Both matrix $M(0)$ and $M(2i)$ are invertible, so that we can impose:
\begin{equation*}
A_1 = \partial_X A, 
\hspace{4em}
A_0 = \absolu{A}^2,
\hspace{4em}
A_2 = A^2,
\end{equation*}
\begin{equation*}
\varrho_0 = -2 M(0)^{-1} \Cal{N}_2(\varrho_\symbc),
\hspace{4em}
\varrho_2 = - M(2i)^{-1} \Cal{N}_2(\varrho_\symbc).
\end{equation*}
We also fix $\varrho_1$ such that $M'(i) \varrho_\symbc + M(i) \varrho_1 = 0$, which is possible due to $M'(i)\varrho_\symbc \in \range(M(i))$.
The $\varepsilon^3 e^{ix}$ equation then writes
\begin{align*}
\partial_T A\, \varrho_\symbc = {} & \partial_{XX} A \, (M''(i) \varrho_\symbc + M'(i) \varrho_1) + A
\begin{pmatrix}
0 & 0 \\ 0 & 1
\end{pmatrix} 
\varrho_\symbc\\ 
{} & + A \absolu{A}^2 \left(2 B(\varrho_\symbc, \varrho_0) + 2B(\varrho_\symbc, \varrho_2) + 3\Cal{N}_3(\varrho_\symbc)\right).
\end{align*}
Taking the scalar product with the normalized left eigenvector $\varrho_\symbc^*$ leads to the scalar Ginzburg-Landau equation. In the following, we make explicit its coefficients. 
It is easily seen that $\pscal{M''(i)\varrho_\symbc + M'(i)\varrho_1, \varrho_\symbc^*} = 4$, and that $\Pscal{
\begin{pmatrix}
0 & 0 \\ 0 & 1
\end{pmatrix}
\varrho_\symbc, \varrho_\symbc^*} = 1$. We are left with the coefficient in front of the cubic terms.
We explicitly compute 
\begin{equation*}
M(0)^{-1} = \frac{-1}{2\alpha}
\begin{pmatrix}
1 & \beta \\
0 & 2\alpha
\end{pmatrix}
\hspace{2em}
M(2i)^{-1} = \frac{-1}{9(4d + 2\alpha)}
\begin{pmatrix}
9 & \beta \\
0 & 4d + 2\alpha
\end{pmatrix}
\hspace{2em}
\Cal{N}_2(\varrho) = \begin{pmatrix}
- 3\alpha \beta^2 \\ \gamma \beta (d+2\alpha)
\end{pmatrix},
\end{equation*}
from which we deduce that:
\begin{equation*}
\varrho_0 = 
\begin{pmatrix}
\beta^2 (\gamma \frac{d+2\alpha}{\alpha} - 3) \\ 
2\gamma \beta (d+2\alpha), 
\end{pmatrix},
\hspace{4em}
\varrho_2 = \frac{1}{9(4d + 2\alpha)}
\begin{pmatrix}
\beta^2 (\gamma (d+2\alpha) - 27 \alpha) \\ 
\gamma \beta (d+2\alpha)(4d + 2\alpha) 
\end{pmatrix}.
\end{equation*}
Then, it follows
\begin{equation*}
\Cal{N}_3(\varrho) = - 
\begin{pmatrix}
\alpha \beta^3 \\ 
\sigma(d+2\alpha)^3
\end{pmatrix},
\hspace{4em}
B(\varrho_\symbc, \varrho_0) = 
\begin{pmatrix}
- 3 \alpha \beta^3(\gamma\frac{d + 2\alpha}{\alpha} - 3) \\
\beta^2 \frac{\gamma}{2}(d+2\alpha)\left(2\gamma + \gamma \frac{d+2\alpha}{\alpha} - 3 \right)
\end{pmatrix}, 
\end{equation*}
\begin{equation*}
B(\varrho_\symbc, \varrho_2) = 
\frac{1}{9(4d + 2\alpha)}
\begin{pmatrix}
- 3 \alpha \beta^3(\gamma(d + 2\alpha) - 27 \alpha) \\
\beta^2 \frac{\gamma}{2}(d+2\alpha)(\gamma(5d + 4\alpha) - 27\alpha)
\end{pmatrix}.
\end{equation*}
We can finally compute that:
\begin{align*}
\Pscal{2 B(\varrho_\symbc, \varrho_0) + 2 B(\varrho_\symbc, \varrho_2) + 3 \Cal{N}_3(\varrho_\symbc), \varrho_\symbc^*} = {} & \gamma^2 \beta^2 \left(\frac{19}{9} + (d+2\alpha)\left(\frac{1}{\alpha} + \frac{1}{9(4d + 2\alpha)}\right)\right) \\
& {} - 3\gamma \beta^2\left(1 + \frac{\alpha}{4d + 2\alpha}\right) - 3\sigma(d+2\alpha)^2.
\end{align*}
The right hand side reads as a degree $2$ polynomial in $\gamma$ which we note $P(\gamma)$. Altogether, we obtain the following Ginzburg-Landau equation:
\begin{equation}
\label{e:Ginzburg_Landau}
\tag{\gl}
\partial_T A = 4\partial_{XX} A + A + P(\gamma) A\absolu{A}^2.
\end{equation} 
Since $P$ admits two roots with distincts sign, there exists $\gamma_\gl>0$ such that for all $\gamma \in (0, \gamma_\gl)$, we have $P(\gamma)<0$, \ie hypothesis \eqref{e:hyp_Turing_surcritique} is fulfilled. Recall that $\alpha, d, \sigma >0$. We compute:
\begin{equation*}
\gamma_\gl = \frac{3(4d+3\alpha)}{2a(4d + 2\alpha)} + \sqrt{\left(\frac{3(4d + 3\alpha)}{2a(4d + 2\alpha)}\right)^2 + 3\sigma\frac{(d+2\alpha)^2}{a\beta^2}}, 
\hspace{1em}
a := \frac{19}{9} + (d+2\alpha)\left(\frac{1}{\alpha} + \frac{1}{9(4d + 2\alpha)}\right).
\end{equation*}
\begin{lemma}
\label{l:hyp_are_fulfilled}
There exists an open, nonempty set of parameters $\Omega$ such that for all $(\alpha, \beta, d, \sigma) \in \Omega$, the ordering $\gamma_\Rm{rem} < \gamma_\gl$ holds. In particular for $\gamma \in (\gamma_\Rm{rem}, \gamma_\gl)$, both hypothesis \eqref{e:hyp_marginal_stability} and \eqref{e:hyp_Turing_surcritique} holds true.
\end{lemma}
\begin{proof}
The fact that $\Omega$ is open comes from the continuity of both $\gamma_\Rm{rem}$ and $\gamma_\gl$ with respect to $\alpha, \beta, d, \sigma > 0$. To see that $\Omega \neq \emptyset$, remark that $\gamma_\gl \to +\infty$ when either $\beta \to 0$ or $\sigma \to +\infty$.
\end{proof}
\begin{remark}
The ansatz we propose here only develop up to order $\varepsilon^2$, while the information we extract is held by $\varepsilon^3$ terms. If ones try to push the ansatz one order further:
\begin{align*}
V(t,x) ={} & \varepsilon (e^{ix}A \varrho_\symbc + \cc) + \varepsilon^2(e^{ix} A_{1,1} \varrho_{1,1} + e^{2ix} A_{1,2} \varrho_{1,2} + \cc + A_{1,0} \varrho_{1,0}) \\
& {} + \varepsilon^3(e^{ix} A_{2,1} \varrho_{2,1} + e^{2ix} A_{2,2} \varrho_{2,2} + e^{3ix} A_{2,3} \varrho_{2,3} + \cc + A_{2,0} \varrho_{2,0}),
\end{align*}
the only changes is the presence of an extra $A_{2,1} \pscal{L(i) \varrho_{2,1}, \varrho_\symbc^*}$ term in \eqref{e:Ginzburg_Landau}. By definition of $\varrho_\symbc^*$, this term vanishes anyway. 
\end{remark}

\section{Proof of \cref{p:control_G_lambda_22}}
\label{ss:proof_Glv}
\begin{proof}
\textit{\Cref{p:control_G_lambda_22}.\hspace{2ex}}
Here, we write $a(x,\lambda) \lesssim b(x, \lambda)$ to stand for $a(x, \lambda) \leq C b(x, \lambda)$ where $C$ is a constant. Recall we have construct ODE solutions with exponential behavior. For $0\leq j\leq 3$, we have 
\begin{equation}
\label{e:behavior_phi}
\partial_x^j \phi(\lambda, x) = e^{\nu x}\left(\nu^j + \Cal{O}_{\R^\pm\times K}(e^{-\alpha\absolu{x}})\right),
\end{equation}
where $\phi$ stands for $\phi_i^{\symbsh, \pm}$ and $\nu$ for $\nu_i^{\symbsh,\pm}(\lambda)$. Here, we noted $g(\lambda, x) = \Cal{O}_{\R^+}(f(x))$ if $\absolu{g(\lambda,x)} \lesssim f(x)$ when $x\geq 0$, and if $\lambda \mapsto g(\lambda, x)$ is holomorphic on $K$ for almost all $x\geq 0$. The according notation holds for $\Cal{O}_{\R^-}$. In particular, $\absolu{\phi_i^{\symbsh, \pm}(x)} \lesssim e^{x \real(\nu_i^\pm)}$ when $x\in\R_\pm$. In the following, we drop the \guillemet{$\symbsh$} exponent. Then, the Green function expresses using the decaying solutions:
\begin{equation*}
\Glv(x,y) = \left\lbrace
\begin{array}{ll}
\sum_{i=1}^2 b_i(\lambda,y) \phi_i^{+}(\lambda,x) & \text{ if }y<x,\\[1.5ex]
\sum_{i=3}^4 b_i(\lambda,y) \phi_i^{-}(\lambda,x) & \text{ if }x<y,
\end{array}\right.
\end{equation*}
where the coefficients $b_i(\lambda, y)$ are determined by the jump condition:
\begin{equation*}
\begin{pmatrix}
\phi_1^{+} & \phi_2^{+} & \phi_3^{-} & \phi_4^{-}\\[1.2ex]
\partial_x \phi_1^{+} & \partial_x \phi_2^{+} & \partial_x \phi_3^{-} & \partial_x \phi_4^{-}\\[1.2ex]
\partial_{xx} \phi_1^{+} & \partial_{xx} \phi_2^{+} & \partial_{xx} \phi_3^{-} & \partial_{xx} \phi_4^{-}\\[1.2ex]
\partial_{xxx} \phi_1^{+} & \partial_{xxx} \phi_2^{+} & \partial_{xxx} \phi_3^{-} & \partial_{xxx} \phi_4^{-}\\[1.2ex]
\end{pmatrix}
(\lambda, y) \cdot 
\begin{pmatrix}
b_1 \\ b_2 \\ -b_3 \\ -b_4 
\end{pmatrix}
= 
\begin{pmatrix}
0 \\ 0 \\ 0 \\ -1
\end{pmatrix}.
\end{equation*}
Recall notation \eqref{e:notation_determinant} for determinants. We use Cramer's rule to compute the coefficients $b_i(\lambda, y)$:
\begin{equation}
\label{e:propGlv:expression_b}
b_1 = \frac{\Det(\phi_2^{+},\phi_3^{-}, \phi_4^{-})}{\Wv},
\hspace{4em}
b_2 = - \frac{\Det(\phi_1^{+},\phi_3^{-}, \phi_4^{-})}{\Wv},
\end{equation}
and 
\begin{equation*}
b_3 = \frac{\Det(\phi_1^{+},\phi_2^{+}, \phi_4^{-})}{\Wv},
\hspace{4em}
b_4 = - \frac{\Det(\phi_1^{+},\phi_2^{+}, \phi_3^{-})}{\Wv}.
\end{equation*}
It may happens that $\nu_2^{+}-\nu_1^{+}$ or $\nu_4^{-}-\nu_3^{-}$ vanish, causing a singularity in $(b_1, b_2)$ or $(b_3,b_4)$. However, this singularity can be erased in the expression of $\Glv$ since at first order $b_1(\lambda,y) + b_2(\lambda,y) = \frac{1}{\Wv} \Det(\phi_2 - \phi_1, \phi_3, \phi_4) = \Cal{O}_\lambda(1)$. 
Recall that we assume the $\nu_i^\pm$ are sorted by real part: $\real{\nu_1^+} \leq \real{\nu_2^+} \leq \real{\nu_3^+} \leq \real{\nu_4^+}$ (the same ordering holds for the $\nu_i^-$), and that from \cref{l:spatial_eigenvalues_localization} --  \cref{i:spectral_gap}, there exists $\kappa_2>0$ such that:
\begin{equation}
\label{e:propGlv:spectral_gap}
\real{\nu_{1,2}^\pm} \leq -\kappa_2, \hspace{4em} \kappa_2 \leq \real{\nu_{3,4}^\pm} 
\end{equation}
We now prove the claimed result, depending on the ordering between $x$, $y$ and $0$.
\begin{enumerate}[label = (\textit{\roman*})]
\item \label{i:proof_G_22_1}\textcolor{black}{$y<0<x$.} We need to decompose $\phi_2^+(y)$ into the $(\phi_i^-)_{1\leq i \leq 4}$ basis: $\phi_2^+ = \sum_{i=1^4} c_i \phi_i^-$. We can differentiate three times this decomposition to obtain a Cramer system. Solving it leads to
\begin{align*}
\phi_2^+(y) = & \frac{\Det(\phi_2^+, \phi_2^-, \phi_3^-, \phi_4^-)}{\Det(\phi_1^-, \phi_2^-, \phi_3^-, \phi_4^-)} \phi_1^-(y) + 
\frac{\Det(\phi_1^-, \phi_2^+, \phi_3^-, \phi_4^-)}{\Det(\phi_1^-, \phi_2^-, \phi_3^-, \phi_4^-)} \phi_2^-(y) \\
& + \frac{\Det(\phi_1^-, \phi_2^-, \phi_2^+, \phi_4^-)}{\Det(\phi_1^-, \phi_2^-, \phi_3^-, \phi_4^-)} \phi_3^-(y) + 
\frac{\Det(\phi_1^-, \phi_2^-, \phi_3^-, \phi_2^+)}{\Det(\phi_1^-, \phi_2^-, \phi_3^-, \phi_4^-)} \phi_4^-(y).
\end{align*}
Once again, the $y$ dependence in each of the fraction can be dropped thanks to \cref{l:ode_wronskien}. The numerator is holomorphic with respect to $\lambda$, hence bounded when $\lambda \in K$, and the denominator does not vanish.\footnote{A vanishing determinant would implies that $\lambda$ is an eigenvalue, which can not hold due to $\real{\lambda}\geq -2 \eta$.} In the following, we may simply write
\begin{equation*}
c_1^-(\varphi) := \frac{\Det(\varphi, \phi_2^-, \phi_3^-, \phi_4^-)}{\Det(\phi_1^-, \phi_2^-, \phi_3^-, \phi_4^-)}
\end{equation*}
to note the first coefficient of $\varphi$ when decomposed in the $(\phi_i^-)_{1\leq i \leq 4}$ basis. We extend this notation to other coefficients: $c_i^-(\varphi)$, and to the decomposition into $(\phi_i^+)_{1\leq i \leq 4}$ basis: $c_i^+(\varphi)$.

Using the above decomposition into the expression of $\Glv$, we obtain:
\begin{align}
\nonumber
\Absolu{\Det(\phi_2^{+},\phi_3^{-}, \phi_4^{-})(y)} & {} = \Absolu{ \sum_{i = 1}^2 c_i^-(\phi_2^{+}) \Det(\phi_i^{-},\phi_3^{-}, \phi_4^{-})(y)},\\ 
& \lesssim e^{\real(\nu_1^- + \nu_3^- + \nu_4^-)y} + e^{\real(\nu_2^- + \nu_3^- + \nu_4^-)y},
\label{e:propGlv:estimateD}
\end{align}
since $\lambda \mapsto \Det(\phi_i^{-},\phi_3^{-}, \phi_4^{-})(0)$ is holomorphic on $K$. Using again \cref{l:ode_wronskien}, we see that
\begin{equation*}
\Wv(\lambda, y) = e^{\real(\nu_1^- + \nu_2^- + \nu_3^- + \nu_4^-)y}\ \Wv(\lambda, 0).
\end{equation*}
Using \cref{p:point_spectrum}, we see that $\Wv(\lambda, 0)\gtrsim 1$ when $\lambda$ lies in $K$. Collecting the above estimates \eqref{e:propGlv:estimateD}, \eqref{e:behavior_phi} and \eqref{e:propGlv:spectral_gap}, we now see that
\begin{equation*}
\Absolu{\phi_1^{+}(x) \frac{\Det(\phi_2^{+},\phi_3^{-}, \phi_4^{-})(y)}{\Wv(y)}} \lesssim e^{x\real(\nu_1^+)  -y\real(\nu_2^-)} + e^{x\real(\nu_1^+) - y\real(\nu_1^-)} \lesssim e^{-\kappa_2 (x-y)}.
\end{equation*}
The exact same approach for the second term leads to:
\begin{equation*}
\Absolu{\phi_2^{+}(x) \frac{\Det(\phi_1^{+},\phi_3^{-}, \phi_4^{-})(y)}{\Wv(y)}} \lesssim e^{x\real(\nu_2^+)  -y\real(\nu_2^-)} + e^{x\real(\nu_2^+) - y\real(\nu_1^-)} \lesssim e^{-\kappa_2 (x-y)}.
\end{equation*}
\end{enumerate}
The first case is now done. For the following ones, same tools are used. We nevertheless give the general plan of the proofs to clarify some points.
\begin{enumerate}[label = (\textit{\roman*})]
\addtocounter{enumi}{1}
\item \textcolor{black}{$0<y<x$.}\label{i:propGlv:case2} We decompose both $\phi_3^-(y)$ and $\phi_4^-(y)$ into the $(\phi_i^+)_{1\leq i \leq 4}$ basis, and use the ordering \eqref{e:propGlv:spectral_gap} of the spatial eigenvalues to obtain
\begin{align*}
\Absolu{\frac{\Det(\phi_2^{+},\phi_3^{-}, \phi_4^{-})(y)}{\Wv(y)}} & \lesssim \sum_{\substack{i \neq j\\[0.3ex] i,j\neq 2}} \Absolu{ c_i^+(\phi_3^{-}) c_j^+(\phi_4^{-}) \frac{\Det(\phi_2^{+},\phi_i^{+}, \phi_j^{+})(y)}{e^{\real(\nu_1^+ + \cdots + \nu_4^+)y}}}, \\
& \lesssim e^{-y\real(\nu_1^+)} + e^{-y\real(\nu_3^+)} + e^{-y\real(\nu_4^+)}.
\end{align*}
The second and third term are treated directly: $-y\real(\nu_{3,4}^+) \leq -y\kappa_2 \leq y\kappa_2$. For the first term, we need to use the ordering of $x$ and $y$. Since $x-y \geq 0$, we have $\real(\nu_{1}^+)(x-y) \leq -\kappa_2 (x-y)$. Hence:
\begin{equation*}
\Absolu{\phi_1^{+}(x) \frac{\Det(\phi_2^{+},\phi_3^{-}, \phi_4^{-})(y)}{\Wv(y)}}  \lesssim e^{x\real(\nu_1^+) - y\real(\nu_1^+)} + e^{x\real(\nu_1^+) - y\real(\nu_{3,4}^+)}\lesssim e^{-\kappa_2(x-y)},
\end{equation*}
A similar argument leads to:
\begin{equation*}
\Absolu{\phi_2^{+}(x) \frac{\Det(\phi_1^{+},\phi_3^{-}, \phi_4^{-})(y)}{\Wv(y)}}  \lesssim e^{x\real(\nu_2^+) - y\real(\nu_2^+)} + e^{x\real(\nu_2^+) - y\real(\nu_{3,4}^+)}\lesssim e^{-\kappa_2(x-y)}.
\end{equation*}
\item \textcolor{black}{$y<x<0$.} Here we need to decompose both $\phi_1^+(x)$ and $\phi_2^+(y)$. Taken independently, both term $\phi_1^+ b_1$ and $\phi_2^+ b_2$ do not decay as claimed, due to the following asymmetric growing rates, which correspond to terms where the two projections are done on the same element of $(\phi_i^-)_{1\leq i \leq 4}$:
\begin{equation*}
E_1(\lambda, x, y) := c_1^-(\phi_1^+) c_1^-(\phi_2^+) \phi_1^-(x)\frac{\Det(\phi_1^-, \phi_3^-, \phi_4^-)(y)}{\Wv(y)} \sim e^{\nu_1^- x - \nu_2^-y},
\end{equation*}
and
\begin{equation*}
E_2(\lambda, x, y) := c_2^-(\phi_1^+) c_2^-(\phi_2^+) \phi_2^-(x)\frac{\Det(\phi_2^-, \phi_3^-, \phi_4^-)(y)}{\Wv(y)}\sim e^{\nu_2^- x - \nu_1^-y}.
\end{equation*}
However, this terms appear both in $\phi_1^+ b_1$ and $\phi_2^+ b_2$, so that they cancel out in the expression of $\Glv$, recall from \eqref{e:propGlv:expression_b} that $b_1$ and $b_2$ are obtained with opposite sign in their expressions. We compute
\begin{align*}
\Absolu{\phi_1^+ b_1 - E_1 - E_2}(x, y) \lesssim & \left(e^{x\real(\nu_1^-)} + e^{x\real(\nu_3^-)} + e^{x\real(\nu_4^-)}\right) e^{-y\real(\nu_1^-)}\\
& + \left(e^{x\real(\nu_2^-)} + e^{x\real(\nu_3^-)} + e^{x\real(\nu_4^-)}\right) e^{-y\real(\nu_2^-)},\\
\lesssim & \ e^{-\kappa_2(x-y)}.
\end{align*}
For the second inequality we have used the same method as in the previous case. From one hand $(x-y)\real(\nu_1^+) \leq -\kappa_2(x-y)$ due to the ordering of $x$ and $y$. From the other hand, $x\real(\nu_3) \leq x\kappa_2 \leq -x\kappa_2$ coupled with $-y\real(\nu_1^-)\leq -y\kappa_2$.
Similar computations leads to 
\begin{equation*}
\Absolu{\phi_2^+b_2 + E_1 +  E_2}(x, y) \lesssim e^{-\kappa_2(x-y)}.
\end{equation*}
We conclude using triangular inequality: $\absolu{\phi_1^+(x)b_1(y) + \phi_2^+(x)b_2(y)} \lesssim e^{-\kappa_2(x-y)}$.
\end{enumerate}
We are now done with all cases where $y<x$. The $x<y$ cases are mirror versions of the three above cases. Computations can be adapted easily, we omit them.
This complete the proof of the first stated inequality. We now treat the second one \eqref{e:control_G_lambda_22_extra}. Remark that when $y\geq 1$, \eqref{e:control_G_lambda_22_extra} adds no information to the first inequality, since $\omega_{\kappa_2,0}(y) = 1$. Hence we are only left with the case $y\leq 0 \leq x$, in which we have:
\begin{equation*}
\Absolu{\Glv(x,y)} \leq C e^{-\kappa_2\absolu{x-y}} = C e^{-\frac{\kappa_2}{2}(x-y)} e^{\frac{\kappa_2}{2} y} e^{-\frac{\kappa_2}{2} x} \leq e^{-\frac{\kappa_2}{2}\absolu{x-y}} \omega_{\kappa_2/2, 0}(y).
\end{equation*}
Up to a change of notation $\tilde{\kappa}_2 = \frac{1}{2}\kappa_2$, this is the claimed \eqref{e:control_G_lambda_22_extra}. The proof is now complete.
\end{proof}

\section{Three equilibrium points}
\label{ss:3-equilibriums}
It appears that for certain sets of parameters $(d, \alpha, \beta, \gamma, \sigma)$, system \eqref{e:original_system} admits nontrivial equilibrium $(u_{eq}, v_{eq})$ with $v_{eq} \sim \sqrt{\mu}$. To keep dynamic as simple as possible, for example when working with numerical simulations, we can restrict to parameters that satisfy the following statement. This is by no mean necessary to our study, since the main result is perturbative.
\begin{proposition}[Three equilibrium points]
\label{p:3equilibriums}
Set $\mu_0<1$ and $\gamma > 0$. Then for positive $d, \alpha, \beta$, there exists $\sigma_0>0$ such that for all $\sigma>\sigma_0$, and $0\leq \mu < \mu_0$, system \eqref{e:original_system} admits only the steady states $(u,v) = (\pm 1, 0)$ and $(u,v) = (0, 0)$.
\end{proposition}
\begin{proof}
Assume that $(u,v)\in\R^2$ is a constant solution of \eqref{e:original_system}. If $v = 0$, the only solutions are $(\pm1, 0)$ and $(0, 0)$. Else, the point $(u,v)$ lies in the intersection of two curves:
\begin{equation*}
\begin{cases}
v = - \frac{\alpha}{\beta}u(1-u^2),\\
u = 1 - \frac{1}{\gamma}\left(\mu - 1 - \sigma v^2\right).
\end{cases}
\end{equation*}
For $v \leq 0$, they do not intersect since they respectively ensures $u\leq 1$ and $u > 1$.\footnote{The second inequality comes from $\gamma > 0$ and $\mu < 1$.}
For $v > 0$, we use the tangent line to the first curve at $(u,v) = (1,0)$:
\begin{equation*}
T = \Set{\left(1 + u, \frac{2\alpha}{\beta} u\right) : u\in\R} = \Set{\left(1 + \frac{\beta}{2\alpha}v , v\right) : v\in\R}.
\end{equation*}
Remark that the function $u\mapsto -\frac{\alpha}{\beta} u (1 - u^2)$ is convex for $u > 1$. Hence the two curves do not intersect for positive $v$ provided the second curve do not intersect $T$, which reads:
\begin{equation*}
1 + \frac{\beta}{2\alpha} v < 1 - \frac{1}{\gamma}(\mu - 1 - \sigma v^2) 
\hspace{2em}
\Leftrightarrow
\hspace{2em}
\sigma v^2 - \frac{\gamma \beta}{2\alpha} v + 1-\mu > 0.
\end{equation*}
It holds true for sufficiently large $\sigma>0$. This complete the proof.
\end{proof}

\section{Proof of \cref{l:ODE_solutions}}
\label{ss:proof_lemma_ODE}
\begin{proof}
\textit{\Cref{l:ODE_solutions}.\hspace{2ex}}
We make the change of variable $z(t) = e^{-tA_\infty}y(t)$, so that 
\begin{equation}
\label{e:lemme_edo:edo2}
z' = \left(A(t) - A_\infty\right)z, 
\end{equation}
which writes as 
\begin{equation}
\label{e:lemme_ode:eq_integral_1}
z(t) = z(0) + \int_0^t  R(s)z(s) \d s = v - \int_t^{+\infty} R(s) z(s) \d s
\end{equation}
with $\norme{R(t)} = \norme{A(t) - A_\infty} \leq C e^{-\alpha t}$, and $v = z(0) + \int_0^{+\infty} R(s) z(s) \d s$. Hence for $z$ bounded, equation \eqref{e:lemme_edo:edo2} together with condition $z(t)\lra v$ when $t\to +\infty$ is equivalent to \eqref{e:lemme_ode:eq_integral_1}. We now make the change of variable $x(t) = e^{\frac{\alpha}{2} t} (z(t) - v)$, that satisfies:
\begin{equation}
\label{e:lemme_ode:eq_integral_2}
x(t) = (K x)(t) := -\int_t^{+\infty} e^{-\frac{\alpha}{2}(t-s)} R(s) (x(s) + v) \d s.
\end{equation}
The operator $K$ is a contraction from $L^\infty(T, +\infty)^{n}$ to itself, provided $T$ is large enough. Indeed, for $t\geq T$ we have
\begin{equation*}
\norme{Kx - K\tilde{x}}(t) \leq \norme{x-\tilde{x}}_{L^\infty(T, +\infty)}\int_t^{+\infty} e^{-\frac{\alpha}{2}(t-s) - \alpha s} \d s \leq \frac{2e^{-\alpha T}}{\alpha} \norme{x-\tilde{x}}_{L^\infty(T, +\infty)}.
\end{equation*} 
The Picard fixed point theorem ensures the existence of a unique $\kappa\in L^\infty(T,+\infty)$ solution of \eqref{e:lemme_ode:eq_integral_2}. Reverting back all changes of variable, we obtain a solution of \eqref{e:lemme_edo:edo1} with
\begin{equation*}
y(t) = e^{tA_\infty}(v + e^{-\frac{\alpha}{2} t} \kappa(t))
\end{equation*}
as claimed. By Cauchy-Lipschitz theorem, we flow this solution backward to define it on $\R$. 

We now assume that $(v_i)_{1\leq i\leq n}$ is a basis of $\R^n$. Since the family $(e^{-tA_\infty} y_i(t))_i$ converges to $(v_i)_i$ when $t\to +\infty$, we obtain that 
\begin{equation*}
\det(e^{-tA_\infty} y_1(t), \dots, e^{-tA_\infty} y_n(t)) = \det(e^{-tA_\infty}) \det(y_1(t), \dots, y_n(t))
\end{equation*}
is nonzero for $t$ large enough, by continuity of the determinant. This ensures that $(y_i)_i$ is a basis for the solutions of equation \eqref{e:lemme_edo:edo1}.

Finally, we suppose that $A(t)$, $A_\infty$ and $v$ are holomorphic with respect to $\lambda$. Then the operator $K_\lambda: L^\infty(T, +\infty)^n \to L^\infty(T, +\infty)^n$ defined by \eqref{e:lemme_ode:eq_integral_2} is holomorphic with respect to $\lambda$, and so is $\kappa$. This conclude the proof.
\end{proof}

\section{Standard Lemmas}
\begin{lemma}
\label{l:conjugaison_operateur_matrice}
Let $M$ be a matrix with polynomial coefficients: $M(X) = \left(P_{i,j}(X)\right)_{1\leq i,j \leq n}$, we note $d:= \max_{i,j} \deg(P_{i,j})$ the greatest degree in $M$, and $\Cal{L}$ the associated differential operator: $\Cal{L} = M(\partial_x)$. Then the conjugation of $\Cal{L}$ by the weight $e^{ \vartheta x}$ writes as:
\begin{equation*}
e^{-\vartheta x} \Cal{L} e^{\vartheta x} = M(\vartheta + \partial_x) = \sum_{k=0}^{d} \frac{1}{k!} M^{(k)}(\vartheta) \partial_x^k,
\end{equation*}
where $M^{(k)}$ is the matrix whose coefficients are the polynomial's derivatives: $(P_{i,j}^{(k)})_{1\leq i,j\leq n}$.
\end{lemma}
\begin{proof}
We first show the scalar case $n = 1$. By linearity, it is enough to consider the case $\Cal{L} = \partial_x^d$, which corresponds to $P(X) = X^d$. It can be easily computed that $e^{-\vartheta x}\partial_x\left(e^{\vartheta x} \cdot \right) = (\vartheta + \partial_x)$, rising it to the power $d$ reads: $e^{-\vartheta x}\Cal{L}\left(e^{\vartheta x} \cdot \right) = (\vartheta + \partial_x)^d = P(\vartheta + \partial_x)$. To finish the scalar case, we do a Taylor expansion at order $d$. Such expansion is exact due to $P$ being polynomial:
\begin{equation*}
P(\vartheta + X) = \sum_{k = 0}^d \frac{P^{(k)}(\vartheta)}{k!} X^k.
\end{equation*} 
We then formally replace $X$ by $\partial_x$. The case of a matrix of differential operator is straightforward, is suffice to apply the above to each component of the matrix.
\end{proof}

\begin{lemma}
\label{l:ode_wronskien}
Let $\Cal{L} = \sum_{j=0}^n a_j(x)\partial_x^j$ be a differential operator with coefficients $x\mapsto a_j(x)$. Let $\phi_1, \dots, \phi_n$ be solutions of the ODE 
\begin{equation}
\label{e:ode_wronskien}
(\lambda - \Cal{L}) \phi = 0,
\end{equation}
and note $\Phi_i = \transpose{(\phi_i, \dots, \partial_x^{n-1} \phi_i)}$. Then 
\begin{equation*}
\det(\Phi_1, \dots, \Phi_n)(\lambda, x) = \exp\left(-\int_0^x a_{n-1}(s) \d s\right) 
\det(\Phi_1, \dots, \Phi_n)(\lambda, 0).
\end{equation*}
\end{lemma}
\begin{proof}
For the proof, we will note $\Phi$ the $n\times n$ matrix whose $i$-th column is $\Phi_i$.
Remark that $\Cal{W}(\lambda, x) := \det(\Phi)(\lambda, x)$ satisfy:
\begin{align*}
\partial_x \Cal{W} = \Rm{Tr}\left(\transpose{\Rm{Com}(\Phi)}\cdot \partial_x(\Phi)\right) = \Rm{Tr}\left(\transpose{\Rm{Com}(\Phi)}\cdot A \cdot \Phi\right),
\end{align*}
where $A(\lambda, x)$ is the $n\times n$ matrix obtained if one vectorise the ODE \eqref{e:ode_wronskien}. Since the trace is invariant under circular permutation, we recognize
\begin{equation*}
\partial_x \Cal{W} = \Rm{Tr}\left(A \cdot \Phi \cdot \transpose{\Rm{Com}(\Phi)}\right) =  \det(\Phi) \Rm{Tr}(A).
\end{equation*}
This scalar ODE on $\Cal{W}$ is easily solved. Noticing that $\Rm{Tr}(A(\lambda, x)) = -a_{n-1}(x)$ conclude the proof.
\end{proof}

\begin{lemma}
\label{l:Sobolev_embeding}
Let $p, q\in \R\cup \lbrace +\infty\rbrace$ such that $1 \leq q \leq p \leq +\infty$, and $X, Y$ be two measurable spaces. Then the following injection is continuous:
\begin{equation*}
L^q(X, L^p(Y, \R)) \subset L^p(Y, L^q(X, \R)).
\end{equation*}
\end{lemma}
\begin{proof}
This is the well-known Minkowsky's integral inequality.
\end{proof}

\begin{lemma}
\label{l:integrale_non_lineaire}
Let $\alpha, \beta > 1$ and $\gamma>0$ be reals. Then there exists $C>0$ such that for all $t\geq 0$:
\begin{equation}
\label{e:integrale_non_lineaire}
\int_0^t \frac{\d s}{(1+s)^\alpha (1+t-s)^\beta} \leq \frac{C}{(1+t)^{\min(\alpha, \beta)}},
\hspace{4em}
\int_0^t \frac{e^{-\gamma (t-\tau)}}{(1+\tau)^\alpha}\d \tau \leq \frac{C}{(1+t)^\alpha}.
\end{equation}
\end{lemma}
\begin{proof}
One can slightly adapt this statement when $\alpha = 1$ or $\beta = 1$, see \cite{Xin_92}. First, remark that 
\begin{equation*}
I_{\alpha, \beta}(t) := \int_0^{t/2} \frac{\d s}{(1+s)^\alpha (1+t-s)^\beta} \leq \frac{1}{(1+t/2)^\beta} \left[\frac{1-\alpha}{(1+s)^{\alpha-1}}\right]_0^{t/2} \leq \frac{\alpha - 1}{(1+t/2)^\beta}.
\end{equation*}
Hence, we cut the integral in \eqref{e:integrale_non_lineaire} at $t/2$, and the change of variable $z = t-s$ leads to
\begin{equation*}
\int_0^t \frac{\d s}{(1+s)^\alpha (1+t-s)^\beta} = I_{\alpha, \beta}(t) + I_{\beta, \alpha}(t) \leq \frac{C}{(1+t/2)^{\min(\alpha, \beta)}}\leq \frac{C\  2^{\min(\alpha, \beta)}}{(1+t)^{\min(\alpha, \beta)}}.
\end{equation*}
Which is the first claimed estimate. The second one follows from $e^{-\gamma(t-\tau)} \leq \frac{C}{(1+t-\tau)^{\alpha}}.$
\end{proof}

\begin{lemma}
\label{l:cloture_non_lineaire}
Let $b, \varepsilon$ be positive constants, and note $a_0(b, \varepsilon) = \frac{1}{2}(2b)^{-\frac{1}{\varepsilon}} >0$. Then for all $0<a<a_0$, the following holds. If $t\mapsto x(t)$ is a positive continuous function  that satisfy $x(t)\leq a + b x(t)^{1+\varepsilon}$ and such that $x(0)$ is small enough, then $x(t)\leq 2a$ for all $t\geq 0$.
\end{lemma}
\begin{proof}
Introduce $f(x) = a + bx^{1+\varepsilon} - x$. Then $f(2a) = a(2^{1+\varepsilon} b a^\varepsilon- 1) < 0$ when $a$ is smaller than $a_0$, while $f(0) = a >0$. Assume $x(0)\leq 2a$. Then the connected component of $f^{-1}(\R_+)$ that contains $x(0)$ is include in $(-\infty, 2a)$. Since $f(x(t))\geq 0$ by construction, the continuity of $x$ implies $x(t)\leq 2a$ for all $t>0$. This conclude the proof.
\end{proof}

\section*{Acknowledgments}
The author thanks Gr{\'e}gory Faye for multiple fruitful discussions.
This works was partially supported by Labex CIMI under grant agreement ANR-11-LABX-0040.

\bibliographystyle{aomalpha}
\bibliography{/home/garenaux/texmf/Bibliographie}

\providecommand{\bysame}{\leavevmode\hbox to3em{\hrulefill}\thinspace}
\providecommand{\noopsort}[1]{}
\providecommand{\mr}[1]{\href{http://www.ams.org/mathscinet-getitem?mr=#1}{MR~#1}}
\providecommand{\zbl}[1]{\href{http://www.zentralblatt-math.org/zmath/en/search/?q=an:#1}{Zbl~#1}}
\providecommand{\jfm}[1]{\href{http://www.emis.de/cgi-bin/JFM-item?#1}{JFM~#1}}
\providecommand{\arxiv}[1]{\href{http://www.arxiv.org/abs/#1}{arXiv~#1}}
\providecommand{\doi}[1]{\url{https://doi.org/#1}}
\providecommand{\MR}{\relax\ifhmode\unskip\space\fi MR }
\providecommand{\MRhref}[2]{%
  \href{http://www.ams.org/mathscinet-getitem?mr=#1}{#2}
}
\providecommand{\href}[2]{#2}
\begin{thebibliography}{BNPR09}

\bibitem[AW78]{Aronson_Weinberger_78}
\bgroup\scshape{}D.~G. Aronson\egroup{} and \bgroup\scshape{}H.~F.
  Weinberger\egroup{}, Multidimensional nonlinear diffusion arising in
  population genetics,  \emph{Advances in Mathematics} \textbf{30} no.~1
  (1978), 33 -- 76. \doi{10.1016/0001-8708(78)90130-5}.

\bibitem[AS21]{Avery_Scheel_21}
\bgroup\scshape{}M.~Avery\egroup{} and \bgroup\scshape{}A.~Scheel\egroup{},
  Asymptotic stability of critical pulled fronts via resolvent expansions near
  the essential spectrum,  \emph{SIAM Journal on Mathematical Analysis}
  \textbf{53} no.~2 (2021), 2206--2242. \doi{10.1137/20M1343476}.

\bibitem[BGS09]{Beck_Ghazaryan_Sandstede_09}
\bgroup\scshape{}M.~Beck\egroup{}, \bgroup\scshape{}A.~Ghazaryan\egroup{}, and
  \bgroup\scshape{}B.~Sandstede\egroup{}, Nonlinear convective stability of
  travelling fronts near turing and hopf instabilities,  \emph{Journal of
  Differential Equations} \textbf{246} no.~11 (2009), 4371 -- 4390.
  \doi{10.1016/j.jde.2009.02.003}.

\bibitem[BNPR09]{Berestycki_Nadin_Perthame_Ryzhik_09}
\bgroup\scshape{}H.~Berestycki\egroup{}, \bgroup\scshape{}G.~Nadin\egroup{},
  \bgroup\scshape{}B.~Perthame\egroup{}, and
  \bgroup\scshape{}L.~Ryzhik\egroup{}, The non-local {Fisher--KPP} equation:
  travelling waves and steady states,  \emph{Nonlinearity} \textbf{22} no.~12
  (2009), 2813--2844. \doi{10.1088/0951-7715/22/12/002}.

\bibitem[CE90]{Collet_Eckmann_90}
\bgroup\scshape{}P.~Collet\egroup{} and \bgroup\scshape{}J.-P.
  Eckmann\egroup{}, The time dependent amplitude equation for the
  swift-hohenberg problem,  \emph{Communications in Mathematical Physics}
  \textbf{132} no.~1 (1990), 139--153. \doi{10.1007/BF02278004}.

\bibitem[CH93]{Cross-Hohenberg-93}
\bgroup\scshape{}M.~C. Cross\egroup{} and \bgroup\scshape{}P.~C.
  Hohenberg\egroup{}, Pattern formation outside of equilibrium,  \emph{Rev.
  Mod. Phys.} \textbf{65} no.~3 (1993), 851--1112.
  \doi{10.1103/RevModPhys.65.851}.

\bibitem[Dav02]{Davies_2002}
\bgroup\scshape{}B.~Davies\egroup{}, \emph{Integral transforms and their
  applications}, \emph{Texts in Applied Mathematics} \textbf{41}, Springer New
  York, 2002. \doi{10.1007/978-1-4684-9283-5}.

\bibitem[EW91]{Eckmann-Wayne-91}
\bgroup\scshape{}J.-P. Eckmann\egroup{} and \bgroup\scshape{}C.~E.
  Wayne\egroup{}, Propagating fronts and the center manifold theorem,
  \emph{Communications in Mathematical Physics} \textbf{136} no.~2 (1991),
  285--307. \doi{10.1007/BF02100026}.

\bibitem[FH15]{Faye_Holzer_15}
\bgroup\scshape{}G.~Faye\egroup{} and \bgroup\scshape{}M.~Holzer\egroup{},
  Modulated traveling fronts for a nonlocal {Fisher-KPP} equation: A dynamical
  systems approach,  \emph{Journal of Differential Equations} \textbf{258}
  no.~7 (2015), 2257 -- 2289. \doi{10.1016/j.jde.2014.12.006}.

\bibitem[FH18]{Faye_Holzer_18}
\bgroup\scshape{}G.~Faye\egroup{} and \bgroup\scshape{}M.~Holzer\egroup{},
  Asymptotic stability of the critical {Fisher--KPP} front using pointwise
  estimates,  \emph{Zeitschrift f{\"u}r angewandte Mathematik und Physik}
  \textbf{70} no.~1 (2018). \doi{10.1007/s00033-018-1048-0}.

\bibitem[FHSS21]{Faye_Holzer_Scheel_Siemer_20}
\bgroup\scshape{}G.~Faye\egroup{}, \bgroup\scshape{}M.~Holzer\egroup{},
  \bgroup\scshape{}A.~Scheel\egroup{}, and \bgroup\scshape{}L.~Siemer\egroup{},
  Invasion into remnant instability: a case study of front dynamics,
  \emph{Indiana Univ. Math. J.} (2021), 1--63.

\bibitem[Fis37]{Fisher_37}
\bgroup\scshape{}R.~A. Fisher\egroup{}, The wave of advance of advantageous
  genes,  \emph{Annals of Eugenics} \textbf{7} no.~4 (1937), 355--369.
  \doi{10.1111/j.1469-1809.1937.tb02153.x}.

\bibitem[Gal94]{Gallay_94}
\bgroup\scshape{}T.~Gallay\egroup{}, Local stability of critical fronts in
  nonlinear parabolic partial differential equations,  \emph{Nonlinearity}
  \textbf{7} (1994), 741 -- 764.

\bibitem[GSU04]{Gallay_Schneider_Uecker_04}
\bgroup\scshape{}T.~Gallay\egroup{}, \bgroup\scshape{}G.~Schneider\egroup{},
  and \bgroup\scshape{}H.~Uecker\egroup{}, Stable transport of information near
  essentially unstable localized structures,  \emph{Discrete \& Continuous
  Dynamical Systems - B} \textbf{4} no.~1531-3492-2004-2-349 (2004), 349.
  \doi{10.3934/dcdsb.2004.4.349}.

\bibitem[GS07]{Ghazaryan_Sandstede_07}
\bgroup\scshape{}A.~Ghazaryan\egroup{} and
  \bgroup\scshape{}B.~Sandstede\egroup{}, Nonlinear convective instability of
  turing unstable fronts near onset: A case study,  \emph{SIAM J. Applied
  Dynamical Systems} \textbf{6} (2007), 319--347. \doi{10.1137/060670262}.

\bibitem[HR14]{Hamel_Ryzhick_14}
\bgroup\scshape{}F.~Hamel\egroup{} and \bgroup\scshape{}L.~Ryzhik\egroup{}, On
  the nonlocal {Fisher--KPP} equation: steady states, spreading speed and
  global bounds,  \emph{Nonlinearity} \textbf{27} no.~11 (2014), 2735--2753.
  \doi{10.1088/0951-7715/27/11/2735}.

\bibitem[Hen81]{Henry_semilinear_parabolic_equation}
\bgroup\scshape{}D.~Henry\egroup{}, \emph{Geometric Theory of Semilinear
  Parabolic Equations}, 1 ed., \emph{Lecture Notes in Mathematics 840},
  Springer-Verlag Berlin Heidelberg, 1981.

\bibitem[HH05]{Howard_Hu_04}
\bgroup\scshape{}P.~Howard\egroup{} and \bgroup\scshape{}C.~Hu\egroup{},
  Pointwise {Green}'s function estimates toward stability for multidimensional
  fourth-order viscous shock fronts,  \emph{Journal of Differential Equations}
  \textbf{218} no.~2 (2005), 325 -- 389. \doi{10.1016/j.jde.2005.01.006}.

\bibitem[KP13]{Kapitula_Promislow}
\bgroup\scshape{}T.~Kapitula\egroup{} and
  \bgroup\scshape{}K.~Promislow\egroup{}, \emph{Spectral and dynamical
  stability of nonlinear waves}, Springer-Verlag New York, 2013.

\bibitem[KPP37]{KPP_37}
\bgroup\scshape{}A.~Kolmogorov\egroup{},
  \bgroup\scshape{}I.~Petrovsky\egroup{}, and
  \bgroup\scshape{}N.~Piskunov\egroup{}, {\'E}tude de l'\'equation de la
  diffusion avec croissance de la quantit\'e de mati\`ere et son application
  \`a un probl\`eme biologique,  \emph{Moscow university bulletin of
  mathematics} \textbf{1} (1937), 1--25.

\bibitem[Mie02]{Mielke_02}
\bgroup\scshape{}A.~Mielke\egroup{}, The {Ginzburg-Landau} equation in its role
  as a modulation equation,  in \emph{Handbook of Dynamical Systems}
  (\bgroup\scshape{}B.~Fiedler\egroup{}, ed.), \textbf{2}, Elsevier Science,
  2002, pp.~759 -- 834. \doi{10.1016/S1874-575X(02)80036-4}.

\bibitem[MS95]{Mielke_Schneider_95}
\bgroup\scshape{}A.~Mielke\egroup{} and \bgroup\scshape{}G.~Schneider\egroup{},
  Attractors for modulation equations on unbounded domains-existence and
  comparison,  \emph{Nonlinearity} \textbf{8} no.~5 (1995), 743--768.
  \doi{10.1088/0951-7715/8/5/006}.

\bibitem[NPT11]{Nadin_Perthame_Tang-11}
\bgroup\scshape{}G.~Nadin\egroup{}, \bgroup\scshape{}B.~Perthame\egroup{}, and
  \bgroup\scshape{}M.~Tang\egroup{}, Can a traveling wave connect two unstable
  states? the case of the nonlocal {Fisher} equation,  \emph{Comptes Rendus
  Mathematique} \textbf{349} no.~9 (2011), 553--557.
  \doi{10.1016/j.crma.2011.03.008}.

\bibitem[SS00]{Sandstede_Scheel_PRSEA_00}
\bgroup\scshape{}B.~Sandstede\egroup{} and \bgroup\scshape{}A.~Scheel\egroup{},
  Spectral stability of modulated travelling waves bifurcating near essential
  instabilities,  \emph{Proceedings of the Royal Society of Edinburgh: Section
  A Mathematics} \textbf{130} no.~2 (2000), 419--448.
  \doi{10.1017/S0308210500000238}.

\bibitem[SS01]{Sandstede_Scheel_01}
\bgroup\scshape{}B.~Sandstede\egroup{} and \bgroup\scshape{}A.~Scheel\egroup{},
  Essential instabilities of fronts: bifurcation, and bifurcation failure,
  \emph{Dynamical Systems} \textbf{16} no.~1 (2001), 1--28.
  \doi{10.1080/02681110010001270}.

\bibitem[SS04]{Sandstede_Scheel_Nonlinearity_05}
\bgroup\scshape{}B.~Sandstede\egroup{} and \bgroup\scshape{}A.~Scheel\egroup{},
  Absolute instabilities of standing pulses,  \emph{Nonlinearity} \textbf{18}
  no.~1 (2004), 331--378. \doi{10.1088/0951-7715/18/1/017}.

\bibitem[Sat76]{Sattinger_76}
\bgroup\scshape{}D.~H. Sattinger\egroup{}, On the stability of waves of
  nonlinear parabolic systems,  \emph{Advances in Mathematics} \textbf{22}
  no.~3 (1976), 312 -- 355. \doi{10.1016/0001-8708(76)90098-0}.

\bibitem[Sch94a]{Schneider-94-may}
\bgroup\scshape{}G.~Schneider\egroup{}, Error estimates for the
  {Ginzburg-Landau} approximation,  \emph{Zeitschrift f{\"u}r angewandte
  Mathematik und Physik ZAMP} \textbf{45} no.~3 (1994), 433--457.
  \doi{10.1007/BF00945930}.

\bibitem[Sch94b]{Schneider_94_juil}
\bgroup\scshape{}G.~Schneider\egroup{}, Global existence via {Ginzburg-Landau}
  formalism and pseudo-orbits of {Ginzburg-Landau} approximations,
  \emph{Communications in Mathematical Physics} \textbf{164} no.~1 (1994),
  157--179. \doi{10.1007/BF02108810}.

\bibitem[Sch94c]{Schneider_94_dec}
\bgroup\scshape{}G.~Schneider\egroup{}, A new estimate for the
  {Ginzburg-Landau} approximation on the real axis,  \emph{Journal of Nonlinear
  Science} \textbf{4} no.~1 (1994), 23--34. \doi{10.1007/BF02430625}.

\bibitem[SH77]{Swift_Hohenberg_77}
\bgroup\scshape{}J.~Swift\egroup{} and \bgroup\scshape{}P.~C.
  Hohenberg\egroup{}, Hydrodynamic fluctuations at the convective instability,
  \emph{Phys. Rev. A} \textbf{15} no.~1 (1977), 319--328.
  \doi{10.1103/PhysRevA.15.319}.

\bibitem[Xin92]{Xin_92}
\bgroup\scshape{}J.~X. Xin\egroup{}, Multidimensional stability of traveling
  waves in a bistable reaction-diffusion equation,  \emph{Communications in
  Partial Differential Equations} \textbf{17} no.~11-12 (1992), 1889--1899.
  \doi{10.1080/03605309208820907}.

\end{thebibliography}
\end{document}